\documentclass{article}

\usepackage{amsmath}
\usepackage{amssymb}
\usepackage{amsthm}
\usepackage{stmaryrd}
\usepackage{multirow}
\usepackage{subcaption}
\usepackage[margin=3cm]{geometry}
\usepackage{mathtools}
\mathtoolsset{showonlyrefs}
\usepackage{natbib}
\setcitestyle{authoryear,open={(},close={)}} 

\newtheorem{theorem}{Theorem}
\newtheorem{lemma}{Lemma}
\newtheorem{corollary}{Corollary}

\newcommand{\transpose}{\mathsf{T}}
\newcommand{\Laplace}{\Delta}
\newcommand{\Bhat}{\mathbf{B}_{\rm o}}

\newcommand{\Lbhat}{\widetilde{L}_{\rm b/o}}
\newcommand{\Shat}{\mathbf{S}_{\rm o}}

\begin{document}


\title{Impact of correlated observation errors on the convergence of the conjugate gradient algorithm
  in variational data assimilation}

\author{Olivier Goux$^{1,2}$, Selime G\"{u}rol$^1$, Anthony T. Weaver$^1$, Youssef Diouane$^3$, Oliver Guillet$^4$}

\date{{\small %
$^{1}$CERFACS / CECI CNRS UMR 5318, 42 avenue Gaspard Coriolis, 31057 Toulouse Cedex 01, France\\
$^{2}$ISAE-SUPAERO, University of Toulouse, 10 avenue Edouard Belin, BP 54032, 31055 Toulouse Cedex 4, France\\
$^{3}$Department of Mathematics and Industrial Engineering, Polytechnique Montr\'eal, QC, Canada\\
$^{4}$CNRM UMR 3589, M\'et\'eo-France and CNRS, 42 avenue Gaspard Coriolis, 31057 Toulouse Cedex 01, France%
}}

\maketitle


\abstract{An important class of nonlinear weighted least-squares problems arises from the assimilation of observations
  in atmospheric and ocean models. In variational data assimilation, inverse error covariance matrices define the
  weighting matrices of the least-squares problem. For observation errors, a diagonal matrix
  ({\it i.e.}, uncorrelated  errors) is often assumed for simplicity even when observation errors are suspected
  to be correlated. While accounting for observation-error correlations should improve the quality of the solution,
  it also affects the convergence rate of the minimization algorithms used to iterate to the solution.
  If the minimization process is stopped before reaching full convergence, which is usually the case in operational
  applications, the solution may be degraded even if the observation-error correlations are correctly
  accounted for.

  In this article, we explore the influence of the observation-error correlation matrix ($\mathbf{R}$)
  on the convergence rate of a preconditioned conjugate gradient (PCG) algorithm applied
  to a one-dimensional variational data assimilation (1D-Var) problem. We design the
  idealised 1D-Var system to include two key features used in more complex systems:
  we use the background error covariance matrix ($\mathbf{B}$) as a preconditioner (B-PCG);
  and we use a diffusion operator to model spatial correlations in $\mathbf{B}$ and $\mathbf{R}$.
  Analytical and numerical results with the 1D-Var system show a strong sensitivity of
  the convergence rate of B-PCG to the parameters of the diffusion-based correlation models.
  Depending on the parameter choices, correlated observation errors
  can either speed up or slow down the convergence.
  In practice, a compromise may be required in the parameter specifications of $\mathbf{B}$ and
  $\mathbf{R}$ between staying close to the best available estimates on the one hand and
  ensuring an adequate convergence rate of the minimization algorithm on the other.}
  
\vspace{0.5cm}

\section{Introduction}
\label{sec_intro}
An important class of nonlinear weighted least-squares problems arises from the assimilation of observations in atmospheric and
ocean models, a procedure known as {\it data assimilation}. In data assimilation,
observations of the state of a system are combined with an \textit{a priori} estimate of the state,
called the {\it background}, to produce an optimal estimate of the state of the system, called the {\it analysis}.
In {\it variational data assimilation}, the optimal estimate is obtained iteratively by minimising
a nonlinear weighted least-squares cost function that is the sum of two terms:
one measuring the model fit to the background state (the background term $\mathcal{J}_{\rm b}$); the other
measuring the model fit to the observations (the observation term $\mathcal{J}_{\rm o}$),
subject to constraints (generally nonlinear) that relate the model state to the observations.
The weighting matrices for $\mathcal{J}_{\rm b}$ and
$\mathcal{J}_{\rm o}$ are defined by an estimate of the inverse of the background and observation-error covariance matrices
($\mathbf{B}^{-1}$ and $\mathbf{R}^{-1}$), respectively.
Variational data assimilation is widely used for operational state estimation in meteorology and oceanography as it is a practical
method for solving nonlinear least-squares problems when the dimensions of the state and observation
vectors are huge (typically $10^6$ to $10^9$).

In variational data assimilation, the cost function is minimised approximately using a Truncated Gauss-Newton (GN)
algorithm \citep{Gratton_2007}
or incremental variational data assimilation as it is known in the data assimilation community \citep{Courtier_1994}.
This reduces the nonlinear problem to a sequence of linear sub-problems (quadratic cost functions), each of which is
solved iteratively using a Preconditioned Conjugate Gradient (PCG) method \citep{Gurol_2013}.
Standard implementations of PCG for data assimilation employ $\mathbf{B}$ as a first-level preconditioner
\citep{Derber_1989, Gurol_2013}, which
we refer to as B-PCG hereafter. Besides significantly improving the conditioning of the Hessian matrix \citep{Lorenc_1997},
$\mathbf{B}$-preconditioning allows the B-PCG algorithm to be formulated in a way that avoids the need to specify
$\mathbf{B}^{-1}$ explicitly. This is important as most $\mathbf{B}$ formulations used in practice are not associated with
convenient representations of~$\mathbf{B}^{-1}$.

There is still a requirement to specify $\mathbf{R}^{-1}$, however.
To simplify its specification, practical implementations of $\mathbf{R}$ tend to have relatively simple structural forms.
In the extreme yet common case, $\mathbf{R}$ is taken to be a diagonal matrix, which amounts to assuming that
the observation errors are uncorrelated.
This is a poor assumption for certain observations, especially from satellites
\citep{Bormann_2010, Waller_2016}.
If observation-error correlations are neglected when they are known to be important then
the solution of the weighted least-squares
problem will result in a degraded (sub-optimal) analysis and poor exploitation of the assimilated data. To mitigate
the former while still using a diagonal $\mathbf{R}$, observation data sets are either `thinned' into a subset
of observations or aggregated into `super-observations' that have reduced error correlations \citep{Liu_2002}.
Furthermore,
the observation-error variances are often multiplied by an `inflation' factor in order to prevent the analysis
from overfitting observations that may still have a substantial component of correlated error.
However, when observation error is correlated over distances similar to or greater than
those of the background error, inflation can actually degrade the analysis \citep{Reid_2020}.
While these methods can alleviate to some extent the inaccuracies associated
with a diagonal $\mathbf{R}$, they still lead to sub-optimal solutions since
potentially valuable observations are excluded and any remaining error correlations from the pre-processed
observations are ignored \citep{Rainwater_2015}.

Several studies have examined the impact from using non-diagonal representations of $\mathbf{R}$
to account for spatially correlated errors
\citep{Healy_2005,Stewart_2013,Ruggiero_2016,Pinnington_2016}.
A general conclusion that arises in most of these works is that accounting for spatial correlations in $\mathbf{R}$
leads to a more accurate solution, especially for the smaller spatial scales, even with a rather crude
correlation model. However, even crude correlation models can lead to impractical representations
of $\mathbf{R}^{-1}$. Various correlation models with accessible
inverse representations have been proposed in the literature
\citep{Brankart_2009,Michel_2018,Bedard_2019,Guillet_2019,Hu_2021}. One of the challenges with
specifying $\mathbf{R}$ and $\mathbf{R}^{-1}$ is that observation
locations tend to be arranged in an arbitrary and unpredictable way
due to the measurement method or quality control procedures that result in
observations being removed.
This means that correlation models developed for structured grids, like those typically associated
with $\mathbf{B}$, are not necessarily applicable for $\mathbf{R}$.

In this study, we use diffusion operators to model spatial correlations in both $\mathbf{B}$ and $\mathbf{R}$.
Diffusion operators can be used to model correlation functions from the Mat\'ern class
\citep{Guttorp_2006} and provide convenient and inexpensive representations of the associated
inverse correlation operators \citep{Mirouze_2010,Weaver_2013}.
They are popular for representing spatially correlated background error in complex boundary
domains such as those encountered in ocean data assimilation
\citep{Derber_1989,Egbert_1994,Weaver_2001,Weaver_2015,Weaver_2020}.
Furthermore, \cite{Guillet_2019} describes how to adapt these operators to unstructured meshes
and hence to make them suitable for $\mathbf{R}$ and $\mathbf{R}^{-1}$.

The rate of convergence of the conjugate gradient (CG) method
is mainly determined by the characteristics of the eigenvalue
spectrum of the Hessian matrix \citep{Axelsson_2000,Gurol_2013}.
As the eigenvalues of the Hessian matrix are strongly dependent on $\mathbf{B}$ and $\mathbf{R}$,
we can expect a non-diagonal $\mathbf{R}$ to have a significant impact on the rate of convergence
of B-PCG. In operational data assimilation, analyses must be delivered subject to strict computational
constraints, which means that the stopping criterion for B-PCG is usually set by a maximum allowed
number of iterations rather than a measure of the convergence of the solution.
Therefore, it is essential to ensure that the rate of convergence of B-PCG from the use
of a non-diagonal $\mathbf{R}$ is not deteriorated to an extent that it outweighs the benefits
brought from specifying a more accurate $\mathbf{R}$.

In previous work, \cite{Haben_2011}
analysed the convergence rate for the special case of a diagonal $\mathbf{R}$ and discussed
the influence on the condition number of observation and background accuracy, observation
density and the background-error correlation length-scale.
\cite{Tabeart_2018} studied the effects of a non-diagonal $\mathbf{R}$ on
the convergence rate of the unpreconditioned CG method and derived general theoretical bounds
for the condition number.
These results were extended by \cite{Tabeart_2021} to the $\mathbf{B}$-preconditioned
case (B-PCG). In both studies, theoretical and experimental results were
obtained for the special case where the background- and observation-error correlation matrices
are defined as circulant matrices. Their numerical experiments were performed using
a particular circulant matrix built from a Second Order Auto-Regressive (SOAR)
correlation function defined on the one-dimensional (1D) circular domain $\mathbb{S}$.

In this article, we are interested in understanding the sensitivity of the condition number
with respect to the basic parameters of the diffusion operators that are
used to model background- and observation-error correlations.
Correlation functions derived from diffusion operators are controlled by a smoothness parameter $M$
(the number of diffusion iterations) as well as a length-scale parameter $L$ (the square-root
of the diffusion coefficient),
and thus are more flexible than the SOAR function, which is controlled by a single length-scale parameter.
In fact, the SOAR function corresponds to a particular member (\mbox{$M=2$})
of the family of correlation kernels represented by the 1D diffusion operator.
We illustrate how the relative choice of $M$ for $\mathbf{B}$ and $\mathbf{R}$
can have a profound effect on the conditioning of the minimisation problem.

The organisation of the article is as follows. In Section~\ref{sec_background_theory}, we introduce
the weighted least-squares problem underlying variational data assimilation and we outline
the solution algorithm based on truncated GN combined with CG. We provide the
background theory on CG (and B-PCG) that is needed in this article for establishing the theoretical results
and for interpreting the results from the numerical experiments with a 1D variational data assimilation (1D-Var)
system. We conclude
this section with a description of $\mathbf{B}$ and $\mathbf{R}$, exposing
the fundamental covariance parameters that control the shape characteristics of the correlation functions
as well as the conditioning of the CG minimisation. The formulation of $\mathbf{B}$ and $\mathbf{R}$
in terms of diffusion operators depends on theoretical results that are summarised
in Appendix~\ref{app_matern}. In Section~\ref{sec_general_conditioning}, we study the eigenvalue spectrum
of the Hessian matrix and derive analytical expressions for bounds on the condition number.
First, we present the general
bounds that were derived by \cite{Tabeart_2018} and \cite{Tabeart_2021}.
Then, we derive specific bounds that take into account the structural properties of the
diffusion operators used to model $\mathbf{B}$ and $\mathbf{R}$. We relegate the technical
details of the proofs of a key theorem and associated corollaries to
Appendices~\ref{app_proof_theorem5}, \ref{app_proof_corollary1}, and \ref{app_proof_corollary2}.
In Section~\ref{sec_numerical_experiments},
we present the results from numerical experiments with the 1D-Var system to examine the
sensitivity of the convergence
of B-PCG to the parameters of $\mathbf{B}$ and $\mathbf{R}$. These results show a strong sensitivity
of the condition number, and hence convergence rate, to the correlation parameters.
We argue that the parameter values should be chosen as a compromise between specifying the
most accurate correlation model on the one hand and achieving a satisfactory convergence rate
for the CG minimisation on the other.
We provide a summary and conclusions in Section~\ref{sec_conclusions}.

\section{The weighted least-squares problem}
\label{sec_background_theory}
\subsection{Problem formulation: variational data assimilation}
\label{sec_problem_formulation}
Variational data assimilation provides an estimate of the physical state of a system
by combining \textit{a priori} information (the background state) and observations,
together with information about their uncertainties. Here, we will use mathematical
notation that is standard in meteorological and ocean data assimilation \citep{Ide_97}.
Assuming unbiased Gaussian error statistics for the background state and observations,
the estimation problem can be formulated as a nonlinear weighted least-squares problem
defined by the cost function
\begin{equation}
  \min_{\mathbf{x}} \mathcal{J}(\mathbf{x})= \tfrac{1}{2}\|\mathbf{x}-\mathbf{x}_{\rm b}\|_{\mathbf{B}^{-1}}^{2}
  +\tfrac{1}{2}\|\mathcal{H}(\mathbf{x})-\mathbf{y}_{\rm o}\|_{\mathbf{R}^{-1}}^{2}
    \label{eq_cost_function}
\end{equation}
where $\mathbf{x} \in \mathbb{R}^{n}$ is the state vector to be optimised and
$\mathbf{x}_{\rm b} \in \mathbb{R}^{n}$ is the background estimate of the state vector.
The vector of observations is $\mathbf{y}_{\rm o} \in \mathbb{R}^{m}$, and $\mathcal{H}(\cdot)$
is the observation operator, which maps an estimate of the state of the system
to its equivalent in observation space. In general, $\mathcal{H}(\cdot)$ is nonlinear
and non-bijective. In four-dimensional variational assimilation (4D-Var),
$\mathcal{H}(\cdot)$ would contain the forecast model operator,
$\mathbf{x}$ would be the initial state vector and
$\mathbf{y}_{\rm o}$ would be a vector that concatenates observations distributed over a given time window.
The unbiased Gaussian distributions of the background and observation errors are characterized
statistically by the covariance matrices \mbox{$\mathbf{B} \in \mathbb{R}^{n \times n}$} and
\mbox{$\mathbf{R} \in \mathbb{R}^{m \times m}$}, respectively. By definition, $\mathbf{B}$ and $\mathbf{R}$ are symmetric
positive-definite (SPD) matrices. The inverse covariance matrices $\mathbf{B}^{-1}$ and $\mathbf{R}^{-1}$
define inner products in the background and observation spaces, and are used as weighting matrices
in the cost function \eqref{eq_cost_function} where $\|\mathbf{x}\|_{\mathbf{P}}^2=\mathbf{x}^\transpose\mathbf{P}\mathbf{x}$
denotes the squared $\mathbf{P}$-norm of a vector.
The analysis is the global minimising solution: \mbox{$\mathbf{x}_{\rm a} = \arg \min \mathcal{J}(\mathbf{x})$}.

Truncated Gauss-Newton (GN) \citep{Gratton_2007}, which is known as incremental
variational assimilation in the meteorological and ocean data assimilation
communities~\citep{Courtier_1994}, is a common method for finding an approximate
minimum of the nonlinear cost function~\eqref{eq_cost_function}. Truncated GN approaches
the solution iteratively by solving, on each GN iteration $k$, the linearized sub-problem
\begin{equation}
\label{eq_Quadratic}
\min_{\delta \mathbf{x}} J(\delta \mathbf{x}) = \tfrac{1}{2}\|\mathbf{x}_k - \mathbf{x}_{\rm b}
+ \delta \mathbf{x} \|^2_{\mathbf{B}^{-1}} + \tfrac{1}{2}\|\mathbf{H}_k \delta \mathbf{x} - \mathbf{d}_k \|^2_{\mathbf{R}^{-1}},
\end{equation}
which is a quadratic approximation of the non-quadratic cost function~\eqref{eq_cost_function}
in a neighbourhood of the current iterate $\mathbf{x}_k$.
In \eqref{eq_Quadratic}, $\mathbf{H}_k \in \mathbb{R}^{m \times n}$ is the observation operator linearized
about  $\mathbf{x}_k$, and \mbox{$\mathbf{d}_k = \mathbf{y}_{\rm o} - \mathcal{H}(\mathbf{x}_k) \in \mathbb{R}^{m}$}
is the misfit between the observation vector and the current iterate mapped to observation
space. If $\delta \mathbf{x}_k$ denotes the solution of \eqref{eq_Quadratic} then the estimate
of the state is updated according to
$$
    \mathbf{x}_{k} = \mathbf{x}_{k-1} +  \delta \mathbf{x}_k,
$$
    where $k=1,\ldots,K$ and $\mathbf{x}_0 = \mathbf{x}_{\rm b}$ (in general). In data assimilation applications
    with atmospheric or ocean models, the maximum number of GN iterations is typically very
    small ($K<10$) for computational reasons.

The quadratic sub-problem~\eqref{eq_Quadratic} can be rewritten in standard quadratic form
\begin{equation}
\label{eq_compactquadratic}
\min_{\delta \mathbf{x}} J(\delta \mathbf{x}) =
\frac{1}{2} \delta \mathbf{x}^{\transpose} \mathbf{A}_k \delta \mathbf{x} - \mathbf{b}^{\transpose}_k \delta \mathbf{x} + c_k,
\end{equation}
where
$$
\mathbf{A}_k = \mathbf{B}^{-1}+\mathbf{H}_k^\mathsf{T}\mathbf{R}^{-1}\mathbf{H}_k
$$
is the 
SPD approximation of the Hessian matrix of the nonlinear cost function,
$$
    \mathbf{b}_k = \mathbf{B}^{-1}(\mathbf{x}_{\rm b} - \mathbf{x}_k) +  \mathbf{H}_k^{\transpose}\mathbf{R}^{-1}\mathbf{d}_k
$$
    is the negative gradient of the nonlinear cost function with respect to the current
    iterate $\mathbf{x}_k$, and $c_k = J(\mathbf{0})$ is a scalar. Satisfying the optimality
    condition of the quadratic sub-problem~\eqref{eq_compactquadratic} requires solving the linear system
\begin{equation}
\mathbf{A}_k \delta\mathbf{x} = \mathbf{b}_k.
\label{eq_BLUE_system}
\nonumber
\end{equation}

For our target applications, the dimension ($n$) of the state vector is large and the matrices are generally
only available as operators ({\it i.e.}, via matrix-vector products, not explicit matrices). For this reason,
it is very common to solve the quadratic sub-problem iteratively using CG methods.

\subsection{Solving the quadratic sub-problem with the conjugate gradient method}
\label{sec_CG}
CG is a Krylov subspace method (see \citet[Section 6.7]{Golub_2013} and \citet[Section 10.2]{Saad_2003})
for solving linear systems where the system matrix is SPD. CG seeks an approximate solution
\begin{equation}
    \delta \mathbf{x}_{\ell} \in \delta \mathbf{x}_0 + \mathcal{K}^{\ell}(\mathbf{A}_k, \mathbf{b}_k),
\nonumber
\end{equation}
where $\delta \mathbf{x}_0$ is the initial approximation and
\begin{equation}
    \mathcal{K}^{\ell}(\mathbf{A}_k,\mathbf{b}_k)
    =\text{span}\{\mathbf{b}_k,\mathbf{A}_k\mathbf{b}_k,\cdots,\mathbf{A}_k^{\ell-1}\mathbf{b}_k\}
\nonumber
\end{equation}
is the Krylov subspace generated by $\mathbf{A}_k$ and $\mathbf{b}_k$. When $\mathbf{x}_0 = \mathbf{x}_{\rm b}$,
the initial iterate $\delta \mathbf{x}_0 = \mathbf{0}$. Hereafter, we will drop the truncated
GN iteration index $k$ for clarity of notation. In order to find a unique solution,
CG imposes the orthogonality condition
\begin{equation}
   \mathbf{r}_{\ell}  \perp \mathcal{K}^{\ell}(\mathbf{A},\mathbf{b}),
\nonumber
\end{equation}
where $\mathbf{r}_{\ell} = \mathbf{b} - \mathbf{A} \delta \mathbf{x}_{\ell}$ is the residual at the $\ell$-th iteration of CG.
As a result, CG minimises the quadratic cost function given by \eqref{eq_compactquadratic} over the subspace
$ \delta \mathbf{x}_0 + \mathcal{K}^{\ell}(\mathbf{A}, \mathbf{b})$~\citep[Theorem 5.2]{JorgeNocedal2006}, so that
the $\ell$-th iterate $\delta \mathbf{x}_{\ell}$ minimises the error
$\mathbf{e}_{\ell} = \delta \mathbf{x}^\ast - \delta \mathbf{x}_{\ell}$
in the $\mathbf{A}$-norm over the same Krylov subspace, $\delta \mathbf{x}^\ast$ being the exact solution~\citep[Lemma 2.1.1]{Kelley_1987}.
The convergence properties of CG can then be analysed in terms of the error in the $\mathbf{A}$-norm~\citep[pages 204-205]{Saad_2003}:
\begin{equation}
  \frac{\|\mathbf{e_{\ell}}\|_{\mathbf{A}}}{\|\mathbf{e_0}\|_{\mathbf{A}}}
  \leq 2\left(\frac{\sqrt{\kappa(\mathbf{A})}-1}{\sqrt{\kappa(\mathbf{A})}+1}\right)^{\ell}, \quad \ell \in \mathbb{N},
\label{eq_kappa_error_bound}
\end{equation}
where $\kappa(\mathbf{A})$ is the condition number of $\mathbf{A}$, which is defined in the 2-norm as  
\begin{equation}
\label{eq_def_kappa}
\kappa(\mathbf{A})=\frac{\lambda_{\rm max}(\mathbf{A})}{\lambda_{\rm min}(\mathbf{A})},
\nonumber
\end{equation}
with $\lambda_{\rm max}(\mathbf{A})$ and $\lambda_{\rm min}(\mathbf{A})$ being the largest
and smallest eigenvalues of $\mathbf{A}$, respectively. Equation~\eqref{eq_kappa_error_bound}
shows that convergence will tend to be fast when the condition number is close to 1.
The condition number can thus be used as an indicator of the convergence rate of CG. Note that
the initial error may also have an influence on the convergence behaviour.
For simplicity, however, we focus only on the effect of the condition number on the convergence rate.

Since Equation~\eqref{eq_kappa_error_bound} depends on the condition number, it does not account for the distribution
of the eigenvalues between the smallest and largest values. As a consequence, it can lead to a pessimistic error bound,
especially when $\kappa$ is very large. More advanced error bounds exist, which account for a more complex representation
of the spectrum (\textit{e.g.}, see Chapter~13 in the book of \cite{Axelsson_1994}). However, for the problem considered
in this article, those error bounds provide little improvement over the error bound given by Equation~\eqref{eq_kappa_error_bound}.
They have been shown to be more accurate only in the limit of a very large number of iterations, while here we are interested mainly
in the solution accuracy in the early iterations of CG. As they are more complex and less general than Equation~\eqref{eq_kappa_error_bound},
we did not apply them in this article.

In order to accelerate the convergence rate of CG, it is common to use a preconditioner. For data assimilation problems
that solve quadratic problem~\eqref{eq_compactquadratic}, it is customary to use $\mathbf{B}$ as a preconditioner, as it usually yields
a significantly smaller condition number compared to that of the unpreconditioned problem, and a more clustered spectrum of
eigenvalues \citep{Lorenc_1988,Lorenc_1997,Gurol_2013}. Therefore, we will focus on solving the $\mathbf{B}$-preconditioned
linear system. Since $\mathbf{B}$ is SPD, it can be factored as
\begin{equation}
\mathbf{B} = \mathbf{U} \mathbf{U}^\transpose
\end{equation}
where $\mathbf{U} \in \mathbb{R}^{n \times n}$.
We can then introduce $\mathbf{B}$-preconditioning symmetrically using a split-preconditioner,
\begin{equation}
\label{eq_preconditioned_system}
    \mathbf{U}^\transpose \mathbf{A}\, \mathbf{U}\, \delta \mathbf{v} = \mathbf{U}^\transpose \mathbf{b}
\end{equation}
where $ \delta \mathbf{x} = \mathbf{U} \delta \mathbf{v}$.
An unpreconditioned CG can be applied to Equation~\eqref{eq_preconditioned_system} by taking $\mathbf{U}^\transpose \mathbf{A} \mathbf{U}$
as the (SPD) system matrix and $\mathbf{U}^\transpose \mathbf{b}$ as the right-hand side.


 In this article, we will evaluate the condition number of the preconditioned Hessian matrix,
\begin{equation}
    \label{eq_def_S}
    \mathbf{S}\, = \, \mathbf{U}^\transpose \mathbf{A}\, \mathbf{U} \\ =\, \mathbf{I}_n\, +\, \mathbf{U}^\transpose\mathbf{H}^\transpose \mathbf{R}^{-1} \, \mathbf{H}\, \mathbf{U},
\end{equation}
and determine its sensitivity to parameters in the covariance matrices $\mathbf{B}$ and $\mathbf{R}$ when their spatial
correlations are modelled by diffusion operators.

\subsection{Weighting matrices formulated as the inverse of diffusion operators}

\label{sec_diffusion_operators}
The operators $\mathbf{B}$ and $\mathbf{R}$ describe the covariance structures of the background and
observation errors. 
In the idealised 1D-Var system used in this study, the covariance matrices are small enough to
be constructed
explicitly using a functional expression to determine the matrix elements ({\it e.g.}, as done in
\cite{Tabeart_2021}). 
However, when one considers a realistic system, the size of $\mathbf{B}$ and $\mathbf{R}$ become too large to perform direct
matrix-vector products. Hence, we prefer to adopt an approach that scales with the size of the problem
and avoids the explicit construction of covariance matrices.

Covariance matrices can be factored as \mbox{$\boldsymbol{\Sigma} \mathbf{C} \boldsymbol{\Sigma}$}
where $\boldsymbol{\Sigma}$ is a diagonal matrix of standard deviations and $\mathbf{C}$ is an
SPD correlation matrix. The computational difficulties are inherent in the specification and
application of $\mathbf{C}$.
\cite{Egbert_1994} showed that multiplying an arbitrary vector by a Gaussian correlation matrix
can approximately be achieved by numerically `time'-stepping a diffusion equation with that arbitrary vector
taken as the `initial' condition\footnote{In the current context, the time coordinate in the diffusion equation
  does not represent physical time but should be interpreted as a pseudo-time coordinate that
  controls the smoothing properties of the diffusion kernel.
  This explains why `time' has been written within quotation marks.}.
This procedure defines an operator that models the product of a correlation matrix with the `initial'
condition without defining each element of the matrix. Since each `time'-step involves
the manipulation of sparse matrices, this strategy is naturally appropriate for large problems.
The product of the diffusion coefficient $\mu$ and the total action `time' \mbox{$T = M\Delta t$}
of the diffusion process, where $M$ is the total number of diffusion steps and $\Delta t$ is
the `time' step, controls the length-scale $D_{\rm g}$ of the Gaussian function that would be
used to construct the correlation matrix, where \mbox{$D_{\rm g}^2 = 2 \mu T$}.
\cite{Weaver_2001} describe the technique in detail and generalize it
to account for anisotropic correlations. \cite{Mirouze_2010} and \cite{Weaver_2013}
describe an extension of the technique that involves solving the diffusion equation using an implicit
`time'-stepping scheme instead of the explicit scheme of the original approach.

With the implicit scheme, the total number of diffusion steps $M$ becomes a free parameter together
with the diffusion coefficient multiplied by the `time' step ($\mu \Delta t$). (With the explicit
scheme, their product is the single free parameter controlling the length-scale $D_{\rm g}$ of the
Gaussian function).
This extra degree of freedom allows the diffusion operator
to represent matrix-vector products with correlation matrices from the Mat\'ern family
\citep{Guttorp_2006} where $M$ is linked to the standard
smoothness parameter $\nu$ of the underlying Mat\'ern correlation functions in $\mathbb{R}^d$
via the relation \mbox{$\nu = M - d/2$}.
In $\mathbb{R}$, these functions are characterised by a polynomial times
the exponential function and are also known as $M$th-order Auto-Regressive (AR) functions.
The parameter $\mu \Delta t$ is precisely the square of the
standard length-scale parameter $L$ of the Mat\'ern functions. Without loss of generality,
$\Delta t$ can be set to 1, so that \mbox{$\mu = L^2$}. The connection between
the Mat\'ern correlation functions and the differential operator describing the inverse
of an implicit diffusion process
has its roots in the seminal work of \cite{Whittle_1963}. We outline the connection
in Appendix~\ref{app_matern}
and provide the key theoretical relations for the 1D-Var problem under consideration
in this study.

While $L$ is the length-scale parameter appearing explicitly in the definition of the
diffusion operator, it is common in data assimilation to use an alternative length-scale
parameter, the Daley length-scale $D$, to control the spatial smoothing properties of the
diffusion kernel. The Daley length-scale can be understood as the half-width of the parabola
osculating the correlation function at its origin \citep{Daley_1991, Pannekoucke_2008}. It is defined for
at least twice differentiable correlations functions, which in our case corresponds to
AR functions with \mbox{$M>1$}. As discussed in Appendix~\ref{app_matern_R}, on $\mathbb{R}$,
$D$ and $L$ are related through the equation
\begin{equation}
  D = L \sqrt{2M-3}
  \label{eq_D}
\end{equation}
where the square-root term generalises to $\sqrt{2M - d - 2}$ in $\mathbb{R}^{d}$ \citep{Weaver_2013}.
An advantage of $D$ over $L$ is that it allows better control of the spectral properties
of the AR functions (see Figure~\ref{fig_matern_plot} in Appendix~\ref{app_matern_R}).
In particular, AR functions converge to a Gaussian function with
length-scale $D$ as $M$ tends to infinity with $L$ simultaneously reduced to zero to keep $D$ constant.
A closely-related length-scale parameter
\begin{equation}
  \rho = L \sqrt{2M-1}
  \label{eq_rho}
\end{equation}
is used in geostatistics \citep[pp.~48--50]{Stein_1999} and machine learning \citep[Chapter~4.2]{Rasmussen_2006}.
AR functions defined in terms of $\rho$ also have
the property of converging to a Gaussian function (with length-scale $\rho$) as $M$ tends to infinity
with $\rho$ fixed.
In $\mathbb{R}^{d}$, the square-root term in Equation~\eqref{eq_rho}  generalises to $\sqrt{2M - d}$.
An advantage of $\rho$ over $D$ is that it is valid for \mbox{$M=1$} as well as \mbox{$M>1$},
while an advantage of $D$ over $\rho$ (and $L$) is that it is easier to estimate in practical
applications when $D$ is spatially dependent \citep[Section~2.4]{Weaver_2020}.
In Section~\ref{sec_general_conditioning}, the analytical results are first derived in terms of $L$ and then
interpreted in terms of both $D$ and $\rho$, where we will refer to the latter as the {\it Stein}
length-scale\footnote{In \cite{Stein_1999}, the square-root term in Equation~\eqref{eq_rho} is effectively
  $\sqrt{2(2M-1)}$ where the extra factor of 2 comes from his alternative definition of the Gaussian function
  that does not include the factor of 2 in the denominator of the function argument as we have assumed here
  (see Equation~\eqref{eq:cg}).} for convenience.
In Section~\ref{sec_numerical_experiments}, the numerical experiments 
are discussed mainly in terms of $D$.

On a circular domain of radius $a$ (see Appendix~\ref{app_matern_S}), we can define the discrete, symmetric diffusion-modelled
correlation operator $\mathbf{C}$ as a sequence of linear operators represented by their respective matrices
\citep[Section~3.1]{Weaver_2015}:
\begin{equation}
  \mathbf{C} \; = \; \mathbf{\Gamma} \, \mathbf{L} \, \mathbf{W}^{-1} \, \mathbf{\Gamma}
    \label{eq_original_diff_operator}
\end{equation}
where $\mathbf{L} = \mathbf{T}^{-M}$ is a self-adjoint diffusion operator, $\mathbf{T}$ being a discrete representation
of the shifted Laplacian operator $\mathcal{T}$. On the circular domain with constant $L$, we have from
Equation~\eqref{eq:impS} that
\mbox{$\mathcal{T} \equiv I - L^2 \partial^2 / a^2 \partial \phi^2$} where \mbox{$-\pi \leq \phi \leq \pi$}.
The matrix $\mathbf{W}$ contains
geometry- and grid-dependent weights.
It defines the weighting matrix of the discrete form
of the $L^2(\mathbb{S})$-inner product with respect to which $\mathbf{L}$ is self-adjoint; {\it i.e.},
\mbox{$\mathbf{L} = \mathbf{W}^{-1} \mathbf{L}^\transpose \mathbf{W}$}.
Sections~3.2 and 3.3 of \cite{Guillet_2019} provide a comprehensive discussion of this point
within the context of a Finite Element Method discretisation of the diffusion equation.
The diagonal matrix $\mathbf{\Gamma}$ contains normalisation factors so that the diagonal elements of $\mathbf{C}$
are approximately equal to one. On the circular domain with constant \mbox{$L \ll a$}, we can set
\mbox{$\mathbf{\Gamma}=\gamma\mathbf{I}$} where $\gamma$ is well approximated by the constant product
$\sqrt{\nu L}$ where $\nu$ is a monotonically increasing function of $M$ given by Equation~\eqref{eq:nu}.
For example, the error in $\gamma^2$ is smaller than 0.001\% when \mbox{$L/a = 0.3$}.
\cite{Weaver_2020} provide an overview of methods for estimating $\mathbf{\Gamma}$
on other domains and when the correlation parameters are not constant.

Taking $M$ to be an even number allows us to split \mbox{$\mathbf{T}^{-M} = \mathbf{T}^{-M/2} \, \mathbf{T}^{-M/2}$}
and hence to derive a simple `square-root' factorisation of Equation~\eqref{eq_original_diff_operator}.
The `square-root' operator is convenient for generating random correlated samples and has been
used for this purpose for the numerical experiments in Section~\ref{sec_numerical_experiments}.
Another convenient property that comes specifically from the implicit formulation is that it
provides immediate access to an inexpensive formulation of the inverse
of the operator. This can be noticed from the inverse of Equation~\eqref{eq_original_diff_operator},
\begin{equation}
  \mathbf{C}^{-1} \; = \; \mathbf{\Gamma}^{-1}\, \mathbf{W} \, \mathbf{L}^{-1} \, \mathbf{\Gamma}^{-1},
    \label{eq_original_diff_operator_inv}
\end{equation}
where $\mathbf{L}^{-1} = \mathbf{T}^M$ simply involves $M$ applications of the shifted Laplacian operator.

Here, we consider a finite-difference discretisation of the diffusion equation under the assumption
that the grid resolution is uniform so that $\mathbf{W}=h\mathbf{I}$ where $h$ is the grid size.
Furthermore, we assume that \mbox{$\mathbf{\Sigma}=\sigma\mathbf{I}$} where $\sigma$ is a constant standard deviation.
Given these assumptions together with the assumption that $L$ is constant,
we can simplify the expressions for the diffusion-modelled covariance operators for $\mathbf{B}$ and $\mathbf{R}$ as
\begin{align}
  \mathbf{B}  & =
  \frac{\displaystyle \sigma_{\rm b}^2 \nu_{\rm b} L_{\rm b}}{\displaystyle h_{\rm b}}\mathbf{T}_{\rm b}^{-M_{\rm b}}
  \; = \;
  \frac{ \displaystyle \sigma_{\rm b}^2 \nu_{\rm b} L_{\rm b}}{\displaystyle h_{\rm b}}
  \left( \mathbf{I}_n-L_{\rm b}^2 \mathbf{\Laplace}_{h_{\rm b}} \right)^{-M_{\rm b}},
  \label{eq_Bmatrix_diffusion}
  \\
  \mathbf{R} & =
  \frac{ \displaystyle \sigma_{\rm o}^2 \nu_{\rm o} L_{\rm o}}{\displaystyle h_{\rm o}} \mathbf{T}_{\rm o}^{-M_{\rm o}}
  \; = \;
  \frac{\displaystyle \sigma_{\rm o}^2 \nu_{\rm o} L_{\rm o}}{\displaystyle h_{\rm o}}
  \left( \mathbf{I}_m-L_{\rm o}^2 \mathbf{\Laplace}_{h_{\rm o}} \right)^{-M_{\rm o}},
    \label{eq_Rmatrix_diffusion}
\end{align}
where the subscripts `b' and `o' refer to quantities relative to the background and
observations, respectively. The symbol $\mathbf{\Laplace}_{h}$ denotes the finite-difference representation of the
Laplacian operator, which depends on the grid resolution for the background and observations. The
matrices \mbox{$\mathbf{I}_n \in \mathbb{R}^{n \times n}$} and \mbox{$\mathbf{I}_m  \in \mathbb{R}^{m \times m}$} are identity
matrices. We are also interested in the expression for $\mathbf{R}^{-1}$, which follows immediately from
Equation~\eqref{eq_Rmatrix_diffusion}:
\begin{align}
  \mathbf{R}^{-1}  &  =
  \frac{h_{\rm o}}{ \sigma_{\rm o}^2 \nu_{\rm o} L_{\rm o}}\mathbf{T}_{\rm o}^{M_{\rm o}}
  \; = \;
  \frac{h_{\rm o}}{ \sigma_{\rm o}^2 \nu_{\rm o} L_{\rm o}} \left( \mathbf{I}_m-L_{\rm o}^2 \mathbf{\Laplace}_{h_{\rm o}}\right)^{M_{\rm o}}.
    \label{eq_Rinv}
\end{align}
By taking \mbox{$\mathbf{W}_{\rm o} = h_{\rm o} \mathbf{I}_m$}, we are assuming that the observations are regularly distributed with
a separation distance of $h_{\rm o}$. This is done for mathematical convenience. Relative to the domain size $2\pi a$,
$h_{\rm o}$ is an explicit parameter that reflects observation density and can be compared
to $h_{\rm b}/2\pi a$, the density of background points.
With the simplifying assumptions above, we are able to establish explicit theoretical bounds on the condition number
of the preconditioned Hessian matrix as detailed in Section~\ref{sec_general_conditioning}.


We remark now on the actual values of the parameter pairs $(M_{\rm b},M_{\rm o})$,
$(L_{\rm b},L_{\rm o})$ and $(\sigma^2_{\rm b},\sigma^2_{\rm o})$ that will be considered in this study.
First, values of \mbox{$(M_{\rm b}, M_{\rm o}) \ge 10$} lead to AR functions that are practically Gaussian, so we will not
consider values beyond $10$. Values of $(L_{\rm b},L_{\rm o})$ should be large enough compared to the grid size
$(h_{\rm b},h_{\rm o})$ (at least \mbox{$L_{\rm b}/h_{\rm b}\geq 1$} and \mbox{$L_{\rm o}/h_{\rm o}\geq 1$})
in order to avoid large discretisation errors in the
finite-difference representation of the diffusion operator. Ideally, the parameters
should be chosen to provide the optimal fit to our available estimate of the error covariances
(with \mbox{$(L_{\rm b},L_{\rm o})$} and \mbox{$(\sigma^2_{\rm b},\sigma^2_{\rm o})$}
made spatially dependent in general). Background-error correlations are often specified as
quasi-Gaussian functions (large values of \mbox{$(M_{\rm b}, M_{\rm o})$}).
The reason for this choice can be mainly computational;
{\it i.e.}, efficient models, like diffusion, exist for applying quasi-Gaussian functions
\citep{Gaspari_1999,Weaver_2001,Purser_2003}.
Another reason is that quasi-Gaussian
functions are sufficiently regular that they can be differentiated, which is important for defining
cross-variable (multivariate) covariances in atmospheric and ocean data assimilation
\citep{Daley_1991,Derber_1999,Weaver_2005}. In comparison, estimates of the spatial
correlations of observation error often display a sharp decrease at short range
and slow decay at longer range \citep{Waller_2016,Waller_2016b,Michel_2018},
which with an AR function is best modelled with a small value of $(M_{\rm b}, M_{\rm o})$.
In view of these remarks, the case where \mbox{$M_{\rm o} < M_{\rm b}$} seems to be of particular
interest. Nevertheless, both this case and the case where \mbox{$M_{\rm o} \ge M_{\rm b}$} will
be considered as different data-sets may give rise to different error characteristics.

\section{Conditioning of the preconditioned linear system}
\label{sec_general_conditioning}
In this section, we are interested in analysing the convergence of CG applied to
the linear system \eqref{eq_preconditioned_system} where the system matrix $\mathbf{S}$
depends on the diffusion-modelled covariance matrices described in the previous section.
In particular, we are interested in analysing the sensitivity of the convergence in terms
of the parameters of these covariance matrices. For this purpose, we will focus on the
condition number of $\mathbf{S}$, denoted $\kappa(\mathbf{S})$.

We start by recalling some results from \cite{Haben_2011} and \cite{Tabeart_2021} on the upper
bound of $\kappa(\mathbf{S})$ for general covariance matrices.
\begin{theorem}[Theorem 3 of~\cite{Tabeart_2021}]
\label{th_infnorm_bound}
Let \mbox{$\mathbf{B} \in \mathbb{R}^{n\times n}$} and \mbox{$\mathbf{R} \in \mathbb{R}^{m \times m}$}
be symmetric, positive-definite matrices.
Let \mbox{$\mathbf{U} = \mathbf{U}^{\transpose} \in \mathbb{R}^{n \times n}$} be the (unique) symmetric square root of
\mbox{$\mathbf{B}=\mathbf{U}\mathbf{U}^{\transpose} = \mathbf{U}^2$} and let
\mbox{$\mathbf{V} = \mathbf{V}^{\transpose}\in \mathbb{R}^{m \times m}$}
be the (unique) symmetric square root of
\mbox{$\mathbf{R}=\mathbf{V}\mathbf{V}^{\transpose} = \mathbf{V}^2$}. If \mbox{$\mathbf{H} \in \mathbb{R}^{m \times n}$}, with
\mbox{$m < n$}, and
\mbox{$\mathbf{S} = \mathbf{I}_n + \mathbf{U}\mathbf{H}^\transpose \mathbf{R}^{-1} \mathbf{H}\mathbf{U}$}, then
\begin{equation}
\label{general_upper_bound}
\kappa(\mathbf{S}) \leq 1 + \|\mathbf{V}^{-1}\mathbf{H}\mathbf{B}\mathbf{H}^{\transpose}\mathbf{V}^{-1} \|_{\infty}.
\end{equation}
\end{theorem}
The upper bound given in Equation~\eqref{general_upper_bound} is quite general
and it is not straightforward to understand
the impact of each component of the matrix $\mathbf{S}$ on the condition number.
Moreover, caution is needed when applying Equation~\eqref{general_upper_bound} in physical applications
as the eigen-decomposition of a covariance matrix will not
be independent of the physical units of the variables \cite[Section~4.3.4]{Tarantola_1987}. For this
reason, it is generally more meaningful to consider an eigen-decomposition on the associated
(non-dimensional) correlation matrix. In our idealised study, there is a single ``physical'' variable,
with arbitrary units, and direct observations
of that variable. Furthermore, both the background- and observation-error variances are taken to be constant,
so can be factored out of $\mathbf{B}$ and $\mathbf{R}$. In this case, the eigenvalues of $\mathbf{B}$ and $\mathbf{R}$
will be identical to those of their respective correlation matrices up to a multiplicative factor given
by their respective variances. We can thus continue to consider the eigen-decompositions of $\mathbf{B}$ and
$\mathbf{R}$ without ambiguity.

Other bounds have been proposed by \cite{Haben_2011} and \cite{Tabeart_2021} to understand
the impact of each covariance matrix.
\begin{theorem} [Corollary~1 of~\cite{Tabeart_2021}]
\label{th_naive_bound}
Let \mbox{$\mathbf{B} \in \mathbb{R}^{n\times n}$} and \mbox{$\mathbf{R} \in \mathbb{R}^{m \times m}$} be symmetric,
positive-definite matrices.
Let $\mathbf{U} = \mathbf{U}^\transpose \in \mathbb{R}^{n \times n}$ be the (unique) symmetric square root of
$\mathbf{B}=\mathbf{U}\mathbf{U}^{\transpose} = \mathbf{U}^2$.
If $\mathbf{H} \in \mathbb{R}^{m \times n}$, with $m < n$, and
$\mathbf{S} = \mathbf{I}_n + \mathbf{U}\mathbf{H}^\transpose \mathbf{R}^{-1} \mathbf{H}\mathbf{U}$, then
\begin{equation}
\label{eq_naive_bound}
\kappa(\mathbf{S})\leq 1+\frac{\lambda_{\rm max}(\mathbf{B})}{\lambda_{\rm min}(\mathbf{R})}\lambda_{\rm max}(\mathbf{H}\mathbf{H}^\transpose).
\end{equation}
\end{theorem}
The upper bound in Equation~\eqref{eq_naive_bound} attempts to separate the influence of $\mathbf{B}$ and $\mathbf{R}$ by considering
their eigenvalues separately. While this separation makes the bound easier to use for sensitivity analyses, it can also
degrade the accuracy of the bound. Indeed, compared to the bound given by Equation~\eqref{general_upper_bound}, the bound given
by Equation~\eqref{eq_naive_bound} does not account for any interaction between $\mathbf{B}$ and $\mathbf{R}^{-1}$, which results in a
more pessimistic bound.

As an example, we consider the case where both $\mathbf{B}$ and $\mathbf{R}$ are modelled using
diffusion operators as described in Section~\ref{sec_diffusion_operators}, and $\mathbf{H}$ is a
selection matrix. We consider a domain of length $2000$~km, composed of \mbox{$n=500$} points
that are equally spaced every \mbox{$h_{\rm b} = 4$}~km. We assume that direct observations are
available at every other grid point (\mbox{$m=250$} and \mbox{$h_{\rm o} = 8$}~km). We define
\mbox{$\widetilde{L}_{\rm b} = L_{\rm b}/h_{\rm b}$} and \mbox{$\widetilde{L}_{\rm o} = L_{\rm o}/h_{\rm o}$}
where \mbox{$L_{\rm b} = 60$}~km is fixed and $L_{\rm o}$ is allowed to vary.
Figure~\ref{fig_naive_bound} compares the two upper bounds from
Theorem~\ref{th_infnorm_bound} and Theorem~\ref{th_naive_bound} with the
exact condition number\footnote{The exact condition number is computed using results
from Theorem~\ref{th_eig_S} described later in Section~\ref{sec_general_conditioning}.}
for different parameter specifications in $\mathbf{B}$ and $\mathbf{R}$.
Figure~\ref{fig_naive_bound_left} shows the results for \mbox{$M_{\rm b}=8$} and \mbox{$M_{\rm o}=2$}, while
Figure~\ref{fig_naive_bound_right} shows the results for \mbox{$M_{\rm b}=2$} and \mbox{$M_{\rm o}=8$},
for values of \mbox{${\widetilde{L}_{\rm o}}/{\widetilde{L}_{\rm b}}$} ranging from $0.01$ to $2$.
While the bound of Theorem~\ref{th_infnorm_bound} matches closely the condition number for both settings,
the bound of Theorem~\ref{th_naive_bound} is far less accurate. This discrepancy shows that even though
the diffusion models use different parameters, the structural similarities between $\mathbf{B}$ and $\mathbf{R}$
lead to a crucial interaction between them. Therefore, separating the effect of these matrices in the
bound results in a pessimistic upper bound.

In this article, we are interested in obtaining explicit and accurate theoretical bounds on the condition number
in terms of the key parameters of the diffusion-modelled correlation operators presented in
Section~\ref{sec_properties_diffusion}. We restrict $\mathbf{H}$ to a
class of uniform selection operators, which are associated with uniformly distributed observations.
In Section~\ref{sec_selection_operator}, we study how these
selection operators interact with diffusion operators. Based on the results of these two sections,
we examine the conditioning
of $\mathbf{S}$ in Section~\ref{sec_structure_exact_cond}.
Additional results are derived in Section~\ref{sec_structure_approx_cond}
using a simplified matrix $\Shat$, which is equal to $\mathbf{S}$
when $\mathbf{H}$ is the identity matrix but is an approximation otherwise.
\begin{figure}[htb]
\centering
\begin{subfigure}[t]{0.48\textwidth}
\includegraphics[width=\textwidth]{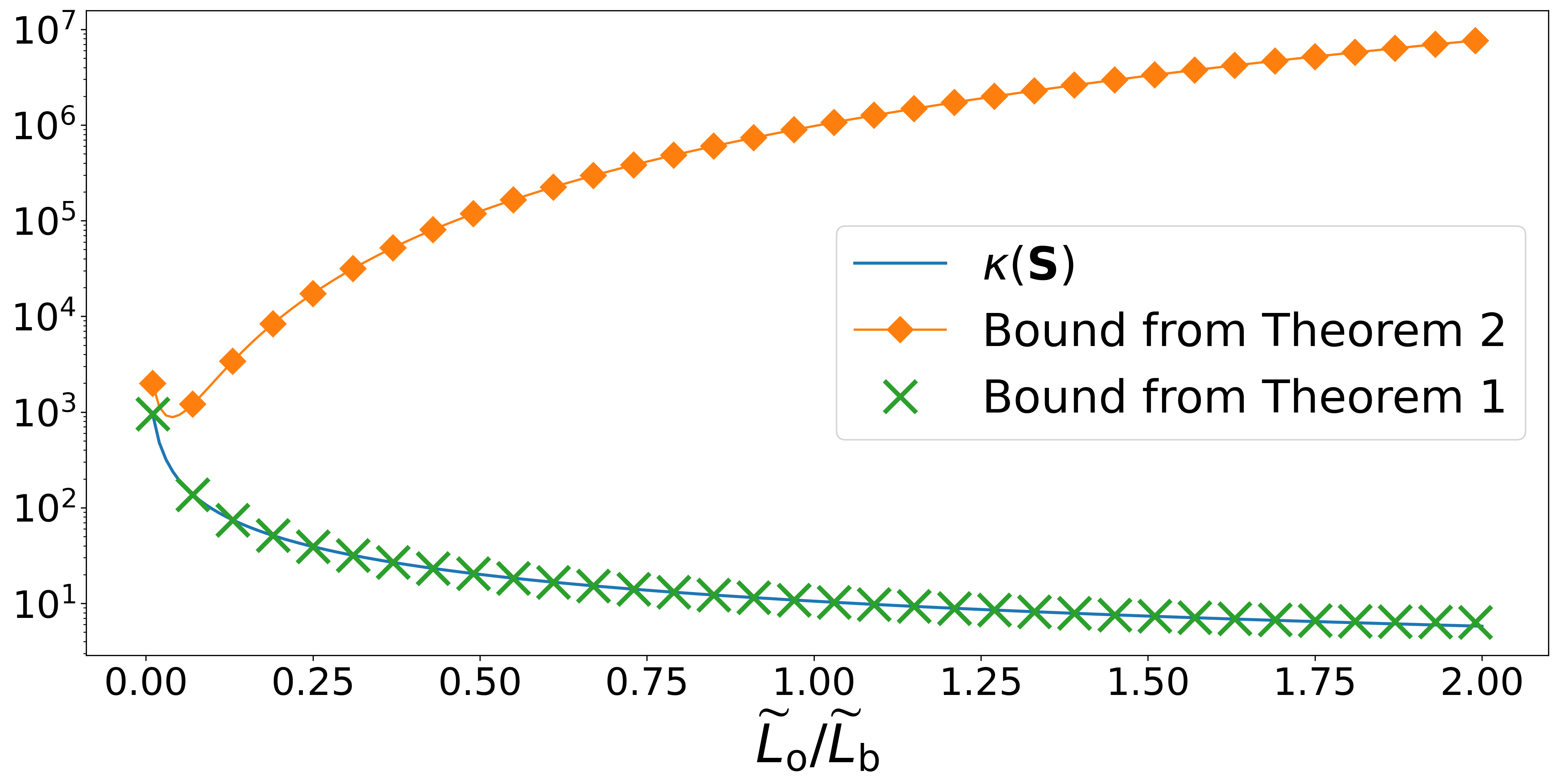}
\caption{The correlation function of the background error is Gaussian-like (\mbox{$M_{\rm b}=8$}),
while the correlation function of the observation error is a SOAR function (\mbox{$M_{\rm o}=2$}).}
\label{fig_naive_bound_left}
\end{subfigure}
\begin{subfigure}[t]{0.48\textwidth}
\includegraphics[width=\textwidth]{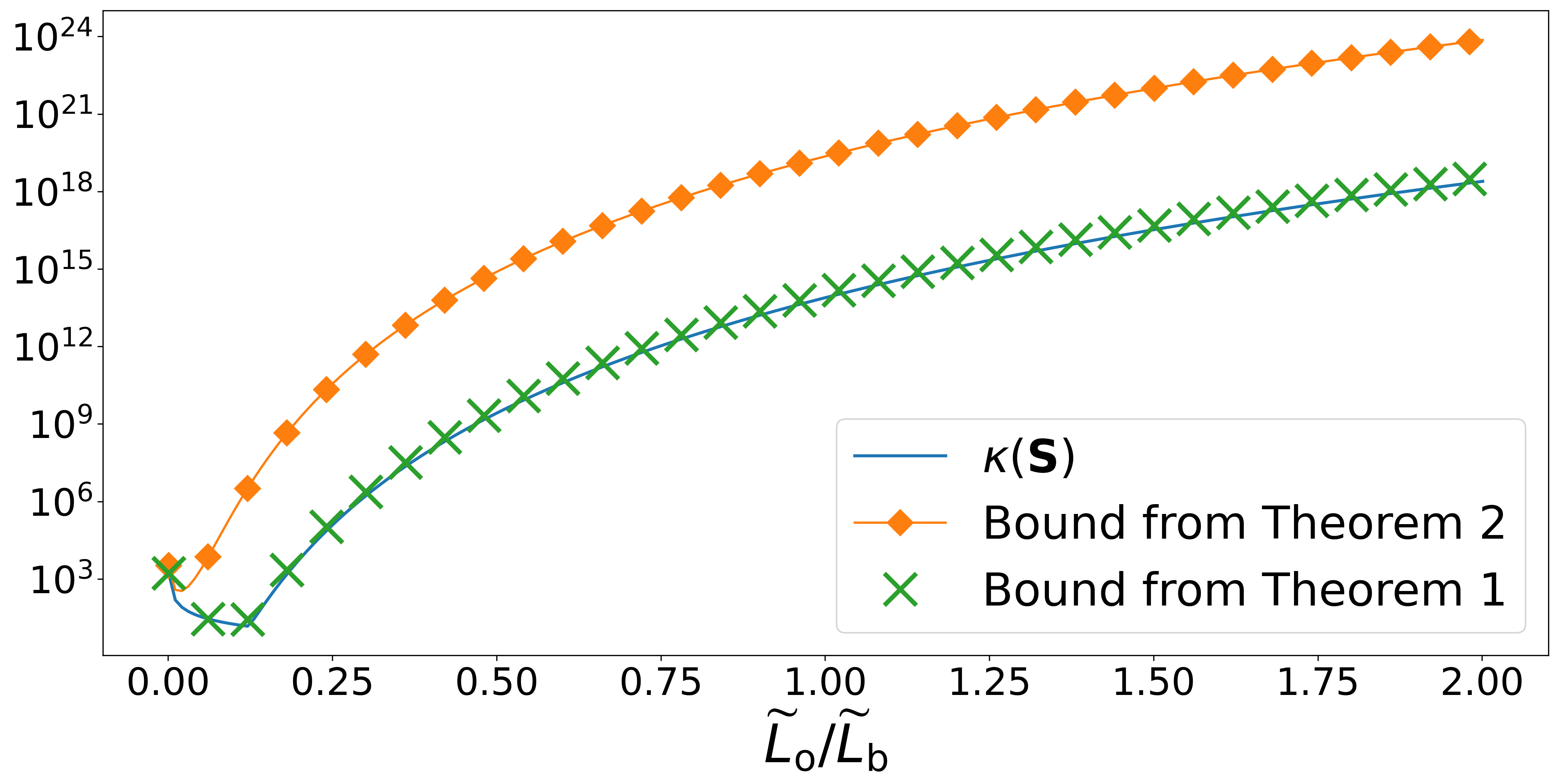}
\caption{The correlation function of the background error is a SOAR function (\mbox{$M_{\rm b}=2$}),
while the correlation function of the observation error is Gaussian-like (\mbox{$M_{\rm o}=8$}).}
\label{fig_naive_bound_right}
\end{subfigure}
\caption{Upper bounds on the condition number of $\mathbf{S}$ compared to the exact condition number
as a function of
\mbox{$\widetilde{L}_{\rm o}/\widetilde{L}_{\rm b} = L_{\rm o} h_{\rm b} / L_{\rm b} h_{\rm o}$}.}
\label{fig_naive_bound}
\end{figure}

\subsection{Spectral properties of diffusion operators}
\label{sec_properties_diffusion}

The diffusion-modelled covariance operators for $\mathbf{B}$ and $\mathbf{R}$
in Equations~\eqref{eq_Bmatrix_diffusion} and \eqref{eq_Rmatrix_diffusion}
are defined in terms of Laplacian matrices \mbox{$\mathbf{\Laplace}_{h_{\rm b}}\in \mathbb{R}^{n\times n}$}
and \mbox{$\mathbf{\Laplace}_{h_{\rm o}}\in \mathbb{R}^{m\times m}$}, formed from a centred finite-difference
discretisation of the Laplacian operator on a uniform grid of resolution
$h_{\rm b}$ and $h_{\rm o}$, respectively. On a periodic domain,
$\mathbf{\Laplace}_{h_{\rm b}}$ and $\mathbf{\Laplace}_{h_{\rm o}}$ are circulant
matrices\footnote{Each column (row) of a circulant matrix is a cyclic permutation of the previous column (row).}
that are \textit{tridiagonal} except for the first and last lines where
additional non-zero elements appear in the corners due to the periodic
boundary conditions. Likewise, the shifted Laplacian matrices \mbox{$\mathbf{T}_{\rm b}\in \mathbb{R}^{n\times n}$}
and \mbox{$\mathbf{T}_{\rm o}\in \mathbb{R}^{m\times m}$}
are circulant, near-tridiagonal matrices. Specifically, for $\mathbf{R}$, we have
\begin{equation}
\mathbf{\Delta}_{h_{\rm o}} \, = \, \frac{1}{h_{\rm o}^2}\left( \begin{matrix}
-2 & 1 & 0 &  & 0 & 1\\
1 & -2 &1 & 0 & & 0\\
0 & \ddots & \ddots & \ddots & &  \\
 & & \ddots & \ddots & \ddots & 0 \\
0 & & 0&1 &-2 & 1\\
1 & 0 &  & 0 & 1 & -2 \end{matrix} \right)
\label{eq_Laplacian}
\end{equation}
and thus
\begin{equation}
\mathbf{T}_{\rm o} \, = \, \left( \begin{matrix}
1+2\widetilde{L}_{\rm o}^2& -\widetilde{L}_{\rm o}^2 & 0 &  & 0 & -\widetilde{L}_{\rm o}^2\\
-\widetilde{L}_{\rm o}^2 & 1+2\widetilde{L}_{\rm o}^2&-\widetilde{L}_{\rm o}^2& 0 & & 0\\
0 & \ddots & \ddots & \ddots & &  \\
 & & \ddots & \ddots & \ddots & 0 \\
0 & & 0& -\widetilde{L}_{\rm o}^2 & 1+2\widetilde{L}_{\rm o}^2 & -\widetilde{L}_{\rm o}^2\\
-\widetilde{L}_{\rm o}^2& 0 &  & 0 & -\widetilde{L}_{\rm o}^2 & 1+2\widetilde{L}_{\rm o}^2 \end{matrix} \right)
\label{eq_T_periodic}
\end{equation}
where \mbox{$\widetilde{L}_{\rm o} = L_{\rm o}/h_{\rm o}$} is a non-dimensional parameter that roughly corresponds to
the number of grid points over which observation-error correlations are significant. The expressions for
 $\mathbf{\Laplace}_{h_{\rm b}}$ and $\mathbf{T}_{\rm b}$ are the same as Equations~\eqref{eq_Laplacian} and
\eqref{eq_T_periodic} with $(h_{\rm b},\widetilde{L}_{\rm b})$ instead of $(h_{\rm o}, \widetilde{L}_{\rm o})$.

An important property of circulant matrices is that, for a given size,
they all share the same eigenvectors, which form a Fourier basis \citep{Gray_2005}. Consequently,
their eigenvalues can be found by taking the discrete Fourier transform of one of the rows.
Let $[ \mathbf{f}^{(i)} ]_p$ be the $p$-th component of the $i$-th eigenvector of $\mathbf{T}_{\rm o}$
and let $\lambda_i$ be the corresponding eigenvalue:
\begin{equation}
\big[ \mathbf{f}^{(i)} \big]_p=\frac{1}{\sqrt{m}}e^{-2\pi j\frac{i p}{m}}, \qquad %
\lambda_i= 1+4\widetilde{L}^2\sin^2\! \left(\pi\frac{i}{m}\right), \;%
i,p\in\llbracket 0, m-1 \rrbracket,
\label{eq_eig_periodic}
\end{equation}
where $j$ denotes the imaginary unit \mbox{($j^2 = -1$)}.
Since $\mathbf{T}_{\rm o}$ is symmetric, its eigenvalues are real and each of them
is repeated twice; \textit{i.e.}, \mbox{$\lambda_i(\mathbf{T}_{\rm o})=\lambda_{m-i}(\mathbf{T}_{\rm o})$},
except $\lambda_0(\mathbf{T}_{\rm o})$ as the index $i$ stops at $m-1$.
If $m$ is even, $\lambda_{m/2}(\mathbf{T}_{\rm o})$ is also unique.

The covariance matrices $\mathbf{B}$ and $\mathbf{R}$ in Equations~\eqref{eq_Bmatrix_diffusion} and \eqref{eq_Rmatrix_diffusion}
are proportional to a power of the inverse of $\mathbf{T}_{\rm b}$ and $\mathbf{T}_{\rm o}$, respectively.
They are also circulant matrices and diagonal in a Fourier basis described by the vectors $\mathbf{f}^{(i)}$.
Two circulant covariance matrices with different parameter specifications 
then share the same eigenvectors as long as they are applied on the same domain.
In addition, since $\mathbf{T}_{\rm o}$ and $\mathbf{T}_{\rm b}$ are symmetric, the diffusion matrices are symmetric and their
eigenvalues are real and proportional to powers of the inverse of the eigenvalues
of $\mathbf{T}_{\rm o}$ and $\mathbf{T}_{\rm b}$. The eigenvalues of $\mathbf{B}$ and $\mathbf{R}$ are then
\begin{gather}
\lambda_i(\mathbf{B}) = \frac{\displaystyle \sigma_{\rm b}^2 \nu_{\rm b} L_{\rm b}}{\displaystyle h_{\rm b}} \lambda_{i}(\mathbf{T}_{\rm b})^{-M_{\rm b}} =  \sigma_{\rm b}^2\nu_{\rm b} \widetilde{L}_{\rm b}
\left[ 1+4\widetilde{L}_{\rm b}^2 \sin^2\! \left(\pi\frac{i}{n}\right)\right]^{-M_{\rm b}}, \quad i\in\llbracket 0, n-1\rrbracket,\label{eq_eig_B}\\
\lambda_i(\mathbf{R}) = \frac{\displaystyle \sigma_{\rm o}^2 \nu_{\rm o} L_{\rm o}}{\displaystyle h_{\rm b}} \lambda_{i}(\mathbf{T}_{\rm o})^{-M_{\rm o}} = \sigma_{\rm o}^2\nu_{\rm o} \widetilde{L}_{\rm o}
\left[ 1+4\widetilde{L}_{\rm o}^2 \sin^2\! \left(\pi\frac{i}{m}\right)\right]^{-M_{\rm o}}, \quad i\in\llbracket 0, m-1\rrbracket.
\label{eq_eig_R}
\end{gather}


\subsection{Influence of $\mathbf{H}$ as a uniform selection operator }
\label{sec_selection_operator}
In the previous section, we derived the eigenvalue spectra of the covariance matrices when they are defined
as circulant matrices. Our simplifying assumption that the covariance parameters are constant is crucial to ensure
the circulant property of the covariance matrices.
In this section, we exploit these results to analyse the spectrum of the matrix \mbox{$\mathbf{H}\mathbf{B}\mathbf{H}^{\transpose}$} that
appears in $\mathbf{S}$. To do so, we will assume that $\mathbf{H}$ is a \textit{uniform selection} operator;
\textit{i.e.} we have observations every $\zeta$ grid points where $\zeta$ is a positive integer. The total number of observations
is then given by \mbox{$m=n/\zeta$} assuming that $\zeta$ is a divisor of $n$. This assumption will allow us to analyse
the sensitivity of the diffusion parameters in a more explicit way since
 it preserves to some extent the structure of $\mathbf{B}$, as shown in the following lemma.

\begin{lemma}
Let \mbox{$\mathbf{B}\in\mathbb{R}^{n\times n}$} be a symmetric circulant matrix and let \mbox{$\mathbf{H}\in\mathbb{R}^{n\times m}$}
be a uniform selection operator where \mbox{$\zeta m= n$} with $\zeta$ a positive integer.
The matrix \mbox{$\mathbf{H}\mathbf{B}\mathbf{H}^{\transpose} \in \mathbb{R}^{m \times m}$}
is then a symmetric circulant matrix with eigenvalues
 \begin{equation}
 \lambda_i\big( \mathbf{H}\mathbf{B}\mathbf{H}^{\transpose} \big)= \frac{1}{\zeta}\sum_{r=0}^{\zeta-1} \lambda_{i+rm}(\mathbf{B}),
 \label{eq_eig_HBH_proof}
 \end{equation}
 for $i \in \llbracket 0,m-1 \rrbracket$.
\label{lm_HBH_circulant}
\end{lemma}
\begin{proof}
 Let the notation $[\mathbf{A}]_{p,q}$ denote the element on the $p$-th row and $q$-th column of any matrix $\mathbf{A}$.
 As $\mathbf{B}$ is a circulant matrix, it can be diagonalised in a Fourier basis:
 \begin{equation}
 \mathbf{B}\, = \, \mathbf{F}_n \mathbf{\Lambda}_{\rm b} \mathbf{F}_n^{\rm H},
 \label{eq_B_diagonal}
 \end{equation}
 where $\mathbf{\Lambda}_{\rm b}$ is a diagonal matrix, and the elements of $\mathbf{F}_n$ are the (normalized) $n$-th roots of unity:
 \begin{equation}
\forall p,q \in \llbracket 0,n-1 \rrbracket, \quad \big[ \mathbf{F}_n \big]_{p,q} = \frac{1}{\sqrt{n}} \omega_n^{pq}
\hspace{2mm} \text{ with } \hspace{1mm} \omega_n = e^{\frac{2\pi j}{n}}.
\nonumber
 \end{equation}
 The superscript ``H'' stands for conjugate (Hermitian) transpose and \mbox{$\mathbf{F}_n^{\rm H} = \mathbf{F}_n^{-1}$}.
 Starting from Equation~\eqref{eq_B_diagonal}, we have
  \begin{equation}
\mathbf{H} \mathbf{B} \mathbf{H}^{\transpose}=\mathbf{H}\mathbf{F}_n \mathbf{\Lambda}_{\rm b} \big( \mathbf{H}\mathbf{F}_n \big)^{\! \rm H}.
 \label{eq_HBH_decomp1}
 \end{equation}
The matrix $\mathbf{H}\mathbf{F}_n$ is of dimension \mbox{$m \times n$} and is composed of the rows of $\mathbf{F}_n$:
\begin{equation}
\forall p\in \llbracket 0,m-1 \rrbracket,  q\in \llbracket 0,n-1 \rrbracket
\quad \big[ \mathbf{H}\mathbf{F}_n \big]_{p,q} = \big[ \mathbf{F}_n \big]_{\zeta p,q}  = \frac{1}{\sqrt{n}} \omega_n^{\zeta pq}.
\nonumber
\end{equation}
The matrix \mbox{$\mathbf{H}\mathbf{F}_n$} can be linked to $\mathbf{F}_m$, the matrix that diagonalises circulant matrices of
dimension \mbox{$m\times m$}.
The elements of the latter are the $m$-th root of unity:
  \begin{equation}
\forall p,q \in \llbracket 0,m-1 \rrbracket, \quad \big[ \mathbf{F}_m \big]_{p,q} =
\frac{1}{\sqrt{m}} \omega_m^{pq} \hspace{2mm} \text{ with } \hspace{1mm} \omega_m = e^{\frac{2j\pi}{m}},
\nonumber
 \end{equation}
which, as \mbox{$n =\zeta m$},  can be linked to the $n$-th root of unity as
\mbox{$\omega_m = \omega_n^{\zeta}$}. Consequently, for the first $m$ columns of $\mathbf{H}\mathbf{F}_n$, we have
\begin{equation}
\forall p\in \llbracket 0,m-1 \rrbracket, q\in \llbracket 0,m-1 \rrbracket \quad \big[ \mathbf{H}\mathbf{F}_n \big]_{p,q}
= \frac{1}{\sqrt{n}} \omega_m^{pq} = \frac{1}{\sqrt{\zeta}}\big[ \mathbf{F}_m \big]_{p,q} .
\nonumber
  \end{equation}
 The other \mbox{$n-m$} columns of \mbox{$\mathbf{H}\mathbf{F}_n$} can be characterized by using the periodicity of
 the $m$-th root of unity: \mbox{$\omega_m^{pq} = \omega_m^{p(q+rm)}$} for any positive integer $r$.
Therefore, $\mathbf{H}\mathbf{F}_n$ is a matrix concatenated with $\zeta$ copies of $\mathbf{F}_m$:
 \begin{equation}
 \mathbf{H}\mathbf{F}_n \, = \,\frac{1}{\sqrt{\zeta}}  [  \mathbf{F}_m \cdots \mathbf{F}_m ].
\nonumber
 \end{equation}
 Equation~\eqref{eq_HBH_decomp1} can thus be rewritten as
  \begin{equation}
\mathbf{H} \mathbf{B} \mathbf{H}^{\transpose} \, = \, \mathbf{F}_m {\frac{1}{\sqrt{\zeta}}\begin{pmatrix} \mathbf{I}_m \cdots \mathbf{I}_m
\end{pmatrix}\mathbf{\Lambda}_{\rm b}
\begin{pmatrix}  \mathbf{I}_m \\ \vdots \\ \mathbf{I}_m \end{pmatrix}\frac{1}{\sqrt{\zeta}}}\, \mathbf{F}_m^{\rm H}
 \, = \, \mathbf{F}_m \mathbf{\Lambda}_{\rm b}' \mathbf{F}_m^{\rm H}
\label{eq_HBH_decomp2}
\nonumber
\end{equation}
where $\mathbf{\Lambda}_{\rm b}'$ is a diagonal matrix of dimension $m \times m$.
As \mbox{$\mathbf{H} \mathbf{B} \mathbf{H}^{\transpose}$}
is diagonal for the basis defined by the columns of $\mathbf{F}_m$, it is a circulant matrix.
Its eigenvalues are the elements of the diagonal matrix $\mathbf{\Lambda}_{\rm b}'$, which are given by Equation~\eqref{eq_eig_HBH_proof}.
\end{proof}

A matrix-vector product with \mbox{$\mathbf{H}\mathbf{B}\mathbf{H}^{\transpose}$}
is therefore in the range of the column vectors $\mathbf{f}_m^{(i)}$, weighted by
the \textit{average of $\zeta$ evenly-distributed eigenvalues} of $\mathbf{B}$.
 This result can be linked to the notion of \textit{aliasing}. Different vectors  from the Fourier basis
 of dimension $n$ (\textit{e.g.}, column $\mathbf{f}_n^{(i)}$) become indistinguishable and equal to the same frequency
 mode in the Fourier basis of dimension $m$ (\textit{e.g.}, column $\mathbf{f}_m^{(i)}$) once they are sub-sampled, as the
 highest frequencies cannot be resolved by the observation grid. This point
 is illustrated in Figure~\ref{fig_aliasing}, which shows multiple distinct columns of $\mathbf{F}_n$ that all
 take the same values on the observation grid, which are the values of a column
 of $\mathbf{F}_m$.
 As shown in Equation~\eqref{eq_eig_HBH_proof}, the weight associated with a frequency mode
 $\mathbf{f}_m^{(i)}$ in \mbox{$\mathbf{H}\mathbf{B}\mathbf{H}^{\transpose}$} is the average of the weights
 associated with the frequency modes  \mbox{$\mathbf{f}_n^{(i)},\; \mathbf{f}_n^{(i+m)}, \; ...,\; \mathbf{f}_n^{(i+(\zeta-1)m)}$}
 in $\mathbf{B}$, which become equal to  $\mathbf{f}_m^{(i)}$ once sub-sampled by $\mathbf{H}$.

  \begin{figure}[htb]
\includegraphics[width=0.8\textwidth]{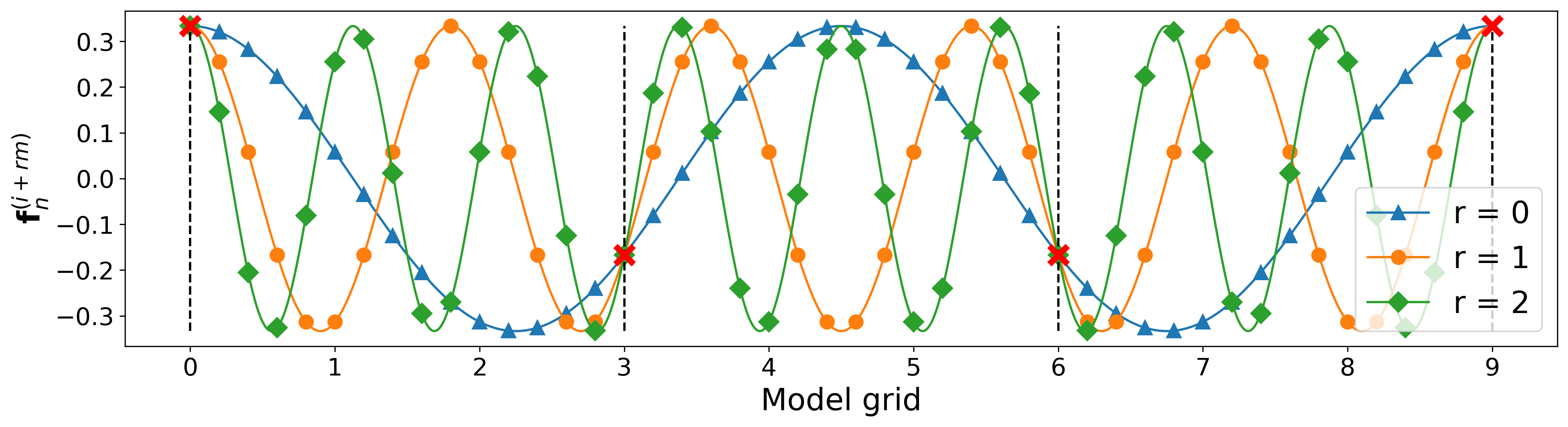}
\centering
\caption{Eigenvectors $\mathbf{f}_n^{(i+rm)}$ of a circulant matrix of size $n \times n$ with $n=9$, for $i=2$,
$r\in\llbracket 0,\zeta-1\rrbracket$ with $\zeta=3$ and $m= n/\zeta = 3$. On the model grid with $n$
points, these eigenvectors can be distinguished. On the observation grid with $m$ points
(marked by dashed lines), they all take the same values, which are the elements of the corresponding
eigenvector  $\mathbf{f}_m^{(i)}$ of a circulant matrix of size $m \times m$.
 Note that the first and last points in the figure represent the same grid point due to
 the periodic boundary conditions. The vectors have been plotted with a higher resolution than the one of
 the model grid  for the sake of clarity.}
\label{fig_aliasing}
\end{figure}

\subsection{Spectral properties of the $\mathbf{B}$-preconditioned Hessian matrix}
\label{sec_structure_exact_cond}

Lemma~\ref{lm_HBH_circulant} implies that $\mathbf{H}\mathbf{B}\mathbf{H}^{\transpose}$ shares the same eigenvectors
as any circulant matrix of size \mbox{$m \times m$}, including the matrices associated with the constant-parameter
diffusion operator applied on the uniform observation grid. This allows us to derive a new expression for the
spectrum of $\mathbf{S}$.
\begin{lemma}
\label{lm_harville}
Let $\mathbf{B}=\mathbf{U}\mathbf{U}^{\transpose}\in\mathbb{R}^{n\times n}$ and $\mathbf{R}\in\mathbb{R}^{m\times m}$, with
$m<n$, be circulant matrices and
let $\mathbf{H}\in\mathbb{R}^{n\times m}$ be a uniform selection operator.
If \mbox{$\mathbf{S} = \mathbf{I}_n + \mathbf{U}^{\transpose}\mathbf{H}^{\transpose}\mathbf{R}^{-1}\mathbf{H}\mathbf{U}$}
then the $i$-th eigenvalue of $\mathbf{S}$ is
\begin{equation}
\lambda_i\big( \mathbf{S} \big) = \begin{cases}
1 + \dfrac{\lambda_i\big( \mathbf{H}\mathbf{B}\mathbf{H}^{\transpose}\big)}{\lambda_i\big( \mathbf{R} \big)} & \text{ if } i \in \llbracket0,m-1\rrbracket,\\
1 &\text{ otherwise. }\end{cases}
\nonumber
\end{equation}
\end{lemma}
\begin{proof}
Let $\Lambda^{*}(\cdot)$ denote the spectrum of a matrix without its zero eigenvalues.
For any matrix $\mathbf{P}$ and $\mathbf{Q}$ of respective sizes \mbox{$n\times m$} and \mbox{$m\times n$},
with \mbox{$n>m$}, we know that~\cite[Theorem~21.10.1]{Harville_1997}
\begin{equation}
\Lambda^{*}\big( \mathbf{P}\mathbf{Q} \big) = \Lambda^{*}\big( \mathbf{Q}\mathbf{P}\big).
\label{eq_theorem_harville}
\end{equation}
With \mbox{$\mathbf{P}=\mathbf{U}^\transpose\mathbf{H}^\transpose$ and $\mathbf{Q}=\mathbf{R}^{-1}\mathbf{H}\mathbf{U}$},
Equation~\eqref{eq_theorem_harville}
implies that \mbox{$\mathbf{U}^\transpose\mathbf{H}^\transpose\mathbf{R}^{-1}\mathbf{H}\mathbf{U}$} has at least \mbox{$n-m$}
eigenvalues equal to zero, and that
\begin{equation}
\Lambda^{*}\big( \mathbf{U}^\transpose\mathbf{H}^\transpose\mathbf{R}^{-1}\mathbf{H}\mathbf{U}\big) =
\Lambda^{*}\big( \mathbf{R}^{-1}\mathbf{H}\mathbf{B}\mathbf{H}^\transpose \big).
\label{eq_nonzero_spectrums}
\nonumber
\end{equation}
Therefore, $\mathbf{S}$ has an eigenvalue of $1$ with multiplicity of \mbox{$n-m$} and the remaining $m$ eigenvalues are
the elements of \mbox{$1 + \Lambda^{*}\big( \mathbf{R}^{-1}\mathbf{H}\mathbf{B}\mathbf{H}^\transpose \big)$}.
Since $\mathbf{H}\mathbf{B}\mathbf{H}^{\transpose}$ and $\mathbf{R}^{-1}$  are circulant matrices (see Lemma~\ref{lm_HBH_circulant}),
they are diagonalisable in the same basis. Therefore, the eigenvalues of their product is the product of their
respective eigenvalues:
\begin{equation}
\forall i \in \llbracket 0, m-1 \rrbracket, \quad   \lambda_i \big(\mathbf{R}^{-1}\mathbf{H}\mathbf{B}\mathbf{H}^{\transpose}\big)
 = \lambda_i \big( \mathbf{R}^{-1} \big)\lambda_i \big( \mathbf{H}\mathbf{B}\mathbf{H}^{\transpose} \big)
 = \dfrac{\lambda_i\big( \mathbf{H}\mathbf{B}\mathbf{H}^{\transpose}\big)}{ \lambda_i\big( \mathbf{R} \big)}.
\nonumber
\end{equation}
\end{proof}

We can now write the eigenvalues of $\mathbf{S}$ in terms of the constant parameters of the
\textit{diffusion-modelled covariance matrices} by using the results of
Lemma~\ref{lm_HBH_circulant} and Lemma~\ref{lm_harville}, and
the expressions for the eigenvalues of $\mathbf{B}$ and $\mathbf{R}$ given
by Equations~\eqref{eq_eig_B} and \eqref{eq_eig_R}, respectively.

\begin{theorem}
Let \mbox{$\mathbf{B}=\mathbf{U}\mathbf{U}^{\transpose}\in\mathbb{R}^{n\times n}$} and
\mbox{$\mathbf{R}\in\mathbb{R}^{m\times m}$} be
circulant matrices defined by Equations~\eqref{eq_Bmatrix_diffusion} and \eqref{eq_Rmatrix_diffusion}
where the shifted Laplacian matrices $\mathbf{T}_{\rm b}$ and $\mathbf{T}_{\rm o}$ are defined in
Section~\ref{sec_properties_diffusion}. Let $\mathbf{H}\in\mathbb{R}^{n\times m}$
be a uniform selection operator where \mbox{$\zeta m = n$} with $\zeta$ a positive integer.
If \mbox{$\mathbf{S} =\mathbf{I}_n + \mathbf{U}^{\transpose}\mathbf{H}^{\transpose}\mathbf{R}^{-1}\mathbf{H}\mathbf{U}$} then
\begin{equation}
\lambda_i(\mathbf{S}) = \begin{cases}
1 + \alpha \displaystyle \sum_{r=0}^{\zeta-1}\limits \dfrac{\left[ 1+4\widetilde{L}_{\rm o}^2
\sin^2 \! \left(\pi\frac{\displaystyle i}{\displaystyle m}\right)\right]^{M_{\rm o}}}{\left[ 1+4\widetilde{L}_{\rm b}^2
\sin^2 \! \left(\pi\frac{\displaystyle i+rm}{\displaystyle \zeta m}\right)\right]^{M_{\rm b}}}
&\text{ if } i \in \llbracket0,m-1\rrbracket,\\
1 &\text{ otherwise, }\end{cases}
\label{eq_eig_S}
\end{equation}
with
\mbox{$\widetilde{L}_{\rm o}=L_{\rm o}/h_{\rm o}$}, \mbox{$\widetilde{L}_{\rm b}=L_{\rm b}/h_{\rm b}$} and
\begin{equation}
\alpha=\frac{\displaystyle \sigma_{\rm b}^2 \nu_{\rm b} L_{\rm b}}{\displaystyle \sigma_{\rm o}^2\nu_{\rm o} L_{\rm o}}.
\nonumber
\end{equation}
\label{th_eig_S}
\end{theorem}
\begin{proof}
 Lemma~\ref{lm_harville} provides an expression for the eigenvalues of $\mathbf{S}$ in terms of the eigenvalues of $\mathbf{R}$
 and $\mathbf{H}\mathbf{B}\mathbf{H}^{\transpose}$. The eigenvalues of $\mathbf{R}$ are known from Equation~\eqref{eq_eig_R},
 and the eigenvalues of $\mathbf{H}\mathbf{B}\mathbf{H}^{\transpose}$ are obtained by applying the result of Lemma~\ref{lm_HBH_circulant}
 to the eigenvalues of $\mathbf{B}$ from Equation~\eqref{eq_eig_B}:
 \begin{equation}
 \lambda_i\left(\mathbf{H}\mathbf{B}\mathbf{H}^{\transpose}\right) = \frac{\sigma_{\rm b}^2 \nu_{\rm b} L_{\rm b}}{h_{\rm o}} \sum_{i=0}^{m-1} \left[ 1+4\widetilde{L}_{\rm b}^2 \sin^2\! \left(\pi\frac{i+rm}{\zeta m}\right)\right]^{-M_{\rm b}}, \quad i\in\llbracket 0, m-1\rrbracket,
 \end{equation}
 as \mbox{$h_{\rm o} = \zeta h_{\rm b}$} and \mbox{$n = \zeta m$}.
\end{proof}

From Theorem~\ref{th_eig_S}, it is clear that the minimum eigenvalue of $\mathbf{S}$, $\lambda_{\min}(\mathbf{S})$,
is equal to one when \mbox{$m<n$} (fewer observations than background variables), and is bounded below by 1 when \mbox{$m=n$}.
The condition number of $\mathbf{S}$ is thus bounded above by the maximum eigenvalue of $\mathbf{S}$, $\lambda_{\max}(\mathbf{S})$.
There is no simple analytical expression for $\lambda_{\max}(\mathbf{S})$ that can be deduced from Theorem~\ref{th_eig_S}.
However, we can already notice that $\lambda_{\max}(\mathbf{S})$ increases with increasing ratio between
the background- and observation-error variances, $\sigma_{\rm b}^2/\sigma_{\rm o}^2$.
This basic dependency of the condition number on the relative variances is
well known \citep{Andersson_2000,Haben_2011,Tabeart_2021}.

\subsubsection{Spectral properties of the simplified $\mathbf{B}$-preconditioned Hessian matrix}
\label{sec_structure_approx_cond}
  Analysing the sensitivity of $\lambda_{\max}\big( \mathbf{S} \big)$ with respect to the diffusion parameters
  is not straightforward from the expression given in Theorem~\ref{th_eig_S}.
We can obtain a better understanding by approximating
the effect of the matrix $\mathbf{H}\mathbf{B}\mathbf{H}^{\transpose}$ with a diffusion operator discretised directly
on the observation grid.
Specifically,
 let  \mbox{$\Bhat \in \mathbb{R}^{m \times m}$} be a diffusion operator with the same covariance parameters
 as \mbox{$\mathbf{B} \in \mathbb{R}^{n \times n}$} but discretised on the observation grid:
\begin{equation}
\Bhat  =\frac{ \sigma_{\rm b}^2 \nu_{\rm b} L_{\rm b}}{h_{\rm o}} \left( \mathbf{I}_m-L_{\rm b}^2 \mathbf{\Laplace}_{h_{\rm o}} \right)^{-M_{\rm b}}.
\label{eq_approx_Bhat}
\end{equation}
As  $\mathbf{H}\mathbf{B}\mathbf{H}^{\transpose}$ and $\Bhat$ are two spatial discretisations of the same continuous
diffusion operator, the difference between the two is solely  due to the error associated with the
spatial discretisation. If there are direct observations at each grid point ($\mathbf{H}=\mathbf{I}_n$),
both $\Bhat$  and $\mathbf{H}\mathbf{B}\mathbf{H}^{\transpose}$ are equal to $\mathbf{B}$ and there is no approximation.
If there are less observations than grid points, $\Bhat$ and $\mathbf{H}\mathbf{B}\mathbf{H}^{\transpose}$ still share
the same eigenvectors  as they are both circulant. However, they have slightly different eigenvalues due
to the different spatial discretisations. As $\Bhat$ is a diffusion operator, its eigenvalues can be deduced
from the results of Section~\ref{sec_properties_diffusion}:
\begin{equation}
\lambda_i(\Bhat) = \sigma_{\rm b}^2\nu_{\rm b} \Lbhat \left[ 1+4\Lbhat\sin^2\! \left(\pi\frac{i}{m}\right)\right]^{-M_{\rm b}},
\label{eq_eig_Bo}
\end{equation}
where \mbox{$\Lbhat = L_{\rm b}/h_{\rm o}$}. The eigenvalues of $\Bhat$ tend to overestimate the eigenvalues
of $\mathbf{H}\mathbf{B}\mathbf{H}^{\transpose}$, with maximum relative error occurring for the smallest eigenvalues,
as illustrated in Figure~\ref{fig_eig_Bo_HBHt} for the case where \mbox{$h_{\rm o}/h_{\rm b} = 2$}.

\begin{figure}[htb]
\includegraphics[width=1\textwidth]{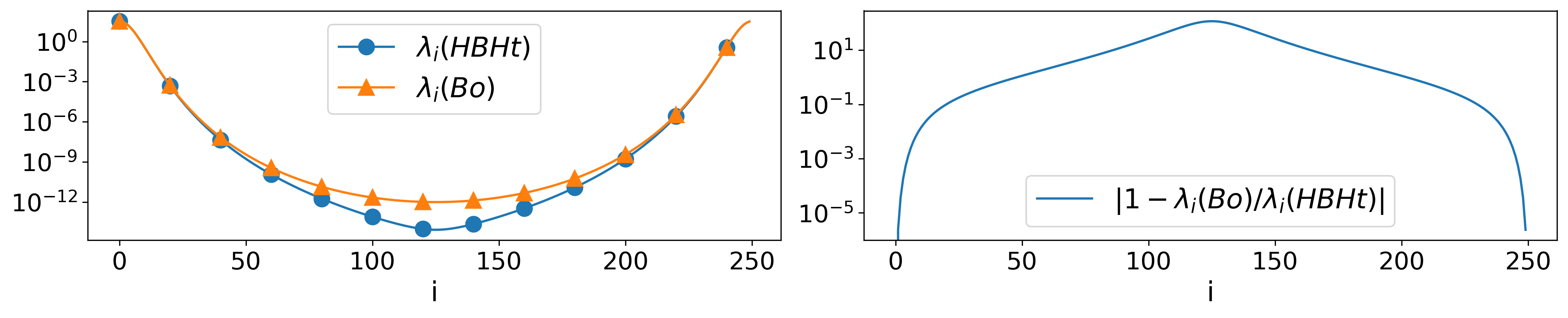}
\centering
\caption{The eigenvalues of $\mathbf{H}\mathbf{B}\mathbf{H}^{\transpose}$ and $\Bhat$ on a domain of $2000$ km,
with a model grid of \mbox{$n=500$} points and an observation at every other model grid point (\mbox{$m=250$}).
The correlation model represented by $\mathbf{B}$ is an AR function of order \mbox{$M_{\rm b}=8$} with a Daley
length-scale of \mbox{$D_{\rm b} = L_{\rm b} \sqrt{2M_{\rm b} -3} = 100$~km}.}
\label{fig_eig_Bo_HBHt}
\end{figure}

Let us recall that $\mathbf{S}$ has an eigenvalue of $1$ with multiplicity of \mbox{$n-m$} and that the remaining $m$ eigenvalues are
the elements of \mbox{$1 + \Lambda^{*}\big( \mathbf{R}^{-1}\mathbf{H}\mathbf{B}\mathbf{H}^\transpose \big)$} (see Lemma~\ref{lm_harville}).
Approximating the matrix $\mathbf{H}\mathbf{B}\mathbf{H}^\transpose$ by $\Bhat$, we are now interested in determining
the eigenvalues of the matrix
\begin{equation}
\Shat = \mathbf{I}_m + \mathbf{R}^{-1} \mathbf{B}_{\rm o}.
\nonumber
\end{equation}


\begin{theorem}
Let \mbox{$\Bhat \in\mathbb{R}^{m\times m}$} be the circulant matrix
defined by Equation~\eqref{eq_Bmatrix_diffusion} with \mbox{$h_{\rm b} = h_{\rm o}$},
and let \mbox{$\mathbf{R}\in\mathbb{R}^{m\times m}$} be the circulant matrix
defined by Equation~\eqref{eq_Rmatrix_diffusion}.
The shifted Laplacian matrices $\mathbf{T}_{\rm b}$ and $\mathbf{T}_{\rm o}$ are defined in
Section~\ref{sec_properties_diffusion}.
If \mbox{$\Shat =\mathbf{I}_m + \ \mathbf{R}^{-1}\Bhat$} then
\begin{equation}
\forall i\in \llbracket 0,m-1 \rrbracket, \quad\lambda_i\big( \Shat \big) =
1+\alpha \frac{\left[ 1+4\widetilde{L}_{\rm o}^2
\sin^2 \! \left( \pi\frac{\displaystyle i}{\displaystyle m}\right) \right]^{M_{\rm o}}}
{\left[ 1+4\Lbhat^2\sin^2 \! \left( \pi\frac{\displaystyle i}{\displaystyle m}\right) \right]^{M_{\rm b}}},
\label{eq_eig_So}
\end{equation}
with
\begin{equation}
\alpha= \frac{\displaystyle \sigma_{\rm b}^2\nu_{\rm b} \Lbhat}{\displaystyle \sigma_{\rm o}^2\nu_{\rm o} \widetilde{L}_{\rm o}}
= \frac{\displaystyle \sigma_{\rm b}^2\nu_{\rm b} {L}_{\rm b} }{\displaystyle \sigma_{\rm o}^2\nu_{\rm o} L_{\rm o}}.
\label{eq_alpha}
\end{equation}
\label{th_eig_So}
\nonumber
\end{theorem}
\begin{proof}
The eigenvalues of $\Shat$ are given by
\mbox{$\lambda_i\big( \Shat \big)=1+\lambda_i\big( \mathbf{R}^{-1}\Bhat \big)$}.
As $\Bhat$ and $\mathbf{R}^{-1}$ are both circulant matrices, they share the same eigenvectors. Therefore, we have
\begin{equation}
 \lambda_i\big( \Shat \big) = 1 +\frac{\lambda_i\big( \Bhat \big)}{\lambda_i \big( \mathbf{R} \big)}.
 \nonumber
\end{equation}
Replacing the eigenvalues of $\Bhat$ and $\mathbf{R}$ by their expressions provided by
Equations~\eqref{eq_eig_Bo} and \eqref{eq_eig_R}, respectively,
yields the expression for the eigenvalues of $\Shat$.
\end{proof}
The next theorem provides a bound on the condition number of $\Shat$ by using the expression
for the eigenvalues of $\Shat$.

\begin{theorem}
\label{th_diffusion_bound}
    Let $\Shat$ be defined as in Theorem~\ref{th_eig_So}, and let $\alpha$ be given by Equation~\eqref{eq_alpha}.
    Then, $\kappa( \Shat ) \leq \eta$ where
		\begin{equation}
		 \eta =\begin{cases}%
        1+\alpha\left(\frac{\displaystyle \widetilde{L}_{\rm o}^2}{\displaystyle M_{\rm b}}\right)^{M_{\rm b}}
        \left(\frac{\displaystyle M_{\rm o}}{\displaystyle \Lbhat^2}\right)^{M_{\rm o}}
        \left(\frac{\displaystyle M_{\rm b}-M_{\rm o}}{\displaystyle \widetilde{L}_{\rm o}^2 - \Lbhat^2}\right)^{M_{\rm b}-M_{\rm o}}  & \text{ if }%
        \hspace{2mm} \text{(i) } \widetilde{L}_{\rm o}^2M_{\rm o} > \Lbhat^2M_{\rm b};  \hspace{2mm} \text{(ii) } M_{\rm o}<M_{\rm b};
        \text{ and }  \\
        & \hspace{6mm}
        \text{(iii) } \Lbhat^2M_{\rm b}  - \widetilde{L}_{\rm o}^2M_{\rm o} > 4  \Lbhat^2 \widetilde{L}_{\rm o}^2 ( M_{\rm o} - M_{\rm b})
        \\[20pt]%
      1+\alpha \max\left\{\frac{\displaystyle (1+4\widetilde{L}_{\rm o}^2)^{M_{\rm o}}}{\displaystyle (1+4\Lbhat^2)^{M_{\rm b}}}, \; 1\right\}
      &\text{ otherwise.}%
        \end{cases}
        \label{eq_diffusion_bound}
		\end{equation}
\end{theorem}
\begin{proof}
See Appendix \ref{app_proof_theorem5}.
\end{proof}

As explained further in this section, Theorem~\ref{th_diffusion_bound} describes the sensitivity of the
condition number of $\Shat$ to the diffusion parameters \textit{while keeping the bound sharp}. The sharpness
of the bound is illustrated  in Figure~\ref{fig_improved_bound}, where it is compared with
the exact condition number of $\mathbf{S}$
and the bounds given in Theorem~\ref{th_infnorm_bound} and Theorem~\ref{th_naive_bound}.
The exact condition number has been evaluated using the extreme eigenvalues taken from the full
spectrum of exact eigenvalues provided by the expression in Theorem~\ref{th_eig_S}.
These results show that
taking into account the specific structure of the covariance matrices improves the bound relative to the
one given in Theorem~\ref{th_naive_bound}.

\begin{figure}[tb]
\centering
\begin{subfigure}[t]{0.48\textwidth}
\includegraphics[width=\textwidth]{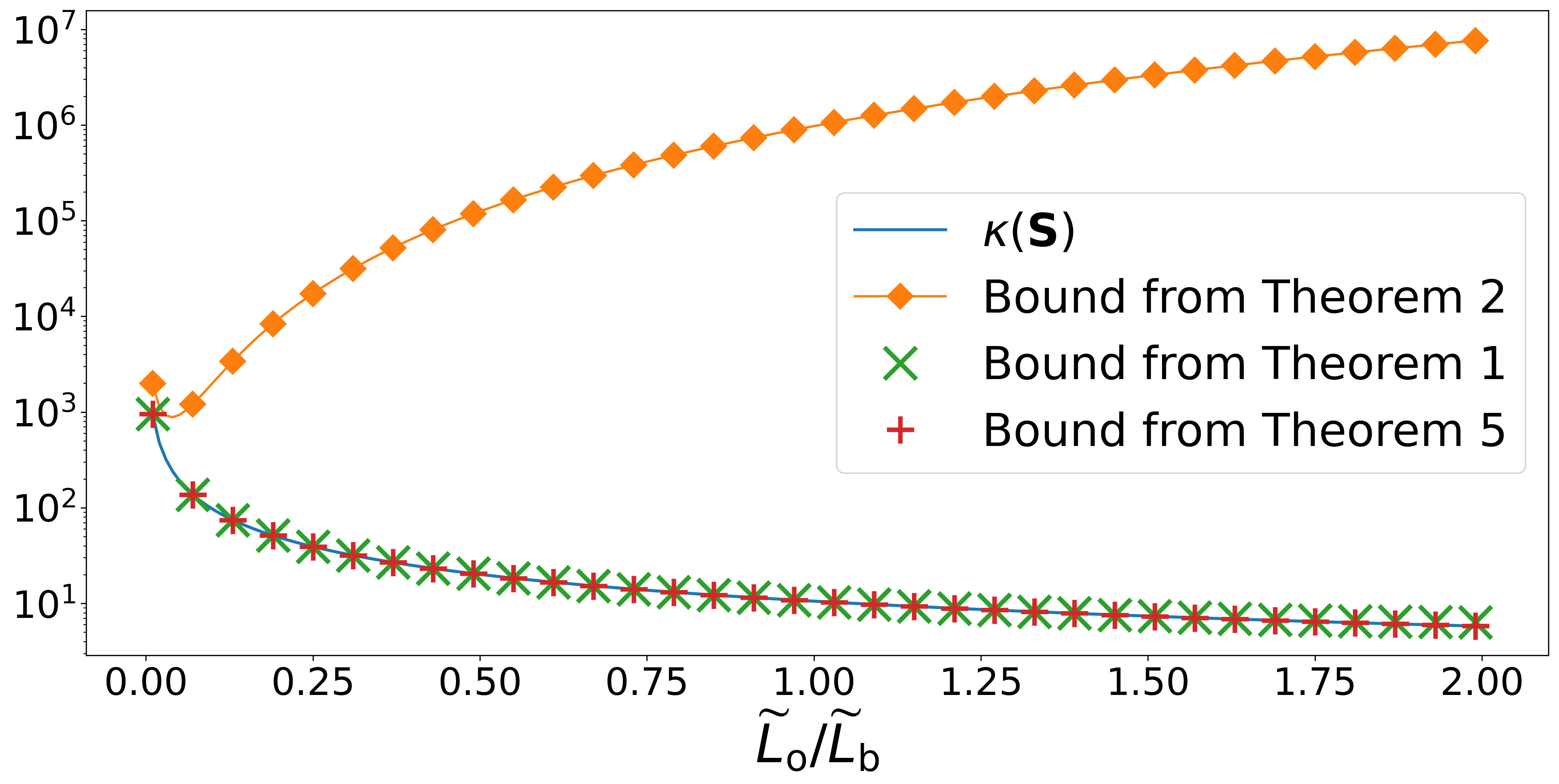}
\centering
\caption{\centering $M_{\rm o}=2$, $M_{\rm b}=8$}
\end{subfigure}
\begin{subfigure}[t]{0.48\textwidth}
\includegraphics[width=\textwidth]{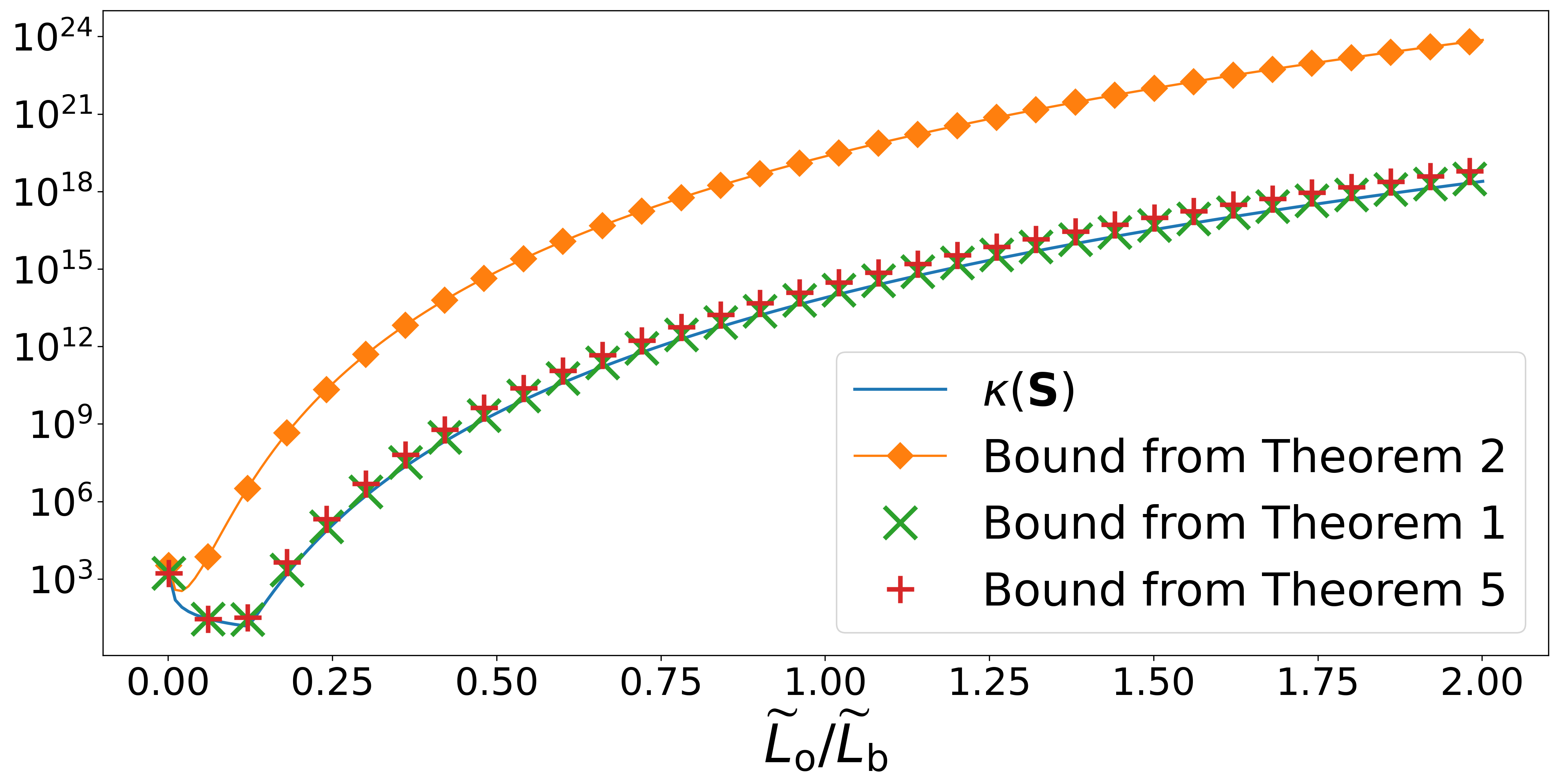}
\caption{\centering $M_{\rm o}=8$, $M_{\rm b}=2$}
\end{subfigure}
\caption{As in Figure~\ref{fig_naive_bound} but with an additional curve for the new bound from
Theorem~\ref{th_diffusion_bound}.}
\label{fig_improved_bound}
\end{figure}

We can further simplify the result in Theorem~\ref{th_diffusion_bound} by
considering $\eta$ as a function of $\widetilde{L}_{\rm o}$ only.
Corollary~\ref{co_minima_Mb_smaller} and Corollary~\ref{co_minima_Mo_smaller} below
characterize the variations of $\eta$ with respect to $\widetilde{L}_{\rm o}$
when \mbox{$M_{\rm o}\geq M_{\rm b}$} and \mbox{$M_{\rm o}< M_{\rm b}$}, respectively.

\begin{corollary}
Consider that $\eta$ defined in Theorem~\ref{th_diffusion_bound} is a function of $\widetilde{L}_{\rm o}$;
\mbox{$\eta = f(\widetilde{L}_{\rm o})$}. Assume that $\widetilde{L}_{\rm o}$
has a lower bound such that 
\mbox{$\widetilde{L}_{\rm o} \sqrt{2 M_{\rm o}-1} > 1/2$} and that condition~(ii)
from Theorem~\ref{th_diffusion_bound} is \emph{not} met; {\it i.e.}, $M_{\rm o} \geq M_{\rm b}$.
Then, the minimum of the function $f$ is unique and reached when
 \begin{equation}
 \big( 1+4\widetilde{L}_{\rm o}^2 \big)^{M_{\rm o}} = \big( 1+4\Lbhat^2 \big)^{M_{\rm b}}.
\label{eq_minima_Mb_smaller}
 \end{equation}
\label{co_minima_Mb_smaller}
\end{corollary}
\begin{proof}
See Appendix~\ref{app_proof_corollary1}.
\end{proof}

\begin{corollary}
Consider that $\eta$ defined in Theorem~\ref{th_diffusion_bound} is a function of $\widetilde{L}_{\rm o}$;
\mbox{$\eta = f(\widetilde{L}_{\rm o})$}. Assume that condition~(ii) from Theorem~\ref{th_diffusion_bound}
is met (\mbox{$M_{\rm o} < M_{\rm b}$}). Assume further that condition~(iii) holds when condition~(i) is satisfied.
Then, the minimum of the function $f$ is unique and reached when
 \begin{equation}
  L_{\rm o} \sqrt{ 2M_{\rm o}-1 }=L_{\rm b} \sqrt{2M_{\rm b}-1}
\label{eq_minima_Mo_smaller}
 \end{equation}
\label{co_minima_Mo_smaller}
\end{corollary}
\begin{proof}
See Appendix~\ref{app_proof_corollary2}.
\end{proof}

Although the assumptions in Corollary~\ref{co_minima_Mb_smaller} and Corollary~\ref{co_minima_Mo_smaller}
appear restrictive, they exclude cases that are of limited practical interest.
In particular, values of $\widetilde{L}_{\rm o}$ that are smaller than the observation grid resolution $h_{\rm o}$
correspond to observation errors that are effectively uncorrelated. To avoid this case, we focus on values of
\mbox{$\widetilde{L}_{\rm o} \geq 1$}, which automatically fulfils the condition on the lower bound
in Corollary~\ref{co_minima_Mb_smaller}. Alternatively, the condition on the lower bound can be seen
as restricting the Stein length-scale \mbox{$\rho_{\rm o} = L_{\rm o} \sqrt{2M_{\rm o} - 1}$} (see Equation~\eqref{eq_rho})
to be greater than half the grid resolution.
In Corollary~\ref{co_minima_Mo_smaller}, we assume that condition~(iii)
of Theorem~\ref{th_diffusion_bound} holds if condition~(i) is satisfied. To simplify condition~(iii),
we can impose a practical bound on the value of $\Lbhat$. For instance, we are not interested in cases
where the length-scale $L_{\rm b}$ is smaller than $h_{\rm o}$, \emph{i.e.}, \mbox{$\Lbhat \le 1$}.
More generally, we can assume that $\Lbhat$ is bounded below
by a positive scalar $\widetilde{L}_{\rm min}$, which leads to the next corollary.
\begin{corollary}
\label{co_simplify_conditions}
Assume that there exists a positive scalar $\widetilde{L}_{\rm min}$ such that
 \begin{align}
  \Lbhat &\geq \widetilde{L}_{\rm min},
\end{align}
and that condition~(ii) of Theorem~\ref{th_diffusion_bound} holds.
Then, condition~(iii) of Theorem~\ref{th_diffusion_bound} simplifies to
 \begin{equation}
 M_{\rm b} \leq 2 \left( 1+4\widetilde{L}_{\rm min}^2 \right ).
 \label{eq:Mb_ineq}
\end{equation}
\end{corollary}
\begin{proof}
Let us define \mbox{$r_M = M_{\rm o} / M_{\rm b}$}. Using the assumption on the length-scale, condition~(iii)
of Theorem~\ref{th_diffusion_bound} can be rewritten as
\begin{align*}
\widetilde{L}_{\rm min}^2  +  4 \widetilde{L}_{\rm min}^2 \widetilde{L}_{\rm o}^2 ( 1 - r_{M} )
- \widetilde{L}_{\rm o}^2 r_M \ge 0.
\end{align*}
As $M_{\rm o}$ and $M_{\rm b}$ are assumed to be even integers and as condition (ii)
is met (\mbox{$M_{\rm o}< M_{\rm b}$}), we know
that \mbox{$r_M \le (M_{\rm b} -2) / M_{\rm b}$}. Using this relation, we obtain that
\begin{align*}
 M_{\rm b} \widetilde{L}_{\rm min}^2  +  8 \widetilde{L}_{\rm min}^2 \widetilde{L}_{\rm o}^2
 - \widetilde{L}_{\rm o}^2 (M_{\rm b} - 2)  \ge 0  &,
\end{align*}
which can be rearranged to give
\begin{align*}
M_{\rm b} \left (1 - {\widetilde{L}_{\rm min}^2}/{\widetilde{L}_{\rm o}^2} \right )
 \le 2 \left ( 1 + 4 \widetilde{L}_{\rm min}^2 \right). &
\end{align*}
Since ${\widetilde{L}_{\rm min}^2}/{\widetilde{L}_{\rm o}^2}$ is positive, we obtain the
inequality \eqref{eq:Mb_ineq}.
\end{proof}
Taking \mbox{$\widetilde{L}_{\rm min} = 1$} in Equation~\eqref{eq:Mb_ineq} results in \mbox{$M_{\rm b} \leq 10$}.
Increasing $M_{\rm b}$ beyond 10 has little practical value
as the correlation function is already approximately Gaussian with this value.

For the case where $\mathbf{B}$ and $\mathbf{R}$ are modelled with SOAR functions (\mbox{$M_{\rm b}=M_{\rm o}=2$}),
\cite{Tabeart_2021} point out that, for fixed $L_{\rm b}$, the minimum of their upper bound for the condition number
of the $\mathbf{B}$-preconditioned Hessian matrix is found by setting \mbox{$L_{\rm o}=L_{\rm b}$}.
Corollary~\ref{co_minima_Mb_smaller} and Corollary~\ref{co_minima_Mo_smaller} confirm this result and extend it
to other AR functions (\mbox{$M_{\rm o}=M_{\rm b}>2$}). They also cover cases where the order of the AR functions
differs between $\mathbf{B}$ and $\mathbf{R}$ (\mbox{$M_{\rm b} \neq M_{\rm o}$}), in which case the function defining
the upper bound on the condition number, \mbox{$\eta = f(\widetilde{L}_{\rm o})$}, does not reach its minimum value
when \mbox{$L_{\rm o}= L_{\rm b}$}.

If \mbox{$M_{\rm o} > M_{\rm b}$ then $\widetilde{L}_{\rm o}$} can be much smaller than $\Lbhat$ to attain the minimum
of the function $f(\widetilde{L}_{\rm o})$. For example, if \mbox{$M_{\rm b}=2$}, \mbox{$M_{\rm o}=10$}
and \mbox{$\widetilde{L}_{\rm o}=1.5$},
then $\Lbhat$ needs to be 158 to satisfy Equation~\eqref{eq_minima_Mb_smaller} of Corollary~\ref{co_minima_Mb_smaller}.
The correlation functions with fixed values of $\big( 1+4\widetilde{L}^2 \big)^{M}$ have very
different range for low values of $M$ as illustrated in Figure~\ref{fig_matching_cond_length}a.
On the other hand, if \mbox{$M_{\rm o} < M_{\rm b}$} then Corollary~\ref{co_minima_Mo_smaller} states that the minimum value
is attained when the Stein length-scales \mbox{$\rho_{\rm o} = L_{\rm o} \sqrt{2M_{\rm o} - 1}$}
and  \mbox{$\rho_{\rm b} = L_{\rm b} \sqrt{2M_{\rm b} - 1}$} are equal.
Note that, unlike condition \eqref{eq_minima_Mb_smaller}, condition \eqref{eq_minima_Mo_smaller} is independent of $h_{\rm o}$.
As shown in Figure~\ref{fig_matching_cond_length}b, the correlation functions with fixed values of $\rho$
are very similar for different values of $M$.

 \begin{figure}[h]
\begin{subfigure}[t]{\textwidth}
\centering
\includegraphics[width = 0.9\textwidth]{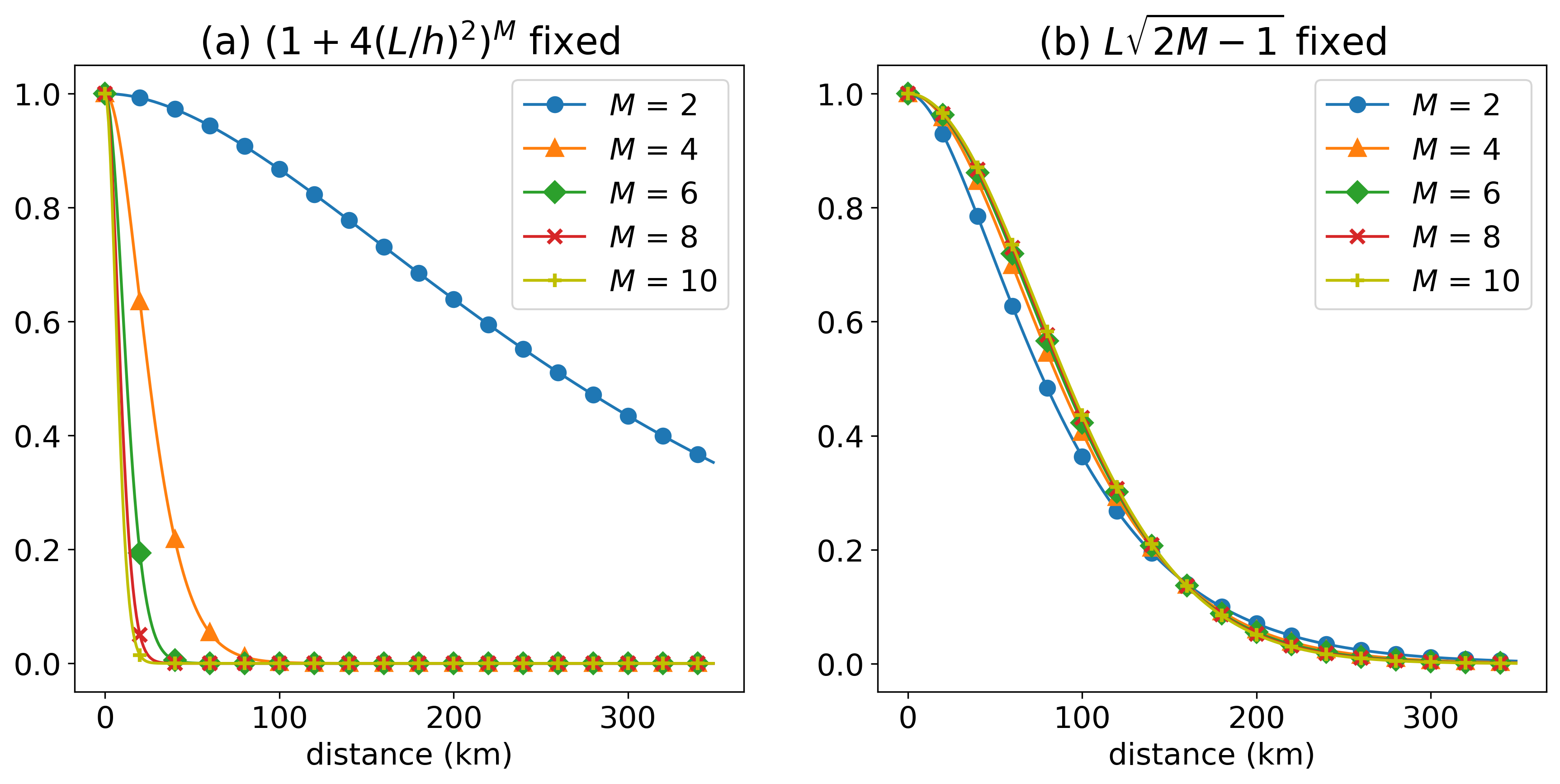}
\end{subfigure}
\caption{AR correlation functions (Equation \eqref{eq:c}) displayed for different values of $M$. For each $M$, the
value of $\widetilde{L}$ is chosen such that the quantities
(a) $\big( 1+4\widetilde{L}^2 \big)^{M}$ and (b) $L\sqrt{2M-1}$
from Corollary~\ref{co_minima_Mb_smaller} and Corollary~\ref{co_minima_Mo_smaller}, respectively,
are kept constant.
The corresponding values of \mbox{$L$}, the Stein length-scale \mbox{$\rho = L\sqrt{2M-1}$},
and the Daley length-scale \mbox{$D = L\sqrt{2M-3}$}
can be found in Table~\ref{tab:length_scales}.}
\label{fig_matching_cond_length}
\end{figure}

\begin{table}[h]
\centering
\renewcommand{\arraystretch}{1.3}
\begin{tabular}{|c|c|c|c|c|c|c|}
\hline
 & \multicolumn{3}{c|}{$(1+4\widetilde{L}^2)^M$ fixed} &  \multicolumn{3}{c|}{$L\sqrt{2M-1}$ fixed}\\
\hline
$M$  & $L \; \textrm{(km)} $ & $\rho \; \textrm{(km)}$ & $D \;\textrm{(km)}$  & $L  \; \textrm{(km)}$ & $\rho \; \textrm{(km)}$ & $D \;\textrm{(km)}$ \\
\hline
2
& 158.1 & 273.8 & 158.1
&  46.2 & 80.0 & 46.2 \\
\hline
4
& 8.9 & 23.5 & 19.9
&  30.2 & 80.0 & 67.6 \\
\hline
6
& 3.4 & 11.2 & 10.1
&  24.1 & 80.0 & 72.4 \\
\hline
8
& 2.0 & 7.7 & 7.4
&  20.7 & 80.0 & 74.5 \\
\hline
10
& 1.5 & 6.5 & 6.2
&  18.3 & 80.0 & 75.7 \\
\hline
\end{tabular}
\caption{ Values of the length-scale parameter \mbox{$L = \widetilde{L}h$} where \mbox{$h=1$}~km,
the Stein length-scale \mbox{$\rho = L \sqrt{2M-1}$}, and
the Daley length-scale \mbox{$D = L \sqrt{2M-3}$} associated with the curves in Figure~\ref{fig_matching_cond_length}.
The fixed values of $\big( 1+4\widetilde{L}^2 \big)^{M}$ and $L\sqrt{2M-1}$ are $10^{10}$ and $80$~km, respectively.}
\label{tab:length_scales}
\end{table}

For an alternative interpretation of Corollary \ref{co_minima_Mb_smaller} and Corollary \ref{co_minima_Mo_smaller}, we can recast
Equations~\eqref{eq_minima_Mb_smaller} and \eqref{eq_minima_Mo_smaller} in terms of the Daley length-scales
$D_{\rm o}$ and $D_{\rm b}$ (Equation~\eqref{eq_D}), which is the length-scale parameter we will use to interpret the numerical
experiments in the following sections.
Assuming \mbox{$M_{\rm b} > 1$} and \mbox{$M_{\rm o} > 1$}, we have
\begin{equation}
L_{\rm b} = \frac{D_{\rm b}}{\sqrt{2 M_{\rm b} - 3}} \hspace{4mm} \mbox{and}
\hspace{4mm} L_{\rm o} = \frac{D_{\rm o}}{\sqrt{2 M_{\rm o} - 3}}.
\label{eq:def_D_recall}
\end{equation}
If \mbox{$M_{\rm o}=M_{\rm b}$} then the minimum is reached when \mbox{$D_{\rm o}=D_{\rm b}$}.

If \mbox{$M_{\rm o}> M_{\rm b}$} then Equation~\eqref{eq_minima_Mb_smaller} translates as
\begin{equation}
\left( 1+ \frac{4 D_{\rm o}^2}{h_{\rm o}^2 (2M_{\rm o}-3)} \right) ^{M_{\rm o}}  =
\left( 1+ \frac{4 D_{\rm b}^2}{h_{\rm o}^2 (2M_{\rm b}-3)} \right) ^{M_{\rm b}}.
\label{eq_minima_Mb_smaller_D}
\end{equation}
The location of the minima is very sensitive
to $M_{\rm o}$ and $M_{\rm b}$ since they appear as exponents in Equation~\eqref{eq_minima_Mb_smaller_D}.
While a small change of $M_{\rm o}$ from 8 to 10 would have limited effect on the
correlation function, it can have a drastic effect on the quantities in
Equation~\eqref{eq_minima_Mb_smaller_D}. In turn, this can significantly
affect the condition number (as seen from Theorem~\ref{th_diffusion_bound})
as well as the criteria in Corollary~\ref{co_minima_Mb_smaller}.
This property can be detrimental if ignored, but can also be exploited to our advantage
to improve the conditioning without significantly altering the correlation shape,
as will be illustrated in Section~\ref{sec_numerical_experiments}.

If \mbox{$M_{\rm o}<M_{\rm b}$} then Equation~\eqref{eq_minima_Mo_smaller} translates as
\begin{equation}
D_{\rm o}^2 \left( \frac{2M_{\rm o}-1}{2M_{\rm o}-3} \right) =
D_{\rm b}^2 \left( \frac{2M_{\rm b}-1}{2M_{\rm b}-3} \right),
\label{eq_minima_Mo_smaller_D}
\end{equation}
from which we can deduce that the minimum is reached when
\mbox{$D_{\rm o}<D_{\rm b}$} (cf. \mbox{$L_{\rm o}>L_{\rm b}$} and \mbox{$\rho_{\rm o} = \rho_{\rm b}$}).
This is evident from the last column
of Table~\ref{tab:length_scales}, which shows $D$ increasing with increasing $M$.
The ratio between $D_{\rm b}$ and $D_{\rm o}$ reaches at most $1.6$ for the limiting values of
\mbox{$M_{\rm o}=2$} and \mbox{$M_{\rm b}=10$}.

Equations~\eqref{eq_minima_Mb_smaller} and \eqref{eq_minima_Mo_smaller}
(respectively, Equations~\eqref{eq_minima_Mb_smaller_D} and \eqref{eq_minima_Mo_smaller_D})
provide simple criteria that can be used to adjust the value of $L_{\rm o}$ (respectively, $D_{\rm o}$)
to minimise the condition number of the $\mathbf{B}$-preconditioned
Hessian matrix. From this perspective, we can use Corollary~\ref{co_minima_Mb_smaller} and
Corollary~\ref{co_minima_Mo_smaller} as the basis of
a method for {\it reconditioning} observation-error covariance matrices that account
for spatial correlations with parametric functions from the Mat\'ern family.
This would be complementary to existing methods for reconditioning
sample covariance matrices, for example, to represent inter-channel error correlations
in satellite observations  \citep{Weston_2014,Tabeart_2020}.
For more complex problems, where the assumptions of these corollaries are not perfectly
satisfied, we can still use criteria \eqref{eq_minima_Mb_smaller} and \eqref{eq_minima_Mo_smaller}
(or \eqref{eq_minima_Mb_smaller_D} and \eqref{eq_minima_Mo_smaller_D})
as a guideline for adjusting covariance parameters in $\mathbf{B}$ and $\mathbf{R}$
to improve the conditioning of the $\mathbf{B}$-preconditioned Hessian matrix.

\subsubsection{Condition number estimates with correlated and uncorrelated observation errors}
\label{sec_exp_condition}

In this section, we compare the condition number of $\mathbf{S}$ for different values of the
correlation parameter pairs $(M_{\rm o}, D_{\rm o})$ and $(M_{\rm b}, D_{\rm b})$. The condition number $\kappa(\mathbf{S})$
is computed using the (exact) analytical expression of the eigenvalues of $\mathbf{S}$ given in Theorem~\ref{th_eig_S}.
In addition, we compute the exact `optimal' parameter pairs ({\it i.e.}, those that minimise the condition number)
and compare them with those predicted by the optimality criteria in Corollary~\ref{co_minima_Mb_smaller} and
Corollary~\ref{co_minima_Mo_smaller}.
As this theorem applies
to the matrix $\Shat$, and not $\mathbf{S}$, these optimality criteria are only exact when there is a direct observation
at each grid point.

In presenting the results, we choose to normalize $\kappa(\mathbf{S})$ by $\kappa \big( \mathbf{S}_{\rm u} \big)$ where $\mathbf{S}_{\rm u}$
is given by Equation~\eqref{eq_def_S} with \mbox{$\mathbf{R} = \sigma_{\rm o}^2 \mathbf{I}_m$};
{\it i.e.}, with observation-error correlations neglected. An analytical expression
for the eigenvalues of $\mathbf{S}_{\rm u}$ can be derived directly from Equation~\eqref{eq_eig_S} of Theorem~\ref{th_eig_S}
by setting \mbox{$M_{\rm o}=0$} (no diffusion) and \mbox{$\nu_{\rm o}\widetilde{L}_{\rm o} = \gamma^2/h_{\rm o} = 1$}
(exact normalisation):
\begin{equation}
\forall i \in \llbracket 0, n-1 \rrbracket, \quad \lambda_i\big( \mathbf{S}_{\rm u} \big) = \begin{cases}
1 + \alpha_{\rm u} \displaystyle \sum_{r=0}^{\zeta-1}\limits\left[ 1+4\widetilde{L}_{\rm b}^2
\sin^2 \! \left(\pi\frac{\displaystyle i+rm}{\displaystyle \zeta m}\right)\right]^{-M_{\rm b}} &\text{ if } i \in \llbracket0,m-1\rrbracket,\\
1 &\text{ otherwise, }\end{cases}
\end{equation}
where
\begin{equation}
\alpha_{\rm u} \, = \, \frac{ \sigma_{\rm b}^2 \nu_{\rm b} L_{\rm b}}{\sigma_{\rm o}^2 h_{\rm b}}.
\nonumber
\end{equation}
As we are considering the case where there are fewer observations than grid points (\mbox{$n>m$}), the minimum eigenvalue
of $\mathbf{S}_{\rm u}$ is one. The maximum eigenvalue is $\lambda_0(\mathbf{S}_{\rm u})$ as can be seen by noting
that the term in square brackets is larger (and hence its inverse is smaller) for all \mbox{$i>0$}.
Consequently, \emph{if \mbox{$n>m$}}, the condition number of $\mathbf{S}_{\rm u}$ is
\begin{equation}
\kappa(\mathbf{S}_{\rm u}) = 1 + \alpha_{\rm u}  \sum_{r=0}^{\zeta-1}\limits\left[ 1+4\widetilde{L}_{\rm b}^2
\sin^2 \! \left(\pi\frac{r}{\zeta}\right)\right]^{-M_{\rm b}}.
\label{eq_kappa_Su}
\end{equation}
Note that the sum in Equation~\eqref{eq_kappa_Su} is larger than one and approximately equal to one for parameter values
of interest; {\it i.e.}, for \mbox{$M_{\rm b} \geq 2$} and \mbox{$\widetilde{L}_{\rm b}\geq 1$}, its maximum is less than 1.04.
The condition number of $\mathbf{S}_{\rm u}$ is thus dominated by $\alpha_{\rm u}$.

We denote $\chi$ the ratio of condition numbers:
\begin{equation}
    \chi=\frac{\kappa(\mathbf{S})}{\kappa(\mathbf{S}_{\rm u})}.
    \label{eq_def_ksi}
\end{equation}
As $M_{\rm o}$ and $D_{\rm o}$ have no effect on $\kappa(\mathbf{S}_{\rm u})$,
variations of $\chi$ with respect to these parameters will reflect variations of $\kappa(\mathbf{S})$.
If \mbox{$\chi<1$} then accounting
for correlated observation error will improve
the conditioning of $\mathbf{S}$ and thus we can expect the convergence rate of CG to be improved.
Conversely, if \mbox{$\chi>1$} then accounting for
correlated observation error will degrade the conditioning of $\mathbf{S}$
and we can expect the convergence rate of CG to be degraded.

In the following, we will compute the condition numbers as a function of the
Daley length-scales defined in Equation~\eqref{eq:def_D_recall}.
Furthermore, since we are mainly interested in the sensitivity of the condition number to the correlation model
parameters, we will assume that the background- and observation-error variances are equal
(\mbox{$\sigma_{\rm b}^2 / \sigma_{\rm o}^2 = 1$}). We consider a domain of length $2000$~km, composed of \mbox{$n=500$} points
that are equally spaced every \mbox{$h_{\rm b}=4$}~km. We assume that a direct observation is available every other
grid point (\mbox{$\zeta =2$}, \mbox{$m=250$}, \mbox{$h_{\rm o} = 8$}~km).

\begin{figure}[hbt]
\centering
\includegraphics[width=\textwidth]{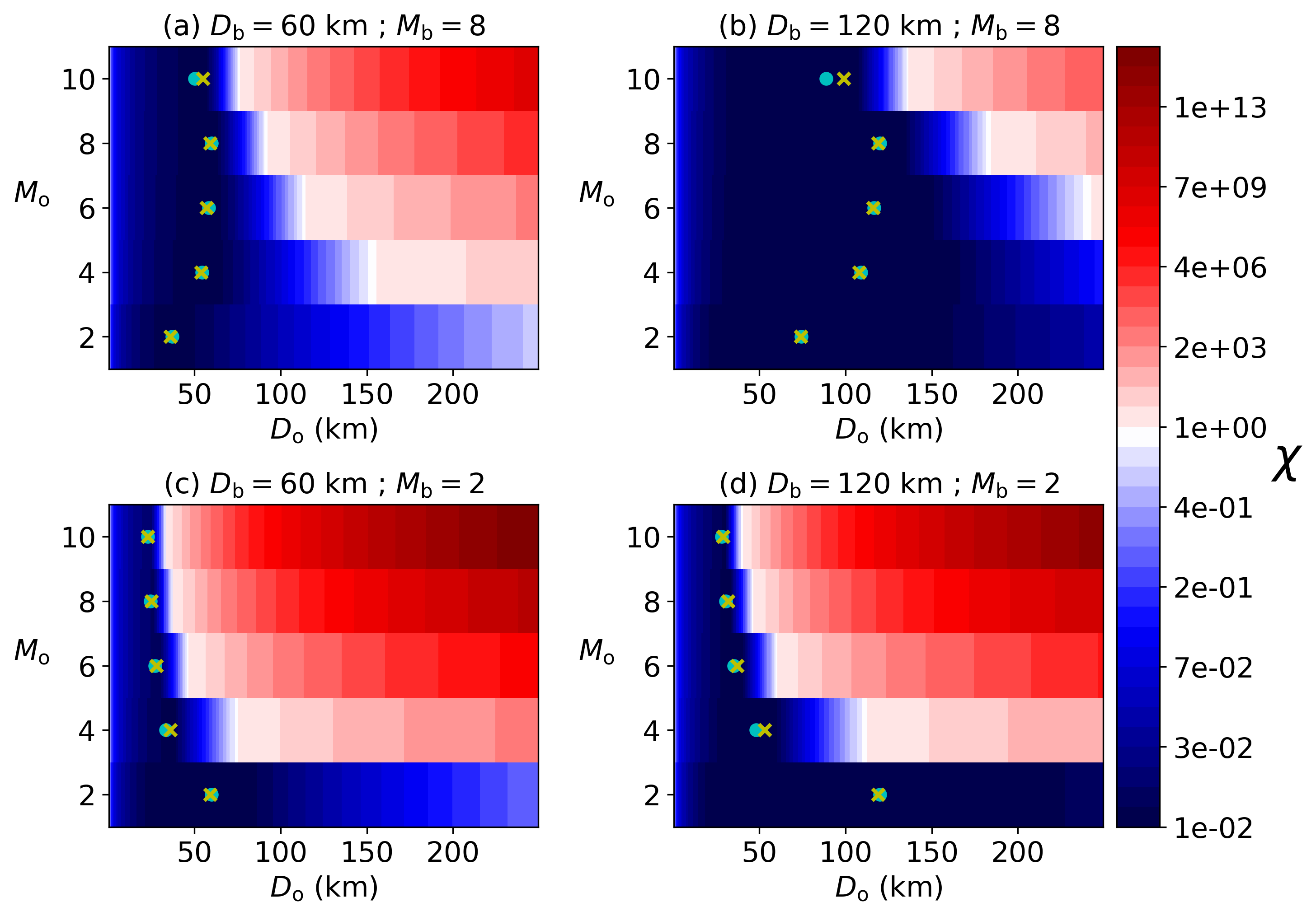}
\caption{The ratio $\chi$ (Equation~\eqref{eq_def_ksi}) is plotted for a fixed parameter pair ($M_{\rm b}, D_{\rm b}$) per panel
(indicated in the title),
and values of $M_{\rm o}$ and $D_{\rm o}$ that vary along the axes. The cyan circles mark the minima predicted
by Corollary~\ref{co_minima_Mb_smaller}  and Corollary~\ref{co_minima_Mo_smaller}.
The yellow crosses mark the true minima. Note that the colour palette uses
a logarithmic scale with a different range below and above \mbox{$\chi = 1$}.}
\label{fig_chi_maps}
\end{figure}

 Figure~\ref{fig_chi_maps} shows $\chi$ as a function of $D_{\rm o}$ (abscissa)
 and $M_{\rm o}$ (ordinate) for different parameter pairs ($M_{\rm b}, D_{\rm b}$) indicated in the title of each panel.
 The zones in blue (respectively, red) correspond to parameter pairs ($M_{\rm o}, D_{\rm o}$) that improve (respectively, degrade)
 the condition number.
 When \mbox{$M_{\rm o} \leq M_{\rm b}$} and \mbox{$D_{\rm o} \leq D_{\rm b}$}, the conditioning is systematically improved.
An improvement is also possible when \mbox{$D_{\rm o} \geq D_{\rm b}$} if $M_{\rm o}$ is small enough.
However, when $M_{\rm o}$ becomes too large compared to $M_{\rm b}$ or when $D_{\rm o}$ becomes too large compared to $D_{\rm b}$,
the conditioning
is degraded and can become significantly degraded even for modest changes in the parameter values. For example,
when \mbox{$M_{\rm b} = 8$} and \mbox{$D_{\rm b} = 60$}~km (Figure~\ref{fig_chi_maps}a), and \mbox{$M_{\rm o} = 10$},
$\chi$ (and thus $\kappa(\mathbf{S})$) increases by several orders of magnitude when the value of $D_{\rm o}$
is increased to less than double $D_{\rm b} $. When $D_{\rm o}$ is approximately four times $D_{\rm b}$,
$\chi$ reaches $10^{10}$ (top right corner of Figure~\ref{fig_chi_maps}a).
In these cases, we can expect the convergence of CG to be significantly affected, as
illustrated later in Section~\ref{sec_numerical_experiments}.

 As predicted by Corollary~\ref{co_minima_Mb_smaller}  and Corollary~\ref{co_minima_Mo_smaller},
 if $D_{\rm b}$, $M_{\rm b}$ and $M_{\rm o}$ are fixed, then $\kappa(\mathbf{S})$ admits a unique minimum.
 When \mbox{$M_{\rm o}>M_{\rm b}$}, the minima predicted by Corollary \ref{co_minima_Mb_smaller}
 are visibly shifted towards lower values of $D_{\rm o}$  (cf. circles and crosses in Figure~\ref{fig_chi_maps}).
 This shift corresponds to an increase of the condition number of up to $5\%$. As the variations of the condition number studied here
 cover a range of more than 10 orders of magnitude, this increase is acceptable. When \mbox{$M_{\rm o}<M_{\rm b}$},
 there is no significant difference in the position of the minima
 predicted by Corollary~\ref{co_minima_Mo_smaller}  and the exact minima; there is an increase of the condition number
 between the predicted minima and
 exact minima that is smaller than $0.1\%$.
 The pattern is similar with each fixed settings
 for ($M_{\rm b}, D_{\rm b}$) (\textit{i.e.}, each panel of Figure~\ref{fig_chi_maps}).
 If $D_{\rm b}$ increases (decreases) then the `optimal'
 values of $D_{\rm o}$ are shifted to the right (left) towards larger (smaller) values of $D_{\rm o}$
 (cf. Figure~\ref{fig_chi_maps}a and b, or Figure~\ref{fig_chi_maps}c and d).

\section{Numerical experiments}
\label{sec_numerical_experiments}
In this section, we illustrate how different covariance parameter settings influence the performance of
the CG minimisation. We evaluate the convergence rate in relation to the condition number diagnostic
$\chi$ presented in Section~\ref{sec_exp_condition} (see Figure~\ref{fig_chi_maps}) and the results of
Corollary~\ref{co_minima_Mb_smaller} and Corollary~\ref{co_minima_Mo_smaller}

\subsection{Experimental framework}
\label{sec_exp_framework}

As in Section~\ref{sec_exp_condition}, we define our baseline 1D-Var experiment as one in which the
background- and observation-error variances are taken to be equal, with their actual
values set to one unit (\mbox{$\sigma_{\rm b}^2 = \sigma_{\rm o}^2 = 1$}).
The domain is periodic with length $2000$~km and there are \mbox{$n=500$} grid points
\mbox{($h_{\rm b} = 4$~km)}.
We define $\mathbf{H}$ as a selection operator where
direct observations are assumed to be available at every other grid point ($m=250$, $h_{\rm o} = 8$~km).

We consider different `scenarios' where observations with perfectly known error correlations
are assimilated together with a background state that also has perfectly known error correlations.
We start by defining a `true state', $\mathbf{x}_{\rm t}$, which is specified by an analytical function.
As $\mathbf{H}$ is linear in our framework, the actual choice of the true state has no impact on the
performance of the CG minimisation as it is subtracts out from the innovation vector.
The background state and observations are then generated by adding to the true state,
unbiased random perturbations of covariance matrices $\mathbf{B}$ and $\mathbf{R}$, respectively. Specifically,
let $\widehat{\mathbf{\epsilon}}_{\rm b}$ and $\widehat{\mathbf{\epsilon}}_{\rm o}$
be normally-distributed vectors with zero mean and covariance matrix
equal to the identity matrix.
We can generate many realisations of
\mbox{$\widehat{\mathbf{\epsilon}}_{\rm b} \sim N(\mathbf{0}, \mathbf{I}_n)$}
and \mbox{$\widehat{\mathbf{\epsilon}}_{\rm o} \sim N(\mathbf{0}, \mathbf{I}_m)$} using a random number generator.
Then, using the factored covariance matrices
\mbox{$\mathbf{B} =\mathbf{U}\mathbf{U}^\transpose$} and \mbox{$\mathbf{R} =\mathbf{V}\mathbf{V}^\transpose$}, we define
\begin{align}
\mathbf{x}_{\rm b} &=  \mathbf{x}_{\rm t} + \mathbf{\epsilon}_{\rm b},
\\
\mathbf{y}_{\rm o} & =  \mathbf{H}\mathbf{x}_{\rm t} + \mathbf{\epsilon}_{\rm o},
\label{eq_ranerrors}
 \end{align}
where \mbox{$\mathbf{\epsilon}_{\rm b} = \mathbf{U} \widehat{\mathbf{\epsilon}}_{\rm b}$}
and \mbox{$\mathbf{\epsilon}_{\rm o} = \mathbf{V} \widehat{\mathbf{\epsilon}}_{\rm o}$}. By construction,
$\mathbb{E}[\mathbf{\epsilon}_{\rm b} \mathbf{\epsilon}_{\rm b}^\transpose] = \mathbf{B}$
and $\mathbb{E}[\mathbf{\epsilon}_{\rm o}\mathbf{\epsilon}_{\rm o}^\transpose] = \mathbf{R}$
where $\mathbb{E}[\; ]$ is the expectation operator.

To assess the convergence rate of the CG algorithm at each iteration, it is common
to monitor the reduction of the cost function or, equivalently, the reduction of
the $\mathbf{A}$-norm of the analysis (solution) error.
However, if we want to compare the convergence rate of CG with different $\mathbf{R}$, the $\mathbf{A}$-norm
is not appropriate since it depends on $\mathbf{R}$ and thus
does not represent the same quantity in all cases.
Since we are working with an idealized system for which the true state $\mathbf{x}_{\rm t}$ is known,
we have access to alternative metrics that would not be available in a realistic system.

At the $\ell$-th iteration of the CG algorithm,
an increment $\delta \mathbf{x}_{\ell}$ is produced. We can deduce from this
increment the analysis error that would result if
the CG algorithm was stopped at the $\ell$-th iteration:
\begin{equation}
\mathbf{\epsilon}_{\rm a}^{(\ell)} \, = \, \mathbf{x}_{\rm b} + \delta \mathbf{x}_{\ell} - \mathbf{x}_{\rm t}.
\nonumber
\end{equation}
In each experiment, there is a random component in the generation of the background and observations, which will
affect $\mathbf{\epsilon}_{\rm a}^{(\ell)} $. By performing multiple experiments with different right-hand sides
($\mathbf{b}$ in Equation~\eqref{eq_preconditioned_system}),
we can obtain multiple realizations of $\mathbf{\epsilon}_{\rm a}^{(\ell)}$ from which analysis-error statistics can
be deduced. In particular, we can estimate at each iteration the total analysis-error variance or,
equivalently, the trace of the analysis-error covariance matrix. This is the quantity that is minimised
explicitly in a statistical analysis based on the Best Linear Unbiased Estimator (BLUE).
It is well known that, when the constraints are linear and when the
background and observation errors are normally distributed, the
minimising solution of the cost function of variational data assimilation is equivalent to the BLUE
when both are formulated under the same assumptions \citep{Gelb_1974}.

As a diagnostic, we compute the square root of the average of the total analysis-error variance:
\begin{equation}
\sigma_{\rm a}^{(\ell)} \,  = \, \sqrt{\frac{1}{n}\mathbb{E}\left[
{\rm Tr}\left( (\mathbf{\epsilon}_{\rm a}^{(\ell)})(\mathbf{\epsilon}_{\rm a}^{(\ell)})^{\rm T}\right)\right]}
\, = \, \sqrt{\frac{1}{n}\mathbb{E}
\left[ (\mathbf{\epsilon}_{\rm a}^{(\ell)})^{\rm T}(\mathbf{\epsilon}_{\rm a}^{(\ell)})\right]}
\label{eq_sigl}
\end{equation}
where $\mathbb{E}$ denotes the expectation operator and ${\rm Tr}$ the trace operator. We approximate the
expectation operator as an average of 1000 realizations with random right-hand sides.

 This metric can be used to assess not only the convergence rate of the minimisation on which we focused in the previous
 sections, but also the accuracy of the solution at each iteration. We expect the solution of the
 minimisation at full convergence to be more accurate when the actual observation-error correlations are accounted for.
 If the condition number is reduced by a non-diagonal $\mathbf{R}$ (\textit{i.e.}, \mbox{$\chi<1$} as in the `blue zone' of
 Figure~\ref{fig_chi_maps}) then the minimisation should converge faster. In this situation,
 we can expect the solution to be more accurate no matter when the minimisation is stopped. On the other hand,
 if the condition number is increased by a non-diagonal $\mathbf{R}$ (\textit{i.e.}, \mbox{$\chi>1$} as in the `red zone' of
 Figure~\ref{fig_chi_maps}) then we can expect the convergence rate to be slower. In this situation, it is not clear
 whether a non-diagonal $\mathbf{R}$ is beneficial to the analysis or not,
 as there is a trade-off between the convergence rate and the expected accuracy at full convergence.
Monitoring the analysis error at each iteration allows us to visualize
 this trade-off as it indicates, at each iteration, how accurate the analysis would be (on average) if the convergence
 was stopped at this point.

  A natural choice of normalization for  $\sigma_{\rm a}^{(\ell)}$ is its initial value $\sigma_{\rm a}^{(0)}$,
  which is equal to $\sigma_{\rm b}$ in the experiments as the initial increment \mbox{$\delta \mathbf{x}_{0}$}
 is zero. The quantity \mbox{$\sigma_{\rm a}^{(\ell)}/\sigma_{\rm a}^{(0)}$} thus indicates the relative
 error reduction on each iteration of CG. We denote  $\sigma_{\rm a}^{\ast}$ the value of
 $\sigma_{\rm a}^{(\ell)}$ at full convergence of CG.
 If the specifications of $\mathbf{B}$ and $\mathbf{R}$ used to compute the analysis match the
 actual error statistics, this quantity  should become equal to its theoretical minimum,
 $\sigma_{\rm a}^{\rm opt}$, which can be computed directly from the trace of the theoretical analysis-error covariance matrix:
 \begin{equation}
  \sigma_{\rm a}^{\rm opt} = \sqrt{ \frac{1}{n}{\rm Tr}\left[\left(\mathbf{B}^{-1}
  + \mathbf{H}^{\rm T}\mathbf{R}^{-1}\mathbf{H} \right)^{-1} \right]}
  \label{eq_sigopt}
 \end{equation}
where $\mathbf{B}$ and $\mathbf{R}$ are the same as those used to generate the random errors in Equation~\eqref{eq_ranerrors}.

In the experiments, $\mathbf{R}$ denotes the `true' observation-error covariance matrix
used to generate the spatially-correlated
random errors that are added to the observations.
The matrix \mbox{$\widetilde{\mathbf{R}}_1 = \sigma_{\rm o}^2 \mathbf{I}_m$} is a diagonal
approximation where $\sigma_{\rm o}^2$ is the same constant variance used in the `true' $\mathbf{R}$.
This corresponds to the common case where spatially-correlated observation errors are ignored
in the weighting matrix in the cost function, which is inconsistent with the
statistical properties of the observations that are assimilated.
The third scenario also uses a diagonal matrix,
\mbox{$\widetilde{\mathbf{R}}_2 = \upsilon \, \sigma_{\rm o}^2 \mathbf{I}_m$},
but the variances are multiplied by an inflation factor ($\upsilon$) to mitigate the effect
of neglecting the error correlations.
This procedure is common in real-data assimilation
problems, to avoid overfitting observations at large spatial scales while
retaining a simple covariance matrix.
In practice, the inflation factor is usually estimated empirically.
In our experiments, we can determine the best-possible inflation factor by minimizing $\sigma_{\rm a}^{\ast}$
with respect to $\upsilon$. As $\sigma_{\rm a}^{\ast}$ behaves approximately as a convex function of $\upsilon$, this
can be achieved by computing $\sigma_{\rm a}^{\ast}$ for increasing values of $\upsilon$ until it stops decreasing
(\textit{i.e.,} until the observations are no longer overfit). This method cannot be applied in an operational
context as it requires access to the true state. Even with a performance metric that uses a proxy
for the true state, the cost of the procedure would be prohibitive as thousands of realisations of
$\sigma_{\rm a}^{\ast}$ are required. The experiments using $\widetilde{\mathbf{R}}_2$ thus
represent the best inflation can offer rather than what could be achieved in practice.

\subsection{Results}
\label{sec_exp_results}

In the first set of experiments, we consider the case where the background- and observation-error
correlation parameters are in the regime \mbox{$M_{\rm o} < M_{\rm b}$} and
\mbox{$D_{\rm o} < D_{\rm b}$}.
\begin{figure}[htb]
\centering
\captionsetup[subfigure]{justification=centering}
\begin{subfigure}[t]{0.38\textwidth}
\includegraphics[width=\textwidth]{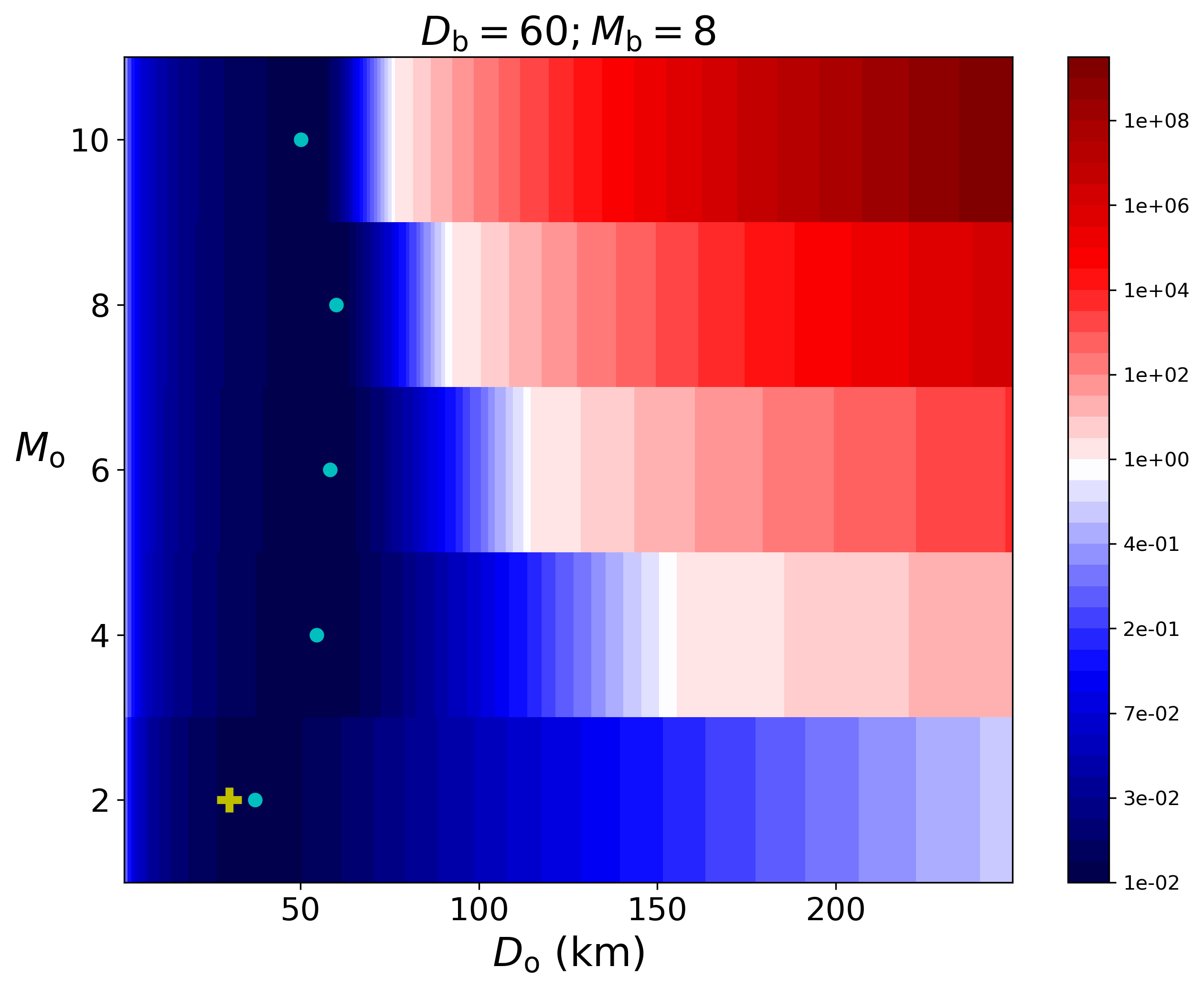}
\caption{$\chi$ for different covariance parameters}
\label{fig_chi_scenario_1}
\end{subfigure}\hfill
\begin{subfigure}[t]{0.60\textwidth}
\includegraphics[width=\textwidth]{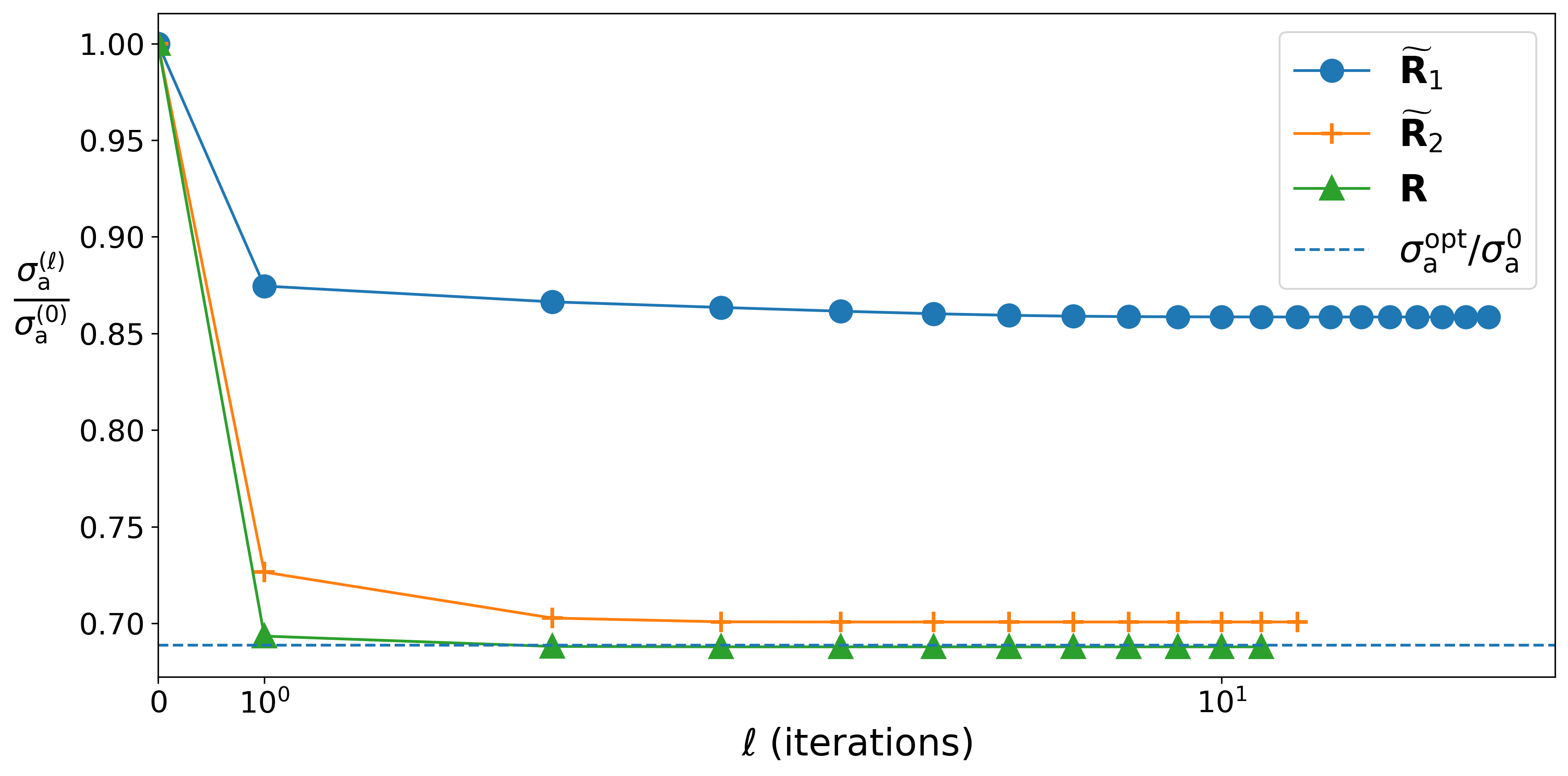}
\caption{CG minimisation}
\label{fig_scenario_1}
\end{subfigure}
\caption{(a) Same as Figure~\ref{fig_chi_maps} but with a plus sign added to indicate
the parameter pair ($M_{\rm o}, D_{\rm o}$) used for the 1D-Var experiments in panel~(b).
(b) $\sigma_{\rm a}^{(\ell)}/\sigma_{\rm a}^{(0)}$ (Equation~\eqref{eq_sigl}) as a function of CG iteration
count $\ell$ for three 1D-Var experiments with the same covariance parameters for $\mathbf{B}$
(\mbox{$\sigma_{\rm b}^2 = 1$, $M_{\rm b}=8$}, \mbox{$D_{\rm b}=60$}~km) but different
covariance parameters for $\mathbf{R}$: (1) $\mathbf{R}$ with the `true' correlation parameters
(\mbox{$\sigma_{\rm o}^2 = 1$, $M_{\rm o}=2$}, \mbox{$D_{\rm o}=30$}~km); (2)
a diagonal approximation,
\mbox{$\widetilde{\mathbf{R}}_1 = \sigma_{\rm o}^2 \mathbf{I}_m$} with \mbox{$\sigma_{\rm o}^2 = 1$};
(3) a diagonal approximation with inflated variances,
\mbox{$\widetilde{\mathbf{R}}_2 =  \upsilon \, \sigma_{\rm o}^2 \mathbf{I}_m$}
where \mbox{$\upsilon =10.5$} is an optimally-estimated inflation factor.
The theoretical minimum analysis-error ratio
$\sigma_{\rm a}^{\rm opt}/\sigma_{\rm a}^{(0)}$ (Equation~\eqref{eq_sigopt})
is marked by a horizontal dashed line.}
\label{fig_s1}
\end{figure}
The observation-error correlation parameters are set to \mbox{$M_{\rm o}=2$} and $D_{\rm o}=30$~km,
which corresponds to a SOAR function as used in \cite{Tabeart_2021}.
These values are roughly similar to those proposed by \cite{Guillet_2019}, where
the parameter settings were determined to provide a suitable fit of a diffusion-model to error
correlation estimates of certain satellite radiance observations \citep{Waller_2016}.
The background-error correlation parameters are set to \mbox{$M_{\rm b}=8$} and \mbox{$D_{\rm b}=60$}~km,
which makes the correlation function more Gaussian-like than that of $\mathbf{R}$.
The correlation length-scale of $\mathbf{B}$ is double the correlation length-scale of $\mathbf{R}$.
These are the same $\mathbf{B}$ parameters that were used in Figure~\ref{fig_chi_maps}.
With these parameters, we know that $\chi<1$ (Figure~\ref{fig_chi_scenario_1}),
which means that the condition number is reduced when observation-error correlations
are accounted for.

The average error-convergence curves from the 1D-Var experiments with $\mathbf{R}$, $\widetilde{\mathbf{R}}_1$
and $\widetilde{\mathbf{R}}_2$ are shown in Figure~\ref{fig_scenario_1}.
For each experiment, minimisations are performed in parallel for all 1000 realisations of
the random right-hand side and are stopped when the 2-norm of the residual normalized by
its initial value reaches $10^{-6}$. Convergence is achieved rather quickly,
taking about 20 iterations with $\widetilde{\mathbf{R}}_1$ and about 10 iterations with $\mathbf{R}$ and
$\widetilde{\mathbf{R}}_2$.

For the experiment with $\widetilde{\mathbf{R}}_1$, the analysis-error standard deviation is only reduced
by about $15\%$ at full convergence, compared to the theoretical limit of $32\%$.
In this set-up, the optimal variance inflation factor is approximately equal to $10.5$. Inflating the error variances significantly
improves the error reduction ($30\%$).
However, the theoretical minimum error cannot be reached, even with an inflation factor
that has been optimized for this specific experiment.
It is also important to remark that the experiment with inflated variances converges faster
than the experiment with the original variances. This is consistent with
Equation~\eqref{eq_kappa_Su}, which shows that, for large $\alpha_{\rm u}$,
the condition number of $\mathbf{S}_{\rm u}$ is approximately
inversely proportional to the observation-error
variance and is thus divided by 10.5 in this case. Best results are obtained with $\mathbf{R}$.
First, the convergence rate is the fastest of the three experiments. Second,
on each iteration, the solution is more accurate than the solutions from either
$\widetilde{\mathbf{R}}_1$ or $\widetilde{\mathbf{R}}_2$. At full convergence, the solution
attains the theoretical minimum error.

\begin{figure}[ht]
\centering
\captionsetup[subfigure]{justification=centering}
\begin{subfigure}[t]{0.38\textwidth}
\includegraphics[width=\textwidth]{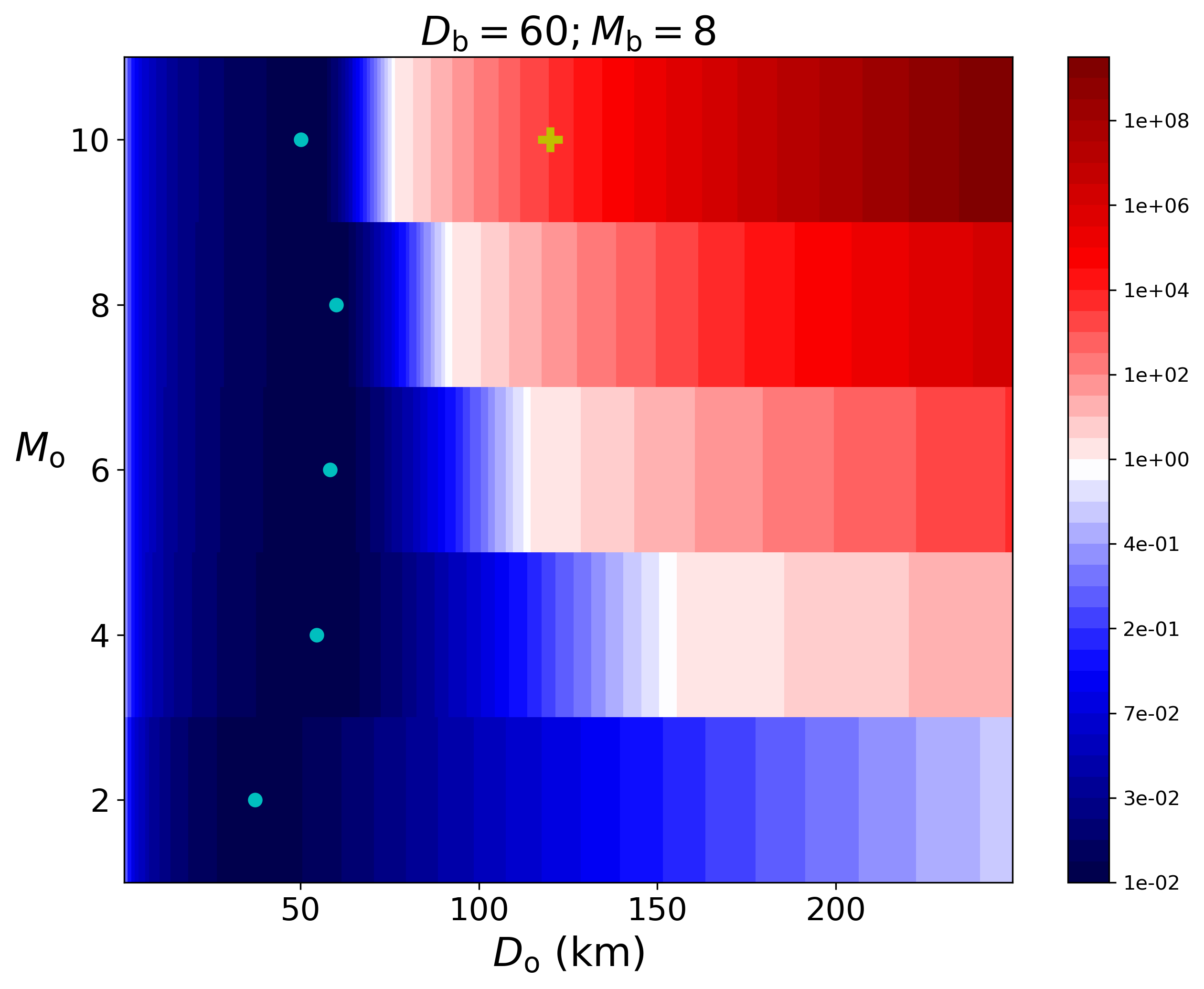}
\caption{$\chi$ for different covariance parameters}
\label{fig_chi_scenario_2}
\end{subfigure}\hfill
\begin{subfigure}[t]{0.60\textwidth}
\includegraphics[width=\textwidth]{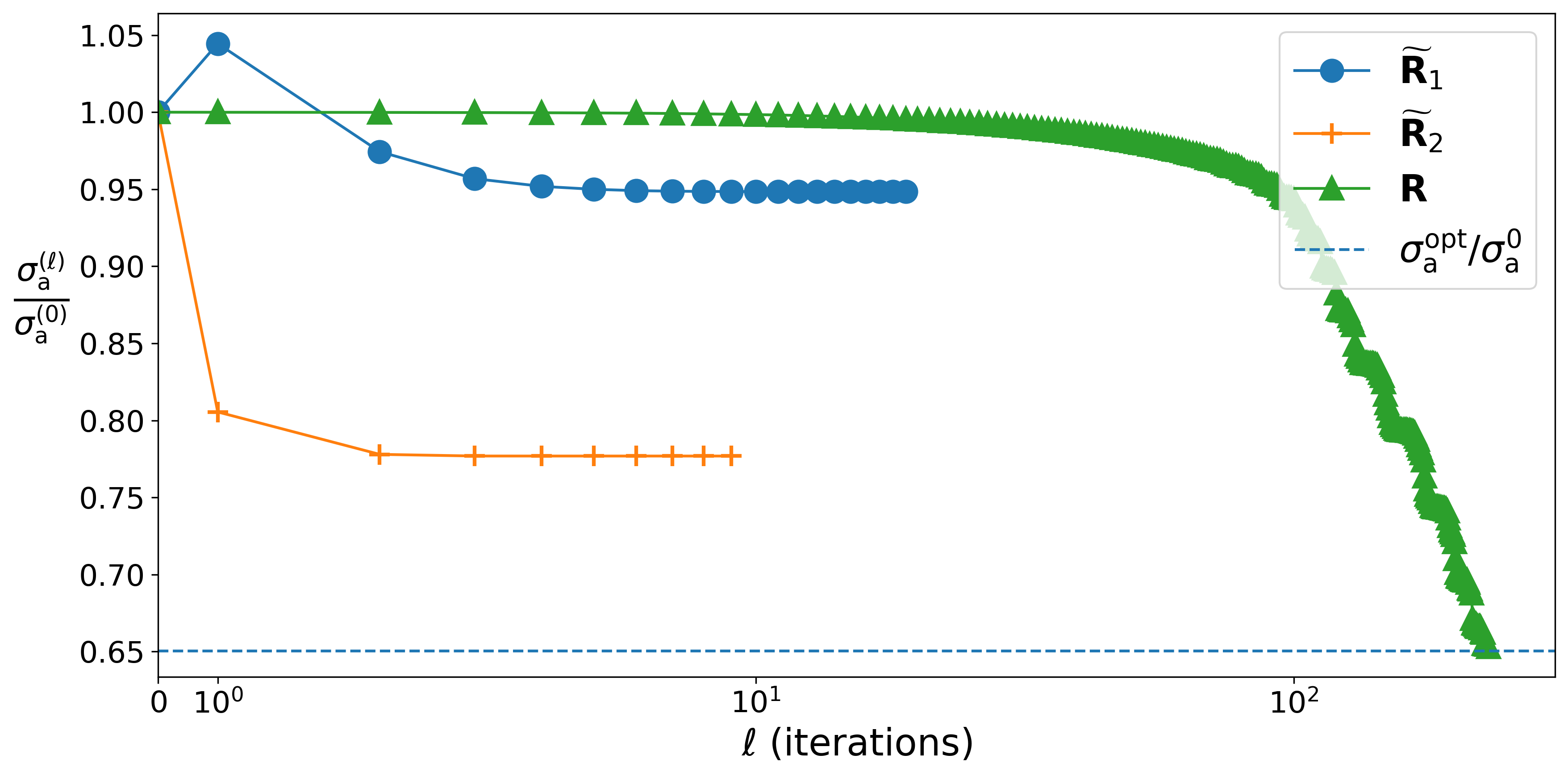}
\caption{CG minimisation}
\label{fig_scenario_2}
\end{subfigure}
\caption{Same as Figure~\ref{fig_s1} but with a different parameter pair
($M_{\rm o} = 10, D_{\rm o} = 120$~km), indicated by the plus sign in panel~(a).
The optimal inflation factor in panel~(b) is \mbox{$\upsilon = 17$}.}
 \label{fig_s2}
 \end{figure}
We now consider a case where the parameter values of the observation-error correlations suggest
that the convergence rate of the minimisation will be degraded (\textit{i.e.,} $\chi>1$) .
Using Figure~\ref{fig_chi_maps}, we can select parameter values
that will increase the condition number. In particular,
we set \mbox{$M_{\rm o}=10$} and \mbox{$D_{\rm o}=120$~km}, while
keeping the background-error parameter values unchanged. In this set-up, the observation-error length-scale $D_{\rm o}$
is double the background-error correlation length-scale.
With this set of parameter values, the condition number is increased by a factor of $10^4$. In this set-up,  the theoretical
minimum error $\sigma_{a}^{\rm opt}/\sigma_{\rm b}$ is lower than in the previous experiment: $65\%$ instead of $68\%$.
This decrease means that observations with highly correlated errors `complement' the background better than those of
the previous experiment.

As shown in Figure~\ref{fig_scenario_2}, while the minimisation with $\mathbf{R}$ does reach the theoretical minimum,
it requires about 200 iterations to converge.
If the minimisation was terminated in its early iterations ($<50$)
then the analysis would be hardly better than that of the background and not as accurate as the solutions
from either the $\widetilde{\mathbf{R}}_1$ or $\widetilde{\mathbf{R}}_2$ experiments.
In this case, it would be clearly detrimental to account for the observation-error correlations instead
of ignoring them.

It is interesting to note that the convergence curve for $\widetilde{\mathbf{R}}_1$
in Figure~\ref{fig_scenario_2} differs
from the one in Figure~\ref{fig_scenario_1} even though the Hessian matrix $\mathbf{S}$ for this case
is the same in both experiments. On the other hand, the assimilated observations are different in each
experiment as they have different correlated errors. The difference in the convergence curves
in the two experiments thus highlights the role of the right-hand side (which depends on the observations)
of the system on the convergence rate of CG. While this is an important issue, we have not attempted
to address it in this article.

The experiment $\widetilde{\mathbf{R}}_1$  results in an error reduction of only $5\%$, compared to the theoretical
minimum of $35\%$. Moreover, the error reduction is non-monotonic, which is symptomatic of a more concerning issue:
as $\widetilde{\mathbf{R}}_1$ is an approximation of the actual error covariances,
there is no guarantee that the analysis will be more accurate than the background
(even at full convergence). Repeating the experiment with a larger $\sigma_{\rm o}$ (in both
$\mathbf{R}$ and $\widetilde{\mathbf{R}}_1$)
than $\sigma_{\rm b}$ actually results in a monotonically increasing error (not shown). In this case, the analysis overfits
the observations due to the neglected correlations in $\widetilde{\mathbf{R}}_1$. This problem is
exacerbated when the observations are less accurate than the background.

In the current scenario, the optimal variance inflation factor is approximately equal to 17, and the experiment with
$\widetilde{\mathbf{R}}_2$ gives the best results when using a modest number of iterations ($<150$).
It produces a similar, rapid convergence rate as in the previous scenario (Figure~\ref{fig_scenario_1}) and
produces an accurate analysis, with a $23\%$ error reduction compared to the theoretical minimum error
reduction of $35\%$.

Rather than adopting a diagonal approximation, an alternative approach would
be to adjust the correlation parameters to accelerate the convergence rate while trying to retain
the salient features of the original correlation function, which in principle should correspond to our
best available estimate of the actual correlation function.
As discussed in Section~\ref{sec_structure_exact_cond}, certain adjustments to the parameter settings
can have a significant impact on the condition number, while inducing
relatively minor changes to the correlation function and hence to $\sigma_{\rm a}^{\ast}$.
Moreover, previous studies have shown that even an approximate correlation structure in $\mathbf{R}$
can yield higher quality analyses
than ones obtained with wrongly assuming
uncorrelated observation errors (\textit{e.g.}, \cite{Stewart_2013}).

\begin{figure}[ht]
\centering
\captionsetup[subfigure]{justification=centering}
\begin{subfigure}[t]{0.38\textwidth}
\includegraphics[width=\textwidth]{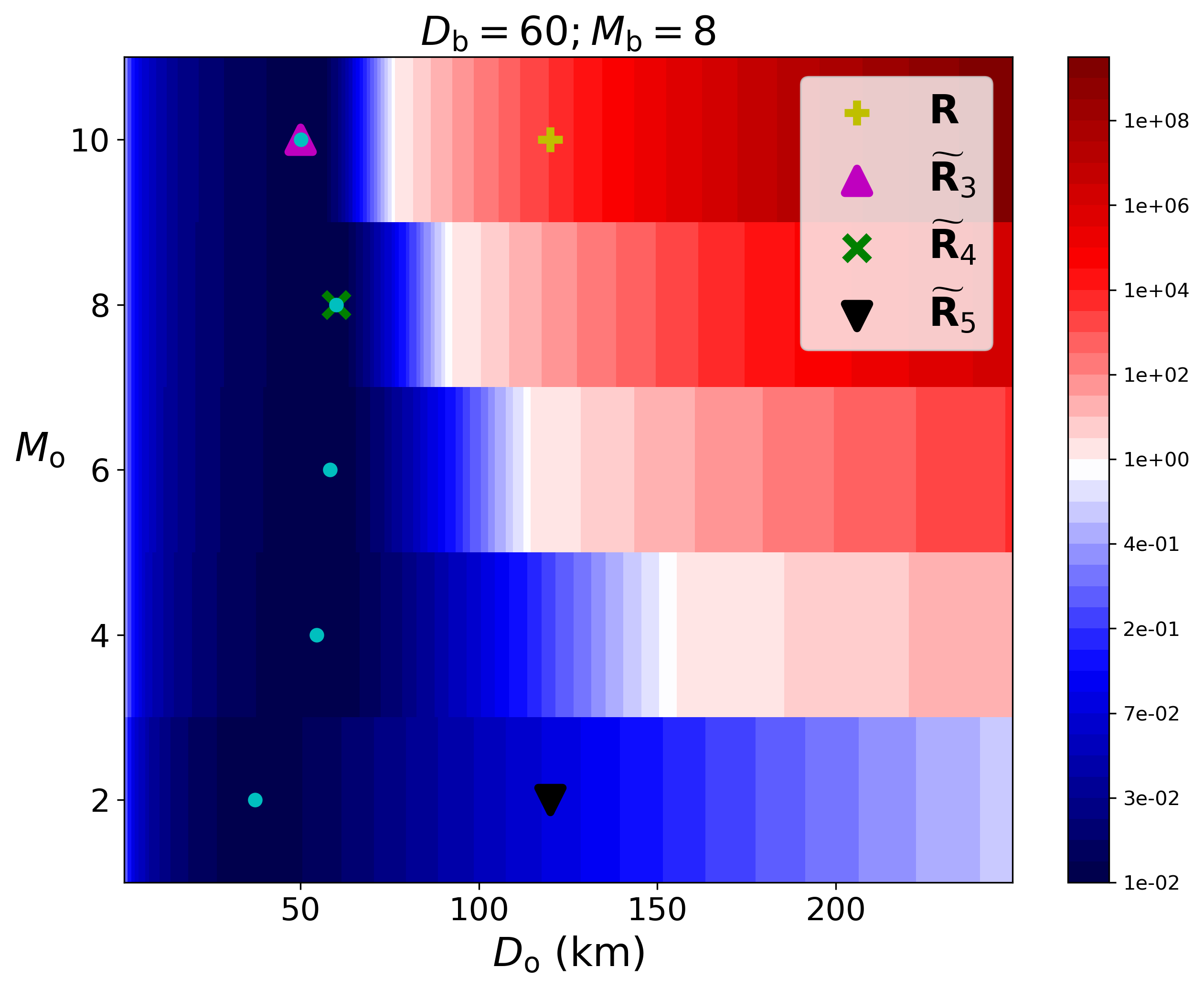}
\caption{$\chi$ for different covariance parameters}
\label{fig_chi_scenario_3}
\end{subfigure}\hfill
\begin{subfigure}[t]{0.60\textwidth}
\includegraphics[width=\textwidth]{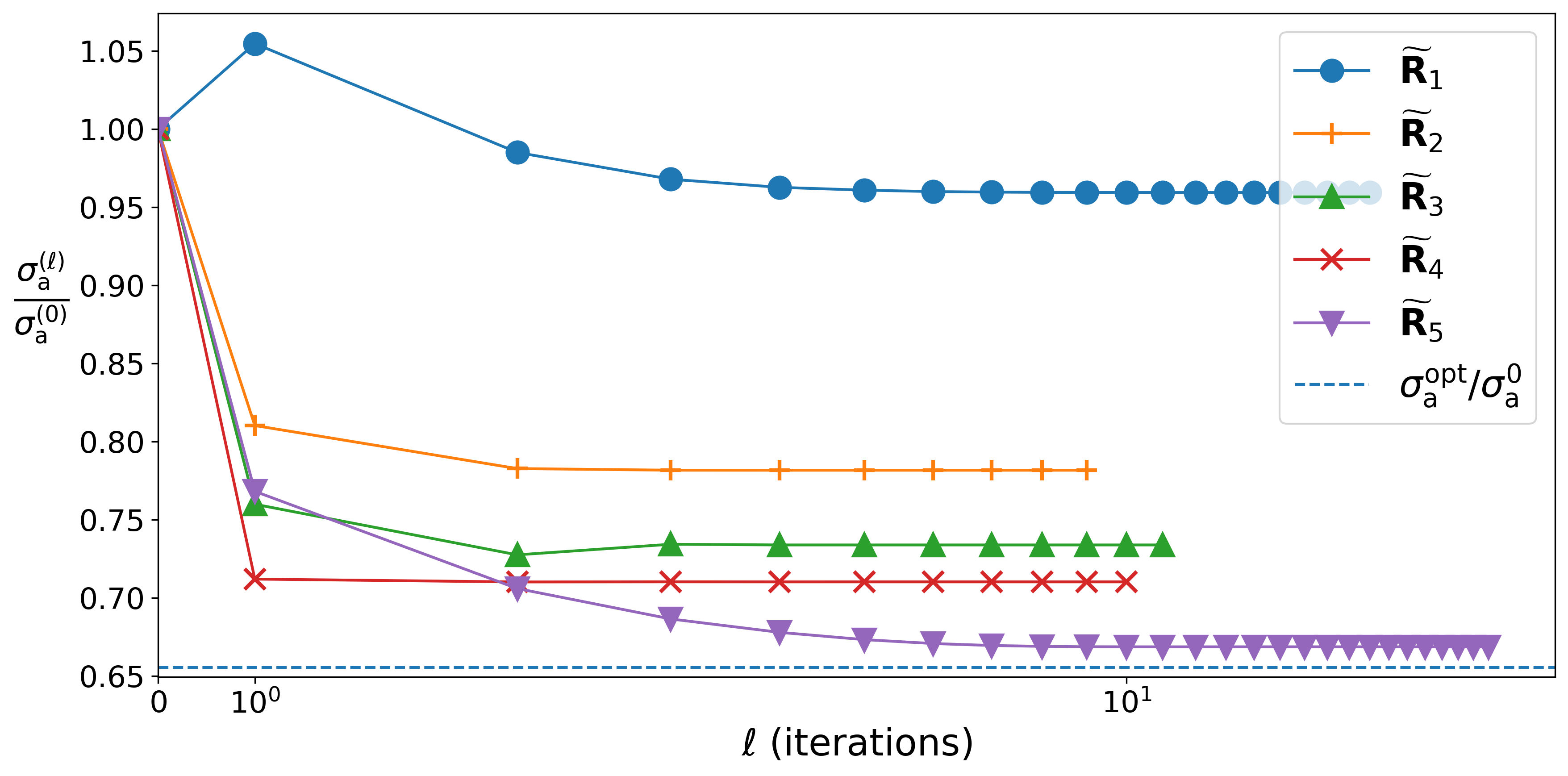}
\caption{CG minimisation}
\label{fig_scenario_3}
\end{subfigure}
\caption{Same as Figure~\ref{fig_s2} but with different parameter pairs
($M_{\rm o}=10$, $D_{\rm o} = 50$~km), ($M_{\rm o}=8$, $D_{\rm o} = 60$~km)
and ($M_{\rm o}=2$, $D_{\rm o} = 120$~km)
for $\mathbf{R}$ in panel~(b)
(experiments labelled $\widetilde{\mathbf{R}}_3$, $\widetilde{\mathbf{R}}_4$
and $\widetilde{\mathbf{R}}_5$, respectively)
compared to the parameter pair ($M_{\rm o}=10$, $D_{\rm o} = 120$~km)
used to generate
the observations. The different pairs are indicated by the different coloured symbols
in panel~(a).
The optimal inflation factor in panel~(b) is \mbox{$\upsilon = 17$ }.}
\label{fig_s3}
\end{figure}

Corollary~\ref{co_minima_Mb_smaller}, Corollary~\ref{co_minima_Mo_smaller}
and Figure~\ref{fig_chi_maps} can be used as a guideline
to find suitable parameters. In this scenario, we would like to pick values of $M_{\rm o}$ and $D_{\rm o}$ that are
`close enough' to our target parameter values of \mbox{$M_{\rm o}=10$} and \mbox{$D_{\rm o}=120$~km}
so as not to increase the analysis error by too much at full convergence ($\sigma_{\rm a}^{\ast}$),
but which produce a much smaller condition number for $\mathbf{S}$. In Figure~\ref{fig_s3}, we consider three additional
experiments
(labelled $\widetilde{\mathbf{R}}_3$, $\widetilde{\mathbf{R}}_4$, and $\widetilde{\mathbf{R}}_5$), which all use
correlation models that are approximate compared to the actual one used to generate the
observation error (\mbox{$M_{\rm o}=10$} and
\mbox{$D_{\rm o}=80$~km}) but which lead to improved convergence rates.
Figure~\ref{fig_s3} shows that
all three experiments outperform the diagonal $\mathbf{R}$ experiments
$\widetilde{\mathbf{R}}_1$ and $\widetilde{\mathbf{R}}_2$ at every iteration. We now discuss
the choice of the parameter values for these experiments in relation to
Corollary~\ref{co_minima_Mb_smaller}, Corollary~\ref{co_minima_Mo_smaller} and Figure~\ref{fig_chi_maps}.

In the situation where using the accurate correlation model would degrade the condition number,
for given values of $M_{\rm o}$, $M_{\rm b}$ and $D_{\rm b}$, we can modify $D_{\rm o}$ to approximate
the theoretical minimum condition number predicted by Corollary~\ref{co_minima_Mb_smaller} or Corollary~\ref{co_minima_Mo_smaller}.
The experiment with $\widetilde{\mathbf{R}}_3$ corresponds to the `extreme' case where $D_{\rm o}$ is modified
using Equation~\eqref{eq_minima_Mb_smaller_D} (\mbox{$M_{\rm o} > M_{\rm b}$} in this experiment)
so its value coincides exactly with the minimum.
To do so, we retain the true value of \mbox{$M_{\rm o}=10$} but use an approximate value of \mbox{$D_{\rm o} = 50$~km}
to compute the analysis, instead of 120~km that was used to generate the correlated observation errors.
Decreasing $D_{\rm o}$ to this value
reduces the condition number by a factor of $10^{6}$. With these new parameters, \mbox{$\chi=10^{-2}$},
and the condition number
obtained with $\widetilde{\mathbf{R}}_3$ is lower than the one obtained with either $\widetilde{\mathbf{R}}_1$ or
$\widetilde{\mathbf{R}}_2$ (the condition number with $\widetilde{\mathbf{R}}_2$ is only 17 times lower than the
condition number with $\widetilde{\mathbf{R}}_1$).
Figure~\ref{fig_chi_scenario_3} shows that the experiment with $\widetilde{\mathbf{R}}_3$
outperforms both diagonal approximations at every iteration.
With this modified value of $D_{\rm o}$, the error reduction is $27\%$ compared to
$35\%$ with the actual value of $D_{\rm o}$, but allows a much faster convergence.

Another possibility is to modify $M_{\rm o}$, so that a smaller modification on $D_{\rm o}$ is required to approximate
a minimum condition number predicted by Corollary~\ref{co_minima_Mb_smaller} or Corollary~\ref{co_minima_Mo_smaller}.
When \mbox{$M_{\rm b} = 8$} and \mbox{$D_{\rm b}=60$~km}, Equation~\eqref{eq_minima_Mb_smaller_D} associated with Corollary~\ref{co_minima_Mb_smaller}
predicts a minimum with the parameter pairs (\mbox{$M_{\rm o}=10, D_{\rm o} = 50$~km})
and (\mbox{$M_{\rm o}=8, D_{\rm o} = 60$~km}). It is thus possible to reach
a minimum condition number with a smaller decrease of $D_{\rm o}$ if $M_{\rm o}$ is reduced from
from 10 to 8. The experiment with $\widetilde{\mathbf{R}}_4$ uses \mbox{$M_{\rm o} = M_{\rm b}=8$} and
\mbox{$D_{\rm b}=D_{\rm o}= 60$~km} (although $\mathbf{B}$ and $\mathbf{R}$
are not equal as there are less observations than grid points). The condition number with $\widetilde{\mathbf{R}}_4$ is slightly
lower than with $\widetilde{\mathbf{R}}_3$ (and thus also lower than with $\widetilde{\mathbf{R}}_1$ and $\widetilde{\mathbf{R}}_2$).
This correlation model allows a slightly better error reduction than with $\widetilde{\mathbf{R}}_3$ ($30\%$ compared to $27\%$),
while also converging slightly faster.
 As for the experiment with $\widetilde{\mathbf{R}}_3$, intermediate values of $D_{\rm o}$ which approach
 the local minima while staying closer to the true parameters might offer a better compromise between
 a fast convergence and good error reduction at full convergence.

In the two previous experiments, the parameter values were chosen in order to reach a local condition number minimum predicted
by Corollary~\ref{co_minima_Mb_smaller} or Corollary~\ref{co_minima_Mo_smaller}. Another strategy, which does not rely
on these corollaries, is to use
Figure~\ref{fig_chi_maps} to select parameter pairs that result in lower condition numbers.
If the actual parameter values have a severe impact on convergence as in Figure~\ref{fig_s2} then Figure~\ref{fig_chi_maps}
suggests that reducing $M_{\rm o}$ and/or $D_{\rm o}$ can reduce the condition number.
In particular, if \mbox{$M_{\rm o}<M_{\rm b}$} and \mbox{$D_{\rm o}<D_{\rm b}$} then \mbox{$\chi<1$.}
Generally speaking, lower values of $D_{\rm o}$ or $M_{\rm o}$, relative to the corresponding values of
$D_{\rm b}$ and $M_{\rm b}$, reduce the risk of the condition number being increased compared to the condition
number with a diagonal $\mathbf{R}$ (\textit{i.e.}, of being in the red area of
Figure~\ref{fig_chi_maps} where $\chi>1$). The parameters $D_{\rm o}$ and $M_{\rm o}$ can be reduced
progressively through trial-and-error to determine a convergence rate at least as good as the one
obtained with $\widetilde{\mathbf{R}}_1$.
For example, in the experiment with $\widetilde{\mathbf{R}}_5$, we set \mbox{$M_{\rm o}=2$}
instead of $10$ while keeping the correct value of \mbox{$D_{\rm o}=120$~km}.
With $\widetilde{\mathbf{R}}_5$, the condition number is approximately $20$ times smaller
than with $\widetilde{\mathbf{R}}_1$ (and thus slightly smaller than with $\widetilde{\mathbf{R}}_2$).
The experiment with $\widetilde{\mathbf{R}}_5$ has similar convergence rate
to the experiments with $\widetilde{\mathbf{R}}_3$ or $\widetilde{\mathbf{R}}_4$,
but achieves a better error reduction ($33\%$), which is close to that of the theoretical minimum ($35\%$).

\section{Summary and conclusions}
\label{sec_conclusions}

Data assimilation concerns the problem of determining the optimal state of a system given observations,
a background (prior) estimate of the state, and constraints that link the system state to the observations.
Mathematically, the problem can be cast as one of nonlinear weighted least-squares.
The technique of variational data assimilation, which is commonly used in atmospheric and ocean applications,
seeks an approximate solution by using a Truncated Gauss-Newton (TGN) algorithm to minimise the cost function
of the nonlinear weighted least-squares problem. The TGN algorithm approximates the nonlinear problem by a connected
sequence of linear sub-problems where each sub-problem is solved using a conjugate gradient (CG) algorithm.

In this article, we have studied the convergence properties of the CG algorithm with respect to parameter
specifications in the background-error covariance matrix ($\mathbf{B}$) and observation-error
covariance matrix ($\mathbf{R}$) whose respective inverse matrices are used to define the weights
for the background and observations in the cost function. In line with common practice in variational
data assimilation, we considered a CG algorithm that uses $\mathbf{B}$ as a preconditioner, which we referred
to as the $\mathbf{B}$-Preconditioned Conjugate Gradient (B-PCG) algorithm. Our results have shown that
the convergence rate of B-PCG (and thus of the whole TGN minimisation) is highly sensitive to
parameters controlling the shape of typical covariance functions used to model $\mathbf{B}$ and $\mathbf{R}$.
In particular, the number of CG iterations needed to reach a given tolerance can change by a few orders of magnitude
depending on the relative parameter specifications between $\mathbf{B}$ and $\mathbf{R}$. This underlines
the importance of including convergence impact as an additional constraint when adjusting covariance model
parameters to fit covariances estimated from statistics.

We began by recalling the general upper bounds on the condition number derived by \cite{Tabeart_2021}.
These upper bounds do not depend on the type of covariance matrices used, and thus can be overly pessimistic
in specific cases.
In order to derive more accurate bounds, we need to consider specific covariance matrices. In this article,
we have focussed on covariance matrices that can be modelled as a matrix-vector product using a diffusion operator.
Diffusion operators are commonly used for modelling spatial covariances in ocean data assimilation
\citep{Weaver_2001} and are closely related to other techniques for modelling spatial covariances
in atmospheric data assimilation \citep{Purser_2003}, geostatistics \citep{Lindgren_2011},
inverse problems \citep{buithanh13} and uncertainty quantification \citep{gmeiner17}.
They are well suited for problems that have large state and observation vectors,
they provide convenient access to an inverse covariance operator (specifically $\mathbf{R}^{-1}$
as required by B-PCG), and they have useful flexibility for specifying covariance functions with different
characteristics.

In order to simplify the theoretical analysis, we assumed that the basic parameters of the diffusion-based
covariance models for both $\mathbf{B}$ and $\mathbf{R}$ were constant.
These parameters consist of the standard deviations
($\sigma_b$ and $\sigma_{\rm o}$), as well as parameters that control the degree of smoothness
(integers $M_{\rm b}$ and $M_{\rm o}$) and spatial range (length-scales $L_{\rm b}$ and $L_{\rm o}$) of the underlying
correlation functions.
 (The quantities with subscripts `${\rm b}$' and `${\rm o}$' refer to the background and observation
 quantities, respectively.)
With these assumptions, the covariance functions implied by the diffusion model are of Mat\'ern type.
In $\mathbb{R}$, they are Auto-Regressive (AR) functions of order $M$, which are the functions relevant
for our one-dimensional (1D) analysis.
Furthermore, we assumed that the grids supporting the background state and observations have uniform
resolution $h_{\rm b}$ and \mbox{$h_{\rm o} = \zeta h_{\rm b}$} where $\zeta$ is a positive integer,
which implies that there are fewer observations than background grid points.
Under these assumptions, we derived an analytical expression
for the eigenvalues of the $\mathbf{B}$-preconditioned Hessian matrix ($\mathbf{S}$). By further assuming
that $\mathbf{H}\mathbf{B}\mathbf{H}^{\transpose}$ can be approximated by a diffusion operator $\Bhat$ that
is discretised directly on the observation grid, it has been possible to derive criteria
that the parameter pairs \mbox{($M_{\rm b}, L_{\rm b}$)} and \mbox{($M_{\rm o}, L_{\rm o}$)}
must jointly satisfy to obtain a minimum upper bound for the condition number of $\mathbf{S}$.
These constraints are exact when the observation and background grids coincide
(\mbox{$\zeta =1$}), but are affected by a minor discretisation error when
the observation grid is coarser than the background grid (\mbox{$\zeta < 1$}).
We used these analytical results to interpret numerical results from experiments with a
1D variational data assimilation system (1D-Var).

First, our analytical expressions expose the already well-known result that increasing (decreasing)
the ratio $\sigma_{\rm o}/\sigma_{\rm b}$ leads to an increase (decrease) of the condition number
of $\mathbf{S}$. Furthermore, when \mbox{$M_{\rm o}=M_{\rm b}$}, our results show that the condition number is
minimised when the same correlation model is used for both background and observation errors
(\textit{i.e.}, \mbox{$L_{\rm b}= L_{\rm o}$}). This is consistent with the results
of \citet{Tabeart_2021} who considered only the special case when \mbox{$M_{\rm o}=M_{\rm b} = 2$}
({\it i.e.}, when the correlation functions are Second-Order AR functions).
However, the main contribution of our work has been to extend the analysis to the more realistic case
where the background and observation errors are modelled with different smoothness parameters
(\mbox{$M_{\rm o} \neq M_{\rm b}$}).

While we have derived the analytical results in terms of the parameter pairs
\mbox{($M_{\rm b}, L_{\rm b}$)} and \mbox{($M_{\rm o}, L_{\rm o}$)}, we have mainly interpreted them
and the results of the numerical experiments in terms of the parameter pairs
\mbox{($M_{\rm b}, D_{\rm b}$)} and \mbox{($M_{\rm o}, D_{\rm o}$)} where
$D_{\rm b}$ and $D_{\rm o}$ are alternative (`Daley') length-scale parameters commonly used
for differentiable correlation functions in data assimilation \citep{Daley_1991}.
Specifically, for the 1D problem under consideration,
\mbox{$D_{\rm b} = L_{\rm b}\sqrt{2M_{\rm b} - 3}$} and \mbox{$D_{\rm o} = L_{\rm o}\sqrt{2M_{\rm o} - 3}$}
where \mbox{$M_{\rm b} > 1$} and \mbox{$M_{\rm o} > 1$}. In terms of fixed values of
$D_{\rm b}$ and $D_{\rm o}$, the AR functions have the convenient property that they converge to
Gaussian functions for large $M_{\rm b}$ and $M_{\rm o}$. Our results have also exposed a direct
relationship with closely-related (`Stein') length-scale parameters,
\mbox{$\rho_{\rm b} = L_{\rm b}\sqrt{2M_{\rm b} - 1}$} and
\mbox{$\rho_{\rm o} = L_{\rm o}\sqrt{2M_{\rm o} - 1}$}, used in geostatistics \citep{Stein_1999}.
In terms of fixed values of $\rho_{\rm b}$ and $\rho_{\rm o}$, the AR functions also converge
to Gaussian functions for large $M_{\rm b}$ and $M_{\rm o}$, and are defined for both
the differentiable (\mbox{$M_{\rm b}>1$} and \mbox{$M_{\rm o}>1$}) and
non-differentiable AR functions (\mbox{$M_{\rm b}=M_{\rm o}=1$}).

The condition number is markedly more sensitive to the parameter specifications
for the case \mbox{$M_{\rm o} > M_{\rm b}$} than \mbox{$M_{\rm o} < M_{\rm b}$}.
This has been illustrated in the numerical experiments and is evident from the analytical expression
(Equation~\eqref{eq_minima_Mb_smaller_D}) that describes the relationship between
\mbox{($M_{\rm b}, D_{\rm b}$)} and \mbox{($M_{\rm o}, D_{\rm o}$)} required to achieve
the minimum upper bound of the condition number when \mbox{$M_{\rm o }>M_{\rm b}$}. In general,
$D_{\rm o}$ needs to be much smaller than $D_{\rm b}$ for this optimality condition to be met because
of the presence of $M_{\rm b}$ and $M_{\rm o}$ as exponents in the expressions.
For this case, the small eigenvalues of $\mathbf{S}$ are amplified by $\mathbf{R}^{-1}$
more than they are damped by $\mathbf{B}$, which can result in a drastic increase of the condition
number and thus a significant risk that the convergence rate of B-PCG will be substantially degraded.
For the case \mbox{$M_{\rm o}<M_{\rm b}$}, the minimum upper bound of the condition number
is attained when the `Stein' length-scales $\rho_{\rm b}$ and $\rho_{\rm o}$ are equal
(Equation~\eqref{eq_minima_Mo_smaller}). In contrast with the case \mbox{$M_{\rm o}>M_{\rm b}$},
this means that the minimum upper bound is obtained when the background- and observation-error
correlation functions have similar spatial range. When \mbox{$M_{\rm o}<M_{\rm b}$}
and \mbox{$D_{\rm o}\leq D_{\rm b}$},
accounting for observation-error correlations systematically improves the conditioning of $\mathbf{S}$
compared to the case when a diagonal $\mathbf{R}$ is used.

While $M_{\rm o}$ and $M_{\rm b}$ are intended as free parameters of the correlation model,
to be adjusted to achieve the best possible fit to statistical estimates of correlated error,
they also provide valuable leverage for controlling the conditioning
of $\mathbf{S}$ when a non-diagonal $\mathbf{R}$ is used. Whereas using a non-diagonal $\mathbf{R}$
is likely to degrade significantly the convergence rate when \mbox{$M_{\rm o}>M_{\rm b}$},
it can accelerate the convergence rate compared to the case where a diagonal $\mathbf{R}$ is used
when \mbox{$M_{\rm o}<M_{\rm b}$}. In practice, this situation would correspond to choosing a
Gaussian-like correlation function for background error ({\it e.g.}, \mbox{$M_{\rm b}\approx 10$})
and a correlation function with fatter tails (more power at smaller scales)
for observation error ({\it e.g.}, \mbox{$M_{\rm o} = 2$}).
Interestingly, there is evidence in the atmospheric data assimilation literature that suggests that
error correlations for certain observation types do exhibit fat tails. In this case, the interest
in using a non-diagonal $\mathbf{R}$ is twofold: it can accelerate the B-PCG convergence rate
as well as providing a more accurate correlation model.
For the case where statistical estimates of the correlation parameters result in unfavourable
values in terms of conditioning (\mbox{$M_{\rm o}>M_{\rm b}$} and/or \mbox{$D_{\rm o}\gg D_{\rm b}$}),
$\mathbf{S}$ can be `reconditioned' by adjusting the values of $M_{\rm o}$ and $D_{\rm o}$ to enable
faster convergence. We have shown in our 1D-Var experiments that approximate correlation models
can be used to reduce the condition number without causing a significant loss
of solution accuracy at full convergence. This corroborates the conclusion
of several previous studies \citep[\textit{e.g.},][]{Stewart_2013}
that even a very crude approximation of the observation-error correlations
can give a better solution than one obtained by ignoring them altogether.

The analysis in this article has exposed important sensitivities of the convergence
rate of B-PCG to fundamental parameters of diffusion model representations of $\mathbf{B}$ and $\mathbf{R}$,
and has lead to conditions for adjusting the parameters to improve the conditioning of $\mathbf{S}$.
We can expect similar results in higher dimensions where diffusion kernels have similar (Mat\'ern-like) functional
forms as those in our 1D study. However, more work is required to extend these results to account for more
sophisticated diffusion models, such as ones that include diffusion tensors that are anisotropic and spatially varying,
or multiple-scale and hybrid formulations that are built from linear combinations of diffusion operators.
In this study, we considered a simple observation network. The convergence properties need to be revisited
in an operational-like framework using a full network of diverse observations for which only a subset
may be affected by spatially correlated errors in $\mathbf{R}$, and where observation operators will be much more complex.
The results from this study are a first step towards understanding and controlling the convergence properties in this
more challenging framework.

\bibliography{references}

\begin{appendix}

\section{Mat\'ern functions and diffusion operators on $\mathbb{R}$ and $\mathbb{S}$}
\label{app_matern}

\subsection{Diffusion on $\mathbb{R}$}

\label{app_matern_R}

Mat\'ern random fields on $\mathbb{R}^d$ can be derived by solving a general
stochastic fractional partial differential equation (PDE)
\citep{Whittle_1963,Guttorp_2006}.
Here, we are interested in the correlation functions of a subset of Mat\'ern fields
on $\mathbb{R}$ (\mbox{$d=1$}) where parameters are chosen such that the generating PDE
has a simplified form for numerical computations.

Let \mbox{$\chi : z \mapsto \chi(z)$} and \mbox{$\eta : z \mapsto \eta(z)$}
be square-integrable functions (\mbox{$\chi, \eta \in L^2(\mathbb{R})$})
of the spatial coordinate \mbox{$z\in \mathbb{R}$}.
We consider solutions of the following elliptic equation on $\mathbb{R}$:
\begin{align}
  \frac{1}{\gamma^2}\left( I -  L^2 \frac{\partial^2}{\partial z^2}\right)^{\! M} \!
   \eta (z) = \chi (z)
\label{eq:impR}
\end{align}
where $I$ is the identity operator, $M$ is a positive integer,
$L$ is a length-scale parameter, and $\gamma^2$ is a normalisation constant.
Equation~\eqref{eq:impR} can be interpreted
as the {\it inverse} of a diffusion operator, \mbox{$\eta \mapsto \mathcal{L}^{-1} \eta$},
which is formed by discretising the time derivative of the diffusion equation with an
Euler backward (implicit) scheme and by applying the resulting operator over
$M$ time steps \citep{Mirouze_2010}. With this interpretation,
\mbox{$L^2 = \mu \Delta t$} where $\mu$ is the diffusion coefficient and
$\Delta t$ is the time step. The integral solution of Equation~\eqref{eq:impR}
is thus a diffusion operator, \mbox{$\chi \mapsto \mathcal{L} \chi$}.
The solution, which is straightforward to derive using the Fourier transform,
is a convolution operator,
\mbox{$\mathcal{L} \chi \equiv c \ast \chi$}, where \mbox{$c = c(r)$} is an $M$th-order AR function
(a polynomial times the exponential function) given by
\begin{equation}
c(r) = \sum_{j=0}^{M-1}\beta_{j}
\left(\frac{r}{L}\right)^{j}
e^{-r/L},
\label{eq:c}
\end{equation}
$r=|z-z^{\prime}|$ is the Euclidean distance between points $z$ and $z^{\prime}$,
and \begin{equation}
\beta_{j}=\frac{2^j(M-1)! \, (2M-j-2)!}{j! \, (M-j-1)!\, (2M-2)!}.
\label{eq:beta}
\nonumber
\end{equation}
Setting the normalisation constant to
\begin{equation}
\gamma^2 = \nu \, L
\label{eq:gamma}
\end{equation}
where
\begin{equation}
\nu = \frac{2^{2M-1}[(M-1)!]^2}{(2M-2)!}
\label{eq:nu}
\end{equation}
ensures that $c(0)=1$ \citep{Mirouze_2010}. The power spectrum of $c$, which
is given by the Fourier transform $\hat{c}$ of $c$, describes
the smoothness properties of $c$ as a function of spectral scale:
\begin{equation}
  \hat{c}(\hat{z}) = \frac{\gamma^2}{\left( 1 + L^2 \hat{z}^2\right)^M}
  \label{eq:chat}
\end{equation}
where $\hat{z}$ is the spectral wavenumber.

We focus on the differentiable AR functions that correspond to $M>1$. For these functions,
we use a standard parameter \citep{Daley_1991}
\begin{equation*}
D =\sqrt{-\frac{1}{\partial^2 c /\partial z^2 |_{z=z^{\prime}}}}
\label{eq:daley1}
\end{equation*}
to characterize the length-scale of the correlation function. The parameter
$D$, which we call the Daley length-scale, corresponds to the distance
between \mbox{$z=z^{\prime}$} and the mid-amplitude point of a parabola that osculates
the AR function at \mbox{$z=z^{\prime}$}. Using Equation~\eqref{eq:c}, it is straightforward
to show that \mbox{$D = L \sqrt{2 M - 3}$} (Equation~\eqref{eq_D}),
which is a function of both $L$ and $M$. An important property of AR functions
is that, for fixed $D$, they converge to the Gaussian function $c_{\rm g}(r)$
as $M\rightarrow \infty$:
\begin{equation}
c_{\rm g}(r) =  \exp{\big( {-r^{2}/2D^{2}} \big)}.
\label{eq:cg}
\end{equation}

Figure~\ref{fig_matern_plot} shows the effect on $c$ and $\hat{c}$ of varying $D$
for a fixed value of $M$, and vice versa.
Increasing $D$ with $M$ held fixed increases the spatial reach of the correlation functions but
does not affect their spectral decay rate at small wavelengths.
On the other hand, increasing $M$ with $D$ held fixed results
in correlation functions with thinner tails and sharper spectral decay
rates at small wavelengths.
\begin{figure}[h]
 \includegraphics[width = \textwidth]{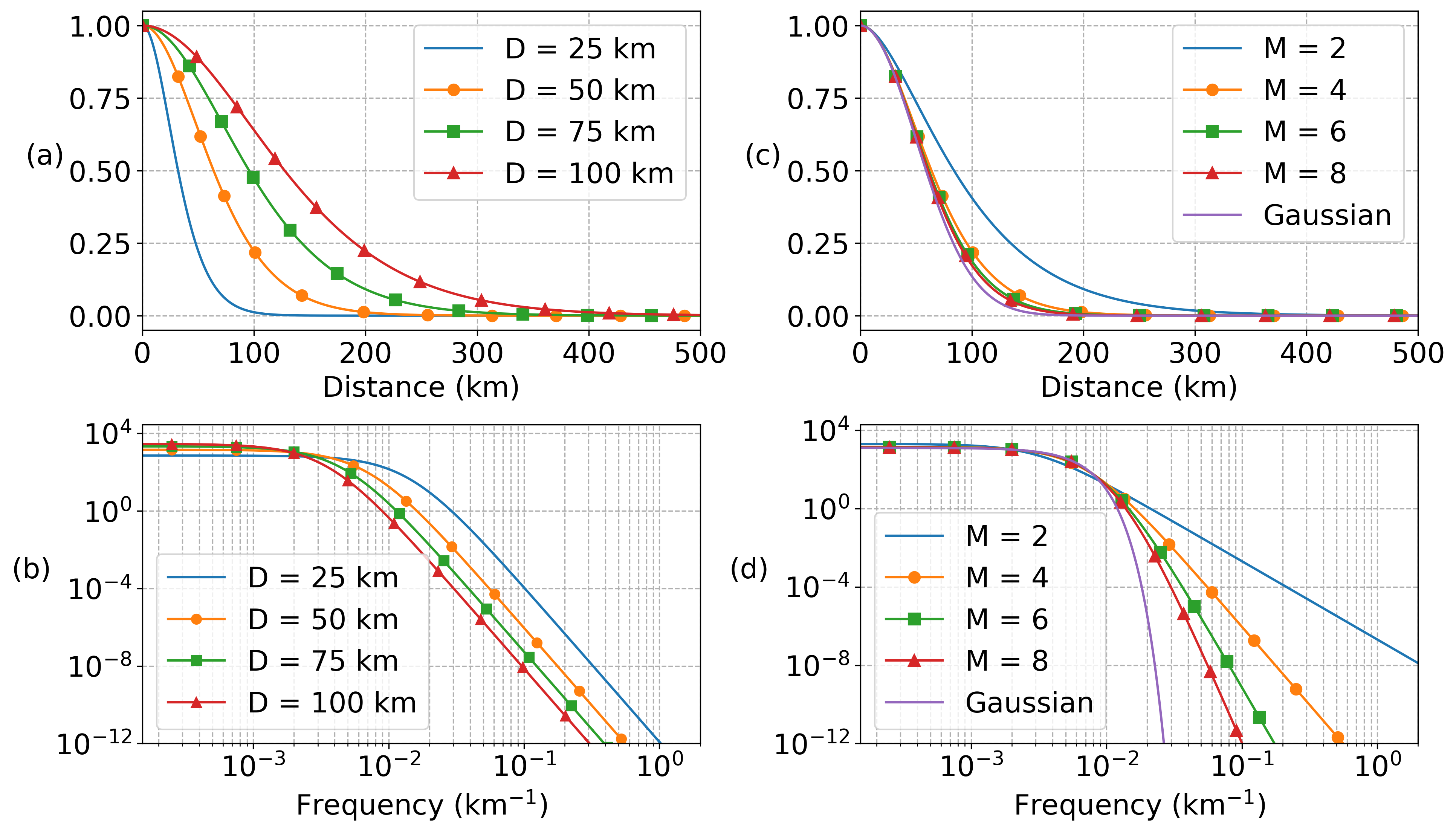}
  \centering
\caption{(a) Equation~\eqref{eq:c} plotted as a function of distance $r$, and
(b) Equation~\eqref{eq:chat} plotted as a function of wavelength $2\pi / \hat{z}$.
The curves are displayed for different values of $D$ and fixed value of $M=4$.
Panels~(c) and (d) show corresponding plots where $M$ is varied for fixed value of \mbox{$D=50$}~km.}
  \label{fig_matern_plot}
\end{figure}

\subsection{Diffusion on $\mathbb{S}$}

\label{app_matern_S}

\cite{Tabeart_2018} and \cite{Tabeart_2021} use a SOAR function,
which is equal to Equation~\eqref{eq:c} with \mbox{$M=2$} and hence \mbox{$D=L$}
from Equation~\eqref{eq_D}.
Furthermore, they restrict the SOAR function to the circular domain ($\mathbb{S}$)
of radius $a$ by using chordal distance \mbox{$r = 2a\sin (\theta /2)$}
where $\theta$ is the angle between points $z$ and $z^{\prime}$ on the circle.
This ensures that $c(r)$ is positive definite on $\mathbb{S}$
\citep{Gaspari_1999}. Taking $a$ as the radius of the Earth,
the domain $\mathbb{S}$ can be interpreted as a latitude circle at
the Equator.

In this article, we have also considered a circular domain of radius $a$.
For length-scales \mbox{$L \ll a$},
the correlation functions associated with the diffusion operator applied
on $\mathbb{S}$ are approximately Mat\'ern since the influence of curvature
is minor. It is instructive nevertheless to present the exact correlation functions
on $\mathbb{S}$, which can be derived by considering the solution of
the elliptic equation
\begin{align}
  \frac{1}{\gamma^2}
  \left( I -  \frac{L^2}{a^2} \frac{\partial^2}{\partial \phi^2}\right)^{\! M} \!
   \eta (\phi) = \mu (\phi),
\label{eq:impS}
\end{align}
subject to periodic boundary conditions on the solution and its derivative:
\begin{align}
\eta(-\pi) &=\eta(\pi),
  \nonumber
\\
\frac{\displaystyle \partial \eta}{\displaystyle \partial \phi}\bigg|_{\phi =-\pi} &=
\frac{\displaystyle \partial \eta}{\displaystyle \partial \phi}\bigg|_{\phi = \pi}.
  \label{eq:bcs}
  \nonumber
\end{align}
Solving Equation~\eqref{eq:impS} is equivalent to solving Equation~\eqref{eq:impR}
on the periodic domain \mbox{$-Z \leq z \leq Z$} with \mbox{$z =a\phi$} and
\mbox{$Z=a\pi$}.

The solutions that satisfy the boundary conditions are of the general form
\begin{equation}
\eta (\phi) = \sum_{m=0}^{\infty} A_m \cos( m \phi ) + B_m \sin( m \phi ).
\label{eq:gensoln}
\end{equation}
The coefficients $A_m$ and $B_m$ can be determined using the orthogonality relations of
the sine and cosine functions:
\begin{align}
\int_{-\pi}^{\pi} \sin (m\phi) \sin (n\phi) \textrm{d} \phi & \, = 
\int_{-\pi}^{\pi} \cos (m\phi) \cos (n\phi) \textrm{d} \phi  \, = \, \pi \, \delta_{mn}, \nonumber \\
\int_{-\pi}^{\pi} \sin (m\phi) \cos (n\phi) \textrm{d} \phi & \, = \, 0 \; \; \forall \, m, n \nonumber
\end{align}
where $\delta_{mn}$ is the Kronecker delta.
To determine $A_m$, we substitute Equation~\eqref{eq:gensoln} in Equation~\eqref{eq:impS}, multiply the
resulting equation by \mbox{$\cos (n\phi)$}, integrate from \mbox{$-\pi$} to $\pi$, and use the
orthogonality relations above. This yields
\begin{equation}
  A_m = \frac{\gamma^2}{\pi} \left( 1 + \frac{L^2}{a^2} \, m^2 \right)^{\! -M}
  \int_{-\pi}^{\pi} \mu(\phi) \cos (m\phi) \, \textrm{d} \phi.
\label{eq:Am}
\end{equation}
To determine $B_m$, we follow the same procedure but multiply by $\sin (n\phi)$.
This yields
\begin{equation}
  B_m = \frac{\gamma^2}{\pi} \left( 1 + \frac{L^2}{a^2} \, m^2 \right)^{\! -M}
  \int_{-\pi}^{\pi} \mu(\phi) \sin (m\phi) \, \textrm{d} \phi.
\label{eq:Bm}
\end{equation}
Substituting Equations~\eqref{eq:Am} and \eqref{eq:Bm} into Equation~\eqref{eq:gensoln}, and using the
trigonometric identity
\begin{equation}
  \cos ( m(\phi - \phi^{\prime}) ) = \cos ( m\phi )\, \cos ( m\phi^{\prime} )
  + \sin ( m\phi )\, \sin ( m\phi^{\prime} )
\label{eq:cos_theta}
\nonumber
\end{equation}
yields the solution
\begin{equation}
\eta(\phi) = \int_{-\pi}^{\pi} c(\theta)\, \mu(\phi^{\prime})\, \textrm{d}\phi^{\prime}
\label{eq:etam}
\nonumber
\end{equation}
where \mbox{$\theta = \phi - \phi^{\prime}$},
\begin{equation}
  c(\theta) = \gamma^2 \sum_{m=0}^{\infty} c_m \cos (m \theta)
\label{eq:g}
\end{equation}
and
\begin{equation}
  c_m = \frac{1}{\pi} \left( 1 + \frac{L^2}{a^2} \, m^2 \right)^{\! -M}.
\label{eq:gm}
\end{equation}
The normalisation factor and Daley length-scale are, respectively,
\begin{equation}
  \gamma^2 = \frac{1}{\sum_{m=0}^{\infty} c_m}
  \nonumber
\end{equation}
and
\begin{align}
D & =
a \sqrt{-\frac{1}{\partial^2 c /\partial \phi^2 |_{\theta =0}}}
= a \, \sqrt{\frac{1}{\sum_{m=0}^{\infty} m^2 c_m}}.
\label{eq:daley}
  \nonumber
\end{align}

All valid continuous isotropic
correlation functions on $\mathbb{S}$ can be represented by a Fourier cosine series expansion
with non-negative Fourier coefficients
(see Theorem~2.11 in \cite{Gaspari_1999}), which is clearly satisfied by
Equations~\eqref{eq:g} and \eqref{eq:gm}. The smoothness properties of the correlation function
are determined by the Fourier coefficients $c_m$ in Equation~\eqref{eq:gm}. They can be seen
to have a similar dependence on $L$ and $M$ as $\hat{c}$ in
Equation~\eqref{eq:chat} where we can associate $\hat{z}$ on $\mathbb{R}$ with
\mbox{$m^2 / a^2$} on $\mathbb{S}$.

\section{Proof of Theorem~5}
\label{app_proof_theorem5}

The eigenvalues of $\Shat$ are bounded below by 1 (see Theorem~\ref{th_eig_So}), which implies that
\begin{equation}
\kappa(\Shat)  \leq \max_{i\in\llbracket 0, m-1 \rrbracket} \lambda_i(\Shat) = \lambda_{\max} (\Shat).
\nonumber
\end{equation}
Since \mbox{$0 \le \sin^2(y) \le 1$} for any \mbox{$y \in [0, \pi]$}, $\lambda_{\max} (\Shat)$ is bounded by
\begin{equation}
\lambda_{\max} (\Shat)   \leq \max_{x \in [0,1]} \phi(x)
\nonumber
\end{equation}
where $\phi$ is a continuously differentiable function given by
\begin{equation}
\phi(x) = 1+ \alpha \frac{\left[1+4\widetilde{L}_{\rm o}^2x\right]^{M_{\rm o}}}
{\left[1+4 \Lbhat^2x\right]^{M_{\rm b}}}.
\label{eq_def_f}
\end{equation}
We seek a solution to the following bound-constraint problem:
\begin{equation}
\max_{x \in [0,1]} \phi(x).
\label{eq_bcp}
\end{equation}
Let \mbox{$x^{\ast} \in [0,1]$} be a stationary point for problem~\eqref{eq_bcp} and
let us first assume that such point is inside the domain; \emph{i.e.}, \mbox{$\phi'(x^\ast) = 0$}.
The derivative of the function $\phi$ can be expressed as
\begin{equation}
\phi'(x) \, = \, 4 \alpha \, v(x )\, w(x)
\nonumber
\end{equation}
where
$$
v(x) = \left( \frac{\widetilde{L}_{\rm o}^2 M_{\rm o}}
{1+4\widetilde{L}_{\rm o}^2 x} -  \frac{ \Lbhat^2M_{\rm b}}{1+4 \Lbhat^2 x} \right)
\quad  \text{and} \quad
w(x) =\frac{\left[1+4\widetilde{L}_{\rm o}^2x\right]^{M_{\rm o}}}{\left[1+4 \Lbhat^2x\right]^{M_{\rm b}}}.
$$
Since $w(x)$ is strictly positive, and \mbox{$\alpha > 0$}, a stationary point inside the domain satisfies \mbox{$v(x^\ast) = 0$}.
This yields
 \begin{equation}
x^\ast =\frac{  \Lbhat^2M_{\rm b} - \widetilde{L}_{\rm o}^2M_{\rm o}}{4\widetilde{L}_{\rm o}^2 \Lbhat^2(M_{\rm o} - M_{\rm b})}.
\label{eq_def_xi}
\end{equation}
The second derivative of $f$ is
\begin{equation}
\phi''(x) \, = \, 16 \alpha \, w(x) \, \left( \frac{\Lbhat^4  M_{\rm b}}{ (1+ 4\Lbhat^2 x)^2}
- \frac{\widetilde{L}_{\rm o}^4 M_{\rm o}}{ (1+ 4\widetilde{L}_{\rm o}^2 x)^2}  \right) \, + \, 16 \alpha\, v(x)^2 \, w(x).
\label{eq_second_derivative}
\end{equation}
Substituting~\eqref{eq_def_xi} into Equation~\eqref{eq_second_derivative} gives
$$
\phi''(x^\ast) \, = \, 16 \alpha ( M_{\rm o} - M_{\rm b} )\, w(x^\ast)
\left( \frac{\Lbhat^4\widetilde{L}_{\rm o}^4\big( M_{\rm o} - M_{\rm b} \big)^2}
{ \big( \Lbhat^2 - \widetilde{L}_{\rm o}^2 \big)^2  M_{\rm o} M_{\rm b}} \right).
$$
Therefore, the stationary point $x^\ast$ can be a maximum point if and only if \mbox{$\phi''(x^\ast) < 0$};
\emph{i.e.}, if \mbox{$M_{\rm o} < M_{\rm b}$}.  In addition, for $x^\ast$ to be a feasible point then
\mbox{$0 < x^\ast < 1$} and from Equation~\eqref{eq_def_xi} the following conditions must be satisfied:
$$
\Lbhat^2 M_{\rm b} - \widetilde{L}_{\rm o}^2 M_{\rm o} < 0
$$
and
$$
\Lbhat^2 M_{\rm b} - \widetilde{L}_{\rm o}^2 M_{\rm o}  \, > \, 4 \widetilde{L}_{\rm o}^2
\Lbhat^2 \big( M_{\rm o}  - M_{\rm b} \big).
$$
For the other cases, $x^\ast$ is equal to either the lower bound (\mbox{$x^\ast = 0$}) or the upper bound (\mbox{$x^\ast = 1$}),
with function values of
\begin{align*}
\phi(0) & = 1 + \alpha, \\
\phi(1) & = 1 + \alpha   \frac{\left[1+4\widetilde{L}_{\rm o}^2\right]^{M_{\rm o}}}
{\left[1+4 \Lbhat^2\right]^{M_{\rm b}}}.
\end{align*}

Finally, substituting~\eqref{eq_def_xi} into \eqref{eq_def_f}, we obtain that
\begin{equation}
\phi(x^\ast) = 1+ \alpha \left(\frac{\widetilde{L}_{\rm o}^2}{M_{\rm b}}\right)^{M_{\rm b}}
\left(\frac{M_{\rm o}}{ \Lbhat^2}\right)^{M_{\rm o}}
\left(\frac{M_{\rm b}-M_{\rm o}}{\widetilde{L}_{\rm o}^2 -  \Lbhat^2}\right)^{M_{\rm b}-M_{\rm o}}.
\end{equation}
\qed

\section{Proof of Corollary~1}
\label{app_proof_corollary1}
We consider $\eta$ as a function of $\widetilde{L}_{\rm o}$, denoted $f(\widetilde{L}_{\rm o})$.
Hereafter, the conditions~(i), (ii) and (iii) will refer to the conditions stated in Theorem~\ref{th_diffusion_bound}.
We consider the case where condition~(ii) does not hold, \textit{i.e.}, \mbox{$M_{\rm o}\geq M_{\rm b}$}.
In this case, Theorem \ref{th_diffusion_bound} states that
\begin{equation}
f(\widetilde{L}_{\rm o})= 1+\frac{\sigma_{\rm b}^2\nu_{\rm b} \Lbhat}{\sigma_{\rm o}^2\nu_{\rm o} \widetilde{L}_{\rm o}}
\max\left\{\frac{\big( 1+4\widetilde{L}_{\rm o}^2 \big)^{M_{\rm o}}}{\big( 1+4 \Lbhat^2 \big)^{M_{\rm b}}} ;\; 1\right\}.
\label{eq:fvalue}
\end{equation}
Let us first assume that
\mbox{$\big( 1+4\widetilde{L}_{\rm o}^2 \big)^{M_{\rm o}}<\big( 1+4 \Lbhat^2 \big)^{M_{\rm b}}$}.
Then, Equation~\eqref{eq:fvalue} simplifies to
\begin{equation*}
f(\widetilde{L}_{\rm o}) = 1+\frac{\sigma_{\rm b}^2\nu_{\rm b} \Lbhat}{\sigma_{\rm o}^2\nu_{\rm o} \widetilde{L}_{\rm o}},
 \end{equation*}
which is a decreasing function of $\widetilde{L}_{\rm o}$. Let us now assume that
\mbox{$\left(1+4\widetilde{L}_{\rm o}^2\right)^{M_{\rm o}}>\left(1+4 \Lbhat^2\right)^{M_{\rm b}}$}.
In this case, Equation~\eqref{eq:fvalue} becomes
\begin{equation*}
f(\widetilde{L}_{\rm o})  = 1+\frac{\sigma_{\rm b}^2\nu_{\rm b} \Lbhat}{\sigma_{\rm o}^2\nu_{\rm o}
\widetilde{L}_{\rm o}}\frac{\big( 1+4\widetilde{L}_{\rm o}^2 \big)^{M_{\rm o}}}{\big( 1+4 \Lbhat^2 \big)^{M_{\rm b}}},
 \end{equation*}
whose derivative is given by
 \begin{equation*}
f'(\widetilde{L}_{\rm o}) =  \underbrace{\frac{\sigma_{\rm b}^2\nu_{\rm b} \Lbhat}
 {\sigma_{\rm o}^2\nu_{\rm o} \widetilde{L}_{\rm o}^2 }
 \frac{\big( 1+4\widetilde{L}_{\rm o}^2\big)^{M_{\rm o}- 1}}{\big( 1+4 \Lbhat^2 \big)^{M_{\rm b}}}}_{>0}
 \left( 4\widetilde{L}_{\rm o}^2\big( 2M_{\rm o}-1 \big) -1\right).
 \end{equation*}
Under the assumption that
\mbox{$\widetilde{L}_{\rm o} > 1/\big(2\sqrt{2M_{\rm o}-1}\big)$},
$f(\widetilde{L}_{\rm o})$ is an increasing function of $\widetilde{L}_{\rm o}$.
 As we already showed that $f(\widetilde{L}_{\rm o})$ is decreasing  when
\mbox{$\big( 1+4\widetilde{L}_{\rm o}^2\big)^{M_{\rm o}}<\big( 1+4 \Lbhat^2 \big)^{M_{\rm b}}$},
 the function $f$ then reaches its unique minimum when
 \begin{equation}
\left(1+4\widetilde{L}_{\rm o}^2\right)^{M_{\rm o}}=\left(1+4 \Lbhat^2\right)^{M_{\rm b}}.
 \end{equation}
 \qed

\section{Proof of Corollary~2}
\label{app_proof_corollary2}
We consider $\eta$ as a function of $\widetilde{L}_{\rm o}$, denoted by $f(\widetilde{L}_{\rm o})$.
Hereafter, the conditions~(i), (ii) and (iii) will refer to the conditions stated in Theorem~\ref{th_diffusion_bound}.
Let us first assume that condition~(ii) holds, \textit{i.e.}, \mbox{$M_{\rm o}<M_{\rm b}$} whereas condition~(i) does not hold,
\textit{i.e.},
\begin{equation}
\widetilde{L}_{\rm o}^2 M_{\rm o} \, \leq  \, \Lbhat M_{\rm b}.
\label{eq_ii_not_hold}
\end{equation}
From~\eqref{eq_ii_not_hold}, we first consider the variations of $f(\widetilde{L}_{\rm o})$
when $\widetilde{L}_{\rm o}$ is in the interval \mbox{$\left [0, \Lbhat\sqrt{M_{\rm b}/M_{\rm o}}\right]$}.
In this case, Theorem~\ref{th_diffusion_bound} states that
\begin{equation}
f(\widetilde{L}_{\rm o})\, = \, 1+\frac{\sigma_{\rm b}^2\nu_{\rm b} \Lbhat}{\sigma_{\rm o}^2\nu_{\rm o}
\widetilde{L}_{\rm o}}\max\left\{\frac{(1+4\widetilde{L}_{\rm o}^2)^{M_{\rm o}}}{(1+4 \Lbhat^2)^{M_{\rm b}}} ;\; 1\right\}.
\label{eq_eta_itrue_iifalse}
\end{equation}
From condition \eqref{eq_ii_not_hold}, we have
\begin{align}
\widetilde{L}_{\rm o}^2   & \, \leq \, \Lbhat^2 \frac{M_{\rm b}}{M_{\rm o}}, \nonumber
\end{align}
which implies that
\begin{align}
\left( 1+4\widetilde{L}_{\rm o}^2\right)^{\frac{M_{\rm o}}{M_{\rm b}}}  &
\leq  \left(1+4\Lbhat^2 \frac{M_{\rm b}}{M_{\rm o}}\right)^{\frac{M_{\rm o}}{M_{\rm b}}} \nonumber
\end{align}
and hence
\begin{align}
\left( 1+4\widetilde{L}_{\rm o}^2\right)^{\frac{M_{\rm o}}{M_{\rm b}}}
-  \left(1+4\Lbhat^2\right) & \, \leq  \, \left(1+4\Lbhat^2 \frac{M_{\rm b}}{M_{\rm o}}\right)^{\frac{M_{\rm o}}{M_{\rm b}}}
- \left(1+4\Lbhat^2\right) \, = \, g(\Lbhat^2) \label{eq_ineq_f}
\end{align}
where the function $g$ is given by \mbox{$g(x)=\left(1+4x M_{\rm b}/M_{\rm o}\right)^{M_{\rm o}/M_{\rm b}} - \left(1+4x\right)$}.
Taking the first derivative of $g$ gives
\begin{equation*}
g'( x)\, = \, 4
\left( \left(1+4 x\frac{M_{\rm b}}{M_{\rm o}}\right)^{\frac{M_{\rm o}}{M_{\rm b}} -1}  -1 \right).
\end{equation*}
Since \mbox{$M_{\rm o}/M_{\rm b} < 1$}, for all $x\ge0$,
we have \mbox{$\big( 1+4 x M_{\rm b}/M_{\rm o} \big)^{M_{\rm o}/M_{\rm b} -1}<1$} and hence \mbox{$g'(x) <0$}; {\it i.e.},
$g$ is a decreasing function on $[0,+\infty)$.
Consequently, \mbox{$g(\Lbhat^2) \le g(0)=0$} and thus inequality \eqref{eq_ineq_f} implies that
\begin{align*}
\left( 1+4\widetilde{L}_{\rm o}^2\right)^{\frac{M_{\rm o}}{M_{\rm b}}}  -  \left(1+4\Lbhat^2\right) & \leq 0,
\end{align*}
or, equivalently,
\begin{align*}
\left( 1+4\widetilde{L}_{\rm o}^2\right)^{M_{\rm o}}  & \leq  \left(1+4\Lbhat^2\right)^{M_{\rm b}}.
\end{align*}
By using this inequality in Equation~\eqref{eq_eta_itrue_iifalse}, we obtain that
\begin{equation*}
f(\widetilde{L}_{\rm o}) = 1+\frac{\sigma_{\rm b}^2\nu_{\rm b} \Lbhat}{\sigma_{\rm o}^2\nu_{\rm o}\widetilde{L}_{\rm o}}.
\end{equation*}
In this expression, $f(\widetilde{L}_{\rm o})$ is inversely proportional to $\widetilde{L}_{\rm o}$. Therefore,
$f$ decreases on the interval \mbox{$\left [0, \Lbhat\sqrt{M_{\rm b}/M_{\rm o}} \right]$}.
As a consequence, the minimum of $f(\widetilde{L}_{\rm o})$ must be located on the interval
\mbox{$\left[ \Lbhat\sqrt{M_{\rm b}/M_{\rm o}}, +\infty \right)$}. We now study the variations of
$f$ on this interval, which means that condition~(i) holds:
\begin{equation}
\widetilde{L}_{\rm o}^2 M_{\rm o} >  \Lbhat M_{\rm b}.
\label{eq_ii_hold}
\end{equation}

We assumed that condition~(iii) holds when condition~(i) is satisfied, which means that $f(\widetilde{L}_{\rm o})$ takes the form
\begin{equation*}
f(\widetilde{L}_{\rm o})= 1+\frac{\sigma_{\rm b}^2\nu_{\rm b} \Lbhat}{\sigma_{\rm o}^2\nu_{\rm o} \widetilde{L}_{\rm o}}
\left(\frac{\widetilde{L}_{\rm o}^2}{M_{\rm b}}\right)^{M_{\rm b}} \left(\frac{M_{\rm o}}{ \Lbhat^2}\right)^{M_{\rm o}}
 \left(\frac{M_{\rm b}-M_{\rm o}}{\widetilde{L}_{\rm o}^2 -  \Lbhat^2}\right)^{M_{\rm b}-M_{\rm o}}.
\end{equation*}
The derivative of $f(\widetilde{L}_{\rm o})$ can be expressed as
\begin{equation}
\frac{\partial f}{\partial \widetilde{L}_{\rm o}}( \widetilde{L}_{\rm o}) = \frac{\sigma_{\rm b}^2\nu_{\rm b}
\Lbhat}{\sigma_{\rm o}^2\nu_{\rm o} \widetilde{L}_{\rm o}^2\big( \widetilde{L}_{\rm o}^2  - \Lbhat^2 \big)}
\left(\frac{\widetilde{L}_{\rm o}^2}{M_{\rm b}}\right)^{M_{\rm b}} \left(\frac{M_{\rm o}}{ \Lbhat^2}\right)^{M_{\rm o}}
 \left(\frac{M_{\rm b}-M_{\rm o}}{\widetilde{L}_{\rm o}^2 -  \Lbhat^2}\right)^{M_{\rm b}-M_{\rm o}}
 \left[ \widetilde{L}_{\rm o}^2 \big( 2M_{\rm o} -1 \big) -  \Lbhat^2 \big( 2M_{\rm b} -1 \big)\right].
 \label{eq_derivative_eta}
\end{equation}
If conditions~(i) and (ii) are met, we have
\mbox{$\widetilde{L}_{\rm o}^2-\Lbhat^2>0$}.
Therefore, the stationary point for $f$ satisfies
\begin{align}
\frac{\partial f}{\partial \widetilde{L}_{\rm o}}( \widetilde{L}_{\rm o})   = 0 & \; \Leftrightarrow \;
\widetilde{L}_{\rm o}^2 \big( 2M_{\rm o} -1 \big) = \Lbhat^2 \big( 2M_{\rm b} -1 \big)\\
&\; \Leftrightarrow  \;\widetilde{L}_{\rm o}  = \Lbhat \sqrt{\frac{2M_{\rm b} -1}{2M_{\rm o} -1}}.
\label{eq:stat_point}
\end{align}
Since \mbox{$M_{\rm o}  <  M_{\rm b}$}, it follows that
\begin{align*}
\Lbhat \sqrt{\frac{2M_{\rm b}-1}{2M_{\rm o}-1}}  \, > \,\Lbhat\sqrt{\frac{M_{\rm b}}{M_{\rm o}}}.\\
\end{align*}
We are now interested in examining the behaviour of $f$ on the intervals
 $\left[\Lbhat\sqrt{M_{\rm b}/M_{\rm o}},  \Lbhat \sqrt{(2M_{\rm b}-1)/(2M_{\rm o}-1)}  \right]$ and
 $\left[\Lbhat \sqrt{(2M_{\rm b}-1)/(2M_{\rm o}-1)} , +\infty \right)$.
 For the first interval, we can show that $f$ is decreasing
 since \mbox{$\partial f( \widetilde{L}_{\rm o})/\partial \widetilde{L}_{\rm o} <0$} from Equation~\eqref{eq_derivative_eta} if
\begin{align}
\widetilde{L}_{\rm o}^2 \big( 2M_{\rm o} -1 \big) < \Lbhat^2 \big( 2M_{\rm b} -1 \big)\; \Leftrightarrow \;
\widetilde{L}_{\rm o} < \Lbhat \sqrt{\frac{2M_{\rm b}-1}{2M_{\rm o}-1}} .
\end{align}
Similarly, for the second interval, we can show that $f$ is increasing
 since \mbox{$\partial f( \widetilde{L}_{\rm o})/\partial \widetilde{L}_{\rm o} >0$} from Equation~\eqref{eq_derivative_eta} if
\begin{align*}
\widetilde{L}_{\rm o}^2 \big( 2M_{\rm o} -1 \big) > \Lbhat^2 \big( 2M_{\rm b} -1 \big)\; \Leftrightarrow \;
\widetilde{L}_{\rm o} > \Lbhat \sqrt{\frac{2M_{\rm b}-1}{2M_{\rm o}-1}}.
\end{align*}
Therefore, the stationary point \eqref{eq:stat_point} is the unique minimum of $f$. Finally, multiplying both
sides of Equation~\eqref{eq:stat_point} by $h_{\rm o}\sqrt{2M_{\rm o} -1}$ yields Equation~\eqref{eq_minima_Mo_smaller}.
 \qed

\end{appendix}
\end{document}